\pdfoutput=1
\documentclass[11pt]{article}

\usepackage[a4paper,margin=1in]{geometry}
\usepackage{amsfonts,amssymb}
\usepackage{amsthm}
\usepackage{graphicx}
\usepackage{mathtools}
\usepackage{booktabs}
\usepackage{longtable}
\usepackage{tabularx}
\usepackage{float}
\usepackage{algorithm}
\usepackage{algpseudocode}
\usepackage[utf8]{inputenc}
\usepackage[numbers,sort&compress]{natbib}
\usepackage{enumitem}
\usepackage{lmodern}
\usepackage{microtype}
\usepackage{subcaption}
\usepackage[section]{placeins}
\usepackage{ifpdf}
\usepackage[colorlinks=true,linkcolor=blue,citecolor=blue,urlcolor=blue]{hyperref}
\usepackage[capitalize,noabbrev]{cleveref}
\usepackage{tikz}
\usetikzlibrary{arrows.meta,positioning,shapes.geometric,calc,decorations.pathreplacing}

\ifpdf
  \DeclareGraphicsExtensions{.pdf,.png,.jpg,.eps}
\else
  \DeclareGraphicsExtensions{.eps}
\fi
\graphicspath{{figures/}}

\newcommand{\R}{\mathbb{R}}

\newcommand{\argmin}{\operatorname*{arg\,min}}

\newcommand{\pred}{\mathrm{pred}}
\newcommand{\ared}{\mathrm{ared}}

\newcommand{\qvec}{\operatorname{qvec}}
\newcommand{\clip}{\operatorname{clip}}
\newcommand{\ctrim}{c_{\mathrm{trim}}}
\newcommand{\Bgeo}{B_\phi^{\mathrm{geo}}}
\newcommand{\Dgeo}[1]{\Delta_{\mathrm{geo},#1}}

\newtheorem{theorem}{Theorem}[section]
\newtheorem{lemma}[theorem]{Lemma}
\newtheorem{proposition}[theorem]{Proposition}
\newtheorem{corollary}[theorem]{Corollary}
\newtheorem{assumption}[theorem]{Assumption}
\theoremstyle{definition}
\newtheorem{definition}[theorem]{Definition}
\theoremstyle{remark}
\newtheorem{remark}[theorem]{Remark}
\crefname{assumption}{Assumption}{Assumptions}
\Crefname{assumption}{Assumption}{Assumptions}
\crefname{remark}{Remark}{Remarks}
\Crefname{remark}{Remark}{Remarks}

\allowdisplaybreaks

\begin{document}

\title{BUP-TR: Bayesian Underdetermined Projection Trust-Region Methods for Derivative-Free Optimization}
\author{Wei Hu\thanks{LSEC, ICMSEC, Academy of Mathematics and Systems Science, Chinese Academy of Sciences, Beijing 100190, China; University of Chinese Academy of Sciences, Beijing 100049, China. Email: \texttt{huwei@amss.ac.cn}.}
\and Pengcheng Xie\thanks{Lawrence Berkeley National Laboratory, Berkeley, CA 94720, USA. Email: \texttt{pxie98@gmail.com}.}
\and Ya-Xiang Yuan\thanks{LSEC, ICMSEC, Academy of Mathematics and Systems Science, Chinese Academy of Sciences, Beijing 100190, China. Email: \texttt{yyx@lsec.cc.ac.cn}.}
\and Li Zhang\thanks{LSEC, ICMSEC, Academy of Mathematics and Systems Science, Chinese Academy of Sciences, Beijing 100190, China. Email: \texttt{zhangli2022@lsec.cc.ac.cn}.}}
\date{}

\maketitle

\begin{abstract}
Underdetermined quadratic interpolation is a central model-construction tool in
model-based derivative-free trust-region methods: it limits sampling costs but
leaves an affine family of interpolating quadratics. Classical solvers select
one element of this family by prescribing a fixed norm or model-change measure,
such as the least-Frobenius-change Hessian update in Powell-type methods. We
introduce BUP-TR (Bayesian Underdetermined Projection Trust-Region), which
instead completes the model by projecting a prior quadratic onto the affine
interpolation set in the precision norm supplied by the prior. The same
precision matrix defines a spectral geometry certificate, MAP-poisedness, and a
repair mechanism for interpolation sets. Under standard smoothness assumptions,
uniform precision bounds, MAP-poisedness, and a trust-region-scale
prior-accuracy condition, the hard-MAP models are fully linear. Consequently,
BUP-TR attains global first-order convergence and
\(O(\varepsilon^{-2})\) evaluation complexity, with geometry-repair evaluations
included. A NEWUOA-style implementation, BUP-NEWUOA, improves fixed-budget
performance on the reported benchmark suite at moderate and stringent accuracy
targets while retaining the computational structure of a Powell-type
trust-region method.

\end{abstract}

\section{Introduction}\label{sec:intro}

Derivative-free optimization (DFO) concerns optimization problems in which
reliable derivative information is unavailable. We consider the
unconstrained black-box problem
\begin{equation}\label{eq:prob}
    \min_{x\in\mathbb{R}^n} f(x),
\end{equation}
where the objective can be queried only through function values. Such
problems arise when one evaluation requires a numerical simulation, an
engineering experiment, or a legacy code whose derivatives are unavailable
or too costly to approximate.

Model-based trust-region methods form a standard class of
algorithms for this setting. At iteration $k$, the algorithm builds a local
model from previously queried function values, computes a trial point by
approximately minimizing this model in a trust region, and accepts or
rejects the trial point according to the ratio of actual to predicted
reduction. For interpolation-based DFO, the quality of the local model is
therefore central: the model is the object from which the trial step is
computed.

\subsection{Model completion in underdetermined interpolation}\label{subsec:motivation}

Quadratic interpolation models are attractive because they can represent
curvature without derivative evaluations. A quadratic polynomial in
$\mathbb{R}^n$ has
\begin{equation}\label{eq:qdim}
    q=\frac{(n+1)(n+2)}{2}
\end{equation}
coefficients. A fully determined quadratic interpolation model therefore
requires $q$ independent function values, which is already $231$ when
$n=20$. In expensive black-box optimization, one often works with fewer
interpolation points than coefficients.

Let $c\in\mathbb{R}^q$ denote the coefficient vector of a local quadratic
model and write the interpolation equations as
\begin{equation}\label{eq:intro-Ac-b}
    A_k c=b_k.
\end{equation}
When the system is underdetermined, its feasible solutions form an affine
set. A second criterion is then needed to choose one interpolating
quadratic. We call this step \emph{model completion}: selecting one model
from the affine family satisfying \eqref{eq:intro-Ac-b}.

Many established underdetermined models fit this description.
Conn--Toint-type minimum-norm models, Powell's minimum-change Hessian
models, least $H^2$-norm updating models, and regional minimal updating
models impose interpolation and then minimize a prescribed norm or change
measure over the remaining freedom
\cite{conn2009introduction,powell2006newuoa,wild2008mnh,xie2025underdetermined,xie2025h2,xie2026remu}.
At a schematic level, these rules take the form
\begin{equation}\label{eq:intro-completion-template}
    \min_{c}\; \frac12\|c-c_k^{\rm ref}\|_{B_k}^2
    \quad \text{subject to}\quad A_kc=b_k,
\end{equation}
where the reference vector and the matrix $B_k$ encode the completion
criterion. The penalized blocks and the reference model differ from one
construction to another.

Thus the completion criterion is itself a modeling choice. This observation
motivates the question considered here: can the metric in
\eqref{eq:intro-completion-template} be selected from information already
available at iteration $k$?

\subsection{Prior models for completion}\label{subsec:history}

A trust-region run usually stores more function-value information than the
active interpolation set. Some evaluated points have been removed from the
current set; previous models have produced successful or unsuccessful
steps; local least-squares or surrogate fits may contain useful curvature
information. Classical minimum-norm and minimum-change rules remain simple,
stable, and effective because their completion metrics are fixed in
advance. BUP-TR keeps the constrained-completion formulation and lets a
prior model define the metric used at the current iteration.

In coefficient form, the proposed rule is
\begin{equation}\label{eq:intro-prior-completion}
    \min_c\;\frac12\|c-c_k^\pi\|_{W_k}^2
    \quad\text{subject to}\quad A_kc=b_k,
\end{equation}
where $c_k^\pi$ is a prior coefficient vector and $W_k\succ0$ is a
precision matrix. These objects may be obtained from a local surrogate,
least-squares curvature statistics, or a previously accepted model. With a
Gaussian prior on $c$ and noiseless interpolation constraints,
\eqref{eq:intro-prior-completion} is the maximum a posteriori (MAP)
estimator. The MAP interpretation supplies the reference vector and the
metric; the trial step and the acceptance test remain those of a standard
trust-region method.

This places BUP-TR in the model-construction line of Powell,
Conn--Toint, and subsequent underdetermined-model work. The change is local:
the norm used to complete the interpolation model is supplied by a prior at
the current iteration. The completed model is then used for the usual
trust-region purpose of generating a reliable trial step.

\subsection{MAP-poisedness and geometry management}\label{subsec:geometry-bottleneck}

The accuracy of an interpolation model also depends on the geometry of the
interpolation set. Classical DFO theory uses $\Lambda$-poisedness and
related conditioning requirements to obtain fully-linear or
fully-quadratic accuracy in a trust region \cite{conn2009introduction}.
Implementations maintain this geometry by point replacement, geometry
checks, and occasional model-improvement steps.

For the completion problem \eqref{eq:intro-prior-completion}, the geometry
condition should be stated for the metric used in that problem. In the
scaled coordinates used throughout this paper, let $\widehat A(Y)$ be the
design matrix associated with a candidate set $Y$, and let $\widehat W_k$
be the precision matrix in the completion step. We define
\begin{equation}\label{eq:intro-map-geom-matrix}
    \widehat M_k(Y)=\widehat A(Y)\widehat W_k^{-1}\widehat A(Y)^\top.
\end{equation}
A set is called $\mu_M$-\emph{MAP-poised} if
\begin{equation}\label{eq:intro-map-poised}
    \lambda_{\min}\bigl(\widehat M_k(Y)\bigr)\ge \mu_M.
\end{equation}
When $\widehat W_k=I$, this condition becomes a lower bound on the smallest
singular value of the scaled design matrix, the spectral form behind usual
poisedness requirements. For general $\widehat W_k$, MAP-poisedness plays
the role of poisedness for the MAP completion problem: it is the geometry
condition under which the completed model is stable and fully linear.
The formal definition and its use in the repair mechanism are given in
Section~\ref{subsec:framework-geometry}.

\subsection{The BUP-TR method}\label{subsec:approach}

At iteration $k$, BUP-TR writes the local quadratic model in the scaled
variable $s=\Delta_k u$ and denotes the corresponding coefficient vector by
$\widehat c\in\mathbb{R}^q$. The interpolation equations become
\begin{equation}\label{eq:intro-normalized-interp}
    \widehat A_k\widehat c=b_k.
\end{equation}
Given a prior mean $\widehat c_k^\pi$ and a precision matrix
$\widehat W_k\succ0$, BUP-TR selects
\begin{equation}\label{eq:intro-map-projection}
    \widehat c_k
    =\argmin_{\widehat c}\;\frac12
    \|\widehat c-\widehat c_k^\pi\|_{\widehat W_k}^2
    \quad\text{subject to}\quad
    \widehat A_k\widehat c=b_k .
\end{equation}
The interpolation values determine the feasible set, while the prior mean
and precision choose one point in that set. The precision matrix also
enters the MAP-poisedness test \eqref{eq:intro-map-poised} and the ranking
of repair candidates. Embedding these components in a Powell-type
trust-region implementation gives the solver BUP-NEWUOA used in
Section~\ref{sec:experiments}.

\subsection{Related work}\label{subsec:related}

BUP-TR belongs primarily to the literature on model-based DFO and
underdetermined quadratic interpolation. Powell's UOBYQA constructs a fully
determined quadratic interpolation model; NEWUOA and BOBYQA showed that
powerful solvers can be built with fewer interpolation points by updating a
quadratic model through a minimum-change Hessian criterion
\cite{powell2002uobyqa,powell2006newuoa,powell2009bobyqa}. Conn--Toint and
minimum-norm Hessian models provide related ways of resolving the
undetermined coefficients, and modern implementations such as PDFO/PRIMA
and DFO-LS have improved the robustness and accessibility of this class of
methods \cite{wild2008mnh,ragonneau2024pdfo,zhang2025prima,cartis2019dfols}.
Recent work has modified the model subproblem itself, including
trust-region-iteration-based updates, Sobolev-norm updates, and regional
minimal updating \cite{xie2025underdetermined,xie2025h2,xie2026remu,XieYuan2023}. The
present paper follows this model-subproblem perspective and introduces a
prior-defined metric into the completion rule.

Probabilistic models have also been used in optimization. Probabilistic
trust-region methods prove convergence under sufficiently accurate model
events, while Bayesian optimization usually chooses new sample locations by
optimizing an acquisition function
\cite{bandeira2014convergence,rasmussen2006gp,snoek2012practical,shahriari2016taking,frazier2018tutorial}.
In BUP-TR, probabilistic modeling enters through local model construction: the
prior defines the reference vector and the metric for completion, and the
resulting model is passed to the usual trust-region step computation and
acceptance test.

Subspace DFO methods address high dimensionality from another direction by
building models in lower-dimensional spaces or by using subspace steps
\cite{zhang2025subspace}. Such ideas are complementary to BUP-TR, which
focuses on completing an underdetermined quadratic model after the working
interpolation set has been selected.

\subsection{Contributions and organization}\label{subsec:contrib}

The main contributions are as follows.
\begin{itemize}[leftmargin=2em]
\item \emph{A model-completion view of underdetermined interpolation.}
We formulate underdetermined quadratic interpolation as the problem of
selecting one model from the affine family satisfying the interpolation
equations. This viewpoint places minimum-norm, minimum-change,
Sobolev-norm, and regional minimal updating models into a common completion
framework.

\item \emph{Prior-regularized completion.}
We replace the prescribed completion norm by a positive definite metric
defined by a prior model. The reference vector and metric can be constructed
from information already available during the run, such as local surrogate
fits, least-squares curvature estimates, or previously accepted models.

\item \emph{MAP-poisedness for geometry management.}
We introduce MAP-poisedness, a spectral conditioning condition for the MAP
completion problem. It plays the role of poisedness in the fully-linear
analysis and is used to filter and rank geometry-repair candidates.

\item \emph{Convergence and evaluation complexity.}
For the hard-MAP version, we prove that the completed models are fully
linear under standard smoothness, geometry, and prior-accuracy assumptions.
The resulting trust-region method satisfies global first-order convergence
and an $\mathcal{O}(\varepsilon^{-2})$ evaluation-complexity bound, with
repair evaluations included.

\item \emph{A NEWUOA-style implementation and numerical validation.}
We implement the construction in a Powell-type trust-region solver,
BUP-NEWUOA, and compare it with NEWUOA, UOBYQA, Nelder--Mead, and CMA-ES on
a benchmark suite of 85 problem--dimension pairs.
\end{itemize}

Section~\ref{sec:prelim} introduces the scaled quadratic representation and
the trust-region framework. Section~\ref{sec:bup} presents the completion
rule, prior constructions, MAP-poisedness repair, and the algorithm.
Section~\ref{sec:conv} proves convergence and evaluation complexity.
Section~\ref{sec:experiments} reports the numerical results. Proofs and
supplementary material are collected in the appendices.

\section{Trust-Region Interpolation Preliminaries}\label{sec:prelim}

\paragraph{Notation.}
Unless otherwise stated, vector norms are Euclidean norms.
For matrices, $\|\cdot\|_2$ and $\|\cdot\|_F$ denote the spectral and Frobenius norms, respectively.
For a symmetric matrix $H$, $\operatorname{vech}(H)$ denotes the half-vectorization formed by stacking the diagonal and upper-triangular entries in a fixed order.
Table~\ref{tab:notation} collects the principal symbols used
throughout the paper; each symbol is defined formally in the section indicated.

\begin{table}[t]
\caption{Principal notation.}\label{tab:notation}
\centering\small
\begin{tabular}{@{}lll@{}}
\toprule
Symbol & Meaning & Defined in \\
\midrule
\multicolumn{3}{@{}l}{\textbf{Problem and Trust Region}}\\
$n$ & ambient dimension & \S\ref{subsec:prelim-tr}\\
$f$ & objective function & \eqref{eq:prob}\\
$x_k,\;\Delta_k$ & TR center and radius at iteration $k$ & \S\ref{subsec:prelim-tr}\\
$\Delta_{\max}$ & maximum radius & \S\ref{subsec:prelim-tr}\\
$\ctrim$ & geometry-radius factor ($\Dgeo{k}=\ctrim\Delta_k$) & \S\ref{subsec:prelim-tr}\\
$Y_k$ & interpolation set ($m+1$ points) & \S\ref{subsec:prelim-tr}\\
$m$ & $|Y_k|-1$; we use the default choice $m=2n$ in the algorithm & \S\ref{subsec:prelim-tr}, \S\ref{subsec:algo-design}\\
\midrule
\multicolumn{3}{@{}l}{\textbf{Coefficient Space and Models}}\\
$q$ & $\frac{(n+1)(n+2)}{2}$; number of quadratic coefficients & \eqref{eq:qdim}\\
$\widehat\phi(u)$ & scaled feature vector & \eqref{eq:phihat}\\
$\widehat c_k$ & scaled coefficient vector (MAP solution) & \eqref{eq:chat-def}\\
$\widehat A_k$ & scaled design matrix & \S\ref{subsec:prelim-norm}\\
$b_k$ & sampled function-value vector & \S\ref{subsec:prelim-vec}\\
$m_k(s)$ & quadratic model at iteration $k$ & \eqref{eq:quad-model}\\
$g_k,\;H_k$ & model gradient and Hessian at $s=0$ & \S\ref{subsec:prelim-cauchy}\\
\midrule
\multicolumn{3}{@{}l}{\textbf{Prior-Regularized MAP Completion}}\\
$\widehat c_k^\pi$ & prior mean in the scaled coefficient representation & \eqref{eq:bup-completion}\\
$\widehat W_k$ & precision matrix ($\widehat\Sigma_k^{-1}$) & \eqref{eq:bup-completion}\\
$w_{\min},\;w_{\max}$ & spectral bounds on $\widehat W_k$ & \eqref{eq:bup-W-bounds}\\
$\widehat c_{k'}$ & accepted-model coefficient vector used to form the prior & \eqref{eq:bup-accepted-map-prior}\\
\midrule
\multicolumn{3}{@{}l}{\textbf{Geometry and Repair}}\\
$\widehat M_k(Y)$ & MAP geometry matrix $\widehat A(Y)\widehat W_k^{-1}\widehat A(Y)^\top$ & \eqref{eq:repair-Mhat}\\
$\mu_M$ & MAP-poisedness threshold & \eqref{eq:repair-map-poised}\\
$\mu_0$ & guaranteed MAP-poisedness of fallback set & \eqref{eq:conv-mu0}\\
$T_{\mathrm{try}}$ & incremental repair attempt budget & \S\ref{subsec:repair-when}\\
$Y_k^{\mathrm{fb}}$ & fallback interpolation set & \eqref{eq:repair-fallback-set}\\
\midrule
\multicolumn{3}{@{}l}{\textbf{Convergence Constants}}\\
$B_\phi,\;\Bgeo$ & feature-vector norm bounds ($\|u\|\le 1$ and $\|u\|\le\ctrim$) & \eqref{eq:conv-Bphi-const}, \eqref{eq:conv-Bphi-geo}\\
$\bar\kappa_\pi$ & prior-accuracy constant (in the $\widehat W_k$-norm) & \eqref{eq:conv-prior-acc}\\
$\kappa_e$ & coefficient-error constant & \eqref{eq:conv-ke}\\
$\kappa_f,\;\kappa_g$ & fully-linear function/gradient constants & \eqref{eq:conv-kappas}\\
$H_{\max}$ & model Hessian bound & \eqref{eq:conv-Hmax}\\
$C_{\mathrm{eval}}$ & evaluation-complexity constant & Theorem~\ref{thm:conv-complexity}\\
\midrule
\multicolumn{3}{@{}l}{\textbf{Repair}}\\
$N_{\mathrm{cand}}$ & repair candidate pool size & \S\ref{subsec:repair-when}\\
\bottomrule
\end{tabular}
\end{table}

\subsection{Scaled quadratic representation}\label{subsec:prelim-tr}
	At iteration $k$, the method maintains
	a center $x_k\in\R^n$, a trust-region radius $\Delta_k\in(0,\Delta_{\max}]$, and an interpolation set
	\[
	Y_k=\{y_k^{(0)},y_k^{(1)},\ldots,y_k^{(m)}\}\subset B(x_k,\Dgeo{k}),
	\qquad y_k^{(0)}=x_k,
	\]
	where
	$B(x_k,r):=\{x\in\R^n:\|x-x_k\|_2\le r\}$ and the \emph{geometry radius} is
	$\Dgeo{k}:=\ctrim\,\Delta_k$ with a fixed constant $\ctrim\ge 1$.
	Allowing $\ctrim>1$ is standard in model-based trust-region methods for DFO
	\cite[Ch.~10]{conn2009introduction}.
	The trust-region step itself is still constrained to $\|s\|_2\le \Delta_k$.
	Define displacements
	\[
	s_k^{(i)}:=y_k^{(i)}-x_k,\qquad i=0,1,\ldots,m,
	\qquad\text{so that}\qquad \|s_k^{(i)}\|_2\le \Dgeo{k}.
	\]
	
	We work with quadratic models in displacement form:
	\begin{equation}\label{eq:quad-model}
		m_k(s)=c_{k,0}+c_{k,1}^\top s+\frac12 s^\top C_{k,2}\,s,
		\qquad s\in\R^n,
	\end{equation}
	where $c_{k,0}\in\R$, $c_{k,1}\in\R^n$, and $C_{k,2}\in\mathbb{S}^n$ is symmetric
	(here $\mathbb{S}^n$ denotes the set of real symmetric $n\times n$ matrices).
	
	\phantomsection\label{subsec:prelim-vec}
The MAP completion and geometry analysis in Section~\ref{sec:bup}
	operate on a vectorized coefficient representation; we set up the precise correspondence here.
	Let $\mathrm{vech}(\cdot)$ denote the half-vectorization of a symmetric matrix (stacking the diagonal and upper-triangular entries).
	For $s\in\R^n$, define the quadratic monomial vector compatible with $\mathrm{vech}(\cdot)$:
	\begin{equation}\label{eq:qvec-def}
		\qvec(s):=
		\begin{bmatrix}
			\frac12 s_1^2\\
			\vdots\\
			\frac12 s_n^2\\
			s_1 s_2\\
			\vdots\\
			s_{n-1}s_n
		\end{bmatrix}
		\in\R^{q_H},
		\qquad q_H:=\frac{n(n+1)}{2},
	\end{equation}
	where the off-diagonal terms follow the same ordering as the upper-triangular part in $\mathrm{vech}(\cdot)$.
	Then, for any $H\in\mathbb{S}^n$,
	\begin{equation}\label{eq:qvec-identity}
		\qvec(s)^\top \mathrm{vech}(H)=\frac12 s^\top H s.
	\end{equation}
	
	Define the coefficient vector
	\begin{equation}\label{eq:cvec}
		c_k:=
		\begin{bmatrix}
			c_{k,0}\\
			c_{k,1}\\
			\mathrm{vech}(C_{k,2})
		\end{bmatrix}
		\in\R^q,
		\qquad q=\frac{(n+1)(n+2)}{2}=1+n+q_H,
	\end{equation}
	and the (unscaled) feature vector
	\begin{equation}\label{eq:phi-unnorm}
		\phi(s):=
		\begin{bmatrix}
			1\\
			s\\
			\qvec(s)
		\end{bmatrix}\in\R^q,
		\qquad\text{so that}\qquad m_k(s)=\phi(s)^\top c_k.
	\end{equation}
	
	Imposing interpolation on $Y_k$ yields
	\begin{equation}\label{eq:interp-unnorm}
		m_k\!\big(s_k^{(i)}\big)=f\!\big(y_k^{(i)}\big),\qquad i=0,1,\ldots,m.
	\end{equation}
	Let $b_k\in\R^{m+1}$ collect the sampled values $(b_k)_i:=f(y_k^{(i)})$, and define the design matrix $A_k\in\R^{(m+1)\times q}$ row-wise by
	\[
	(A_k)_{i,:}:=\phi\!\big(s_k^{(i)}\big)^\top,\qquad i=0,1,\ldots,m.
	\]
	Then \eqref{eq:interp-unnorm} is equivalent to
	\begin{equation}\label{eq:Ac=b}
		A_k c_k=b_k.
	\end{equation}
	When $m+1<q$, the system \eqref{eq:Ac=b} is underdetermined and admits infinitely many solutions.
	
	\phantomsection\label{subsec:prelim-norm}
	To obtain geometry conditions and constants independent of $\Delta_k$, we work in scaled coordinates.
	Define scaled displacements
	\begin{equation}\label{eq:u-def}
		u_k^{(i)}:=\frac{s_k^{(i)}}{\Delta_k},\qquad \|u_k^{(i)}\|_2\le \ctrim,\qquad i=0,1,\ldots,m.
	\end{equation}
	Define the scaled feature vector
	\begin{equation}\label{eq:phihat}
		\widehat\phi(u):=
		\begin{bmatrix}
			1\\
			u\\
			\qvec(u)
	\end{bmatrix}\in\R^q,
	\qquad u\in\R^n.
\end{equation}
	Next, define the scaled coefficient vector $\widehat c_k\in\R^q$ by
	\begin{equation}\label{eq:chat-def}
		\widehat c_k:=
		\begin{bmatrix}
			c_{k,0}\\
			\Delta_k\,c_{k,1}\\
			\Delta_k^2\,\mathrm{vech}(C_{k,2})
		\end{bmatrix}.
	\end{equation}
	Then, for any $s=\Delta_k u$,
	\begin{equation}\label{eq:mhat-identity}
		m_k(\Delta_k u)=\widehat\phi(u)^\top \widehat c_k.
	\end{equation}
	
	Define the scaled design matrix $\widehat A_k\in\R^{(m+1)\times q}$ by
	\[
	(\widehat A_k)_{i,:}:=\widehat\phi\!\big(u_k^{(i)}\big)^\top,\qquad i=0,1,\ldots,m.
	\]
	The interpolation constraints \eqref{eq:interp-unnorm} are equivalently written as
	\begin{equation}\label{eq:Ahat-chat=b}
		\widehat A_k \widehat c_k=b_k.
	\end{equation}
	This scaled representation is the one used throughout the prior-regularized MAP completion and geometry certification in
Section~\ref{sec:bup}.
	
	\subsection{Trust-region decrease condition}\label{subsec:prelim-accept}
	Given $m_k$ and $\Delta_k$, a trial step $s_k$ is computed by approximately solving
	\begin{equation}\label{eq:tr-subprob}
		\min_{\|s\|_2\le \Delta_k} m_k(s).
	\end{equation}
	Define predicted and actual reductions
	\begin{equation}\label{eq:pred-ared}
		\pred_k:=m_k(0)-m_k(s_k),
		\qquad
		\ared_k:=f(x_k)-f(x_k+s_k),
	\end{equation}
	and the safeguarded ratio
	\begin{equation}\label{eq:rho}
		\rho_k:=
		\begin{cases}
			\ared_k/\pred_k, & \text{if }\pred_k>0,\\
			-\infty, & \text{if }\pred_k\le 0.
		\end{cases}
	\end{equation}
	Acceptance and radius updates follow the standard rules: for fixed $0<\eta_1<\eta_2<1$ and $0<\gamma_{\mathrm{dec}}<1<\gamma_{\mathrm{inc}}$,
	\[
	x_{k+1}=
	\begin{cases}
		x_k+s_k, & \text{if }\rho_k\ge \eta_1,\\
		x_k, & \text{otherwise},
	\end{cases}
	\]
	and
	\[
	\Delta_{k+1}=
	\begin{cases}
		\min\{\gamma_{\mathrm{inc}}\Delta_k,\Delta_{\max}\},
		& \text{if }\rho_k\ge \eta_2,\\
		\Delta_k, & \text{if }\eta_1\le \rho_k<\eta_2,\\
		\gamma_{\mathrm{dec}}\Delta_k, & \text{if }\rho_k<\eta_1.
	\end{cases}
	\]
After updating $(x_k,\Delta_k)$, the interpolation set is refreshed and repaired if the geometry certification in
Subsection~\ref{sec:repair} is violated.

\begin{remark}[Radius safeguard in the implementation]\label{rem:radius-safeguard}
In the implementation, after an unsuccessful step we additionally impose
\[
    \Delta_{k+1}
    \leftarrow
    \max\!\left\{
        \gamma_{\mathrm{dec}}\Delta_k,\;
        c_{\mathrm{sg}}\,
        \frac{\|g_k\|_2}{\max\{\|H_k\|_2,1\}}
    \right\},
\]
where $c_{\mathrm{sg}}>0$ is a small constant.
This safeguard prevents premature collapse of the radius when the model gradient is large.
It is not used in the convergence and complexity analysis of Section~\ref{sec:conv}.
\end{remark}
	
	\phantomsection\label{subsec:prelim-cauchy}
	Let
	\[
	g_k:=\nabla m_k(0)=c_{k,1},\qquad H_k:=\nabla^2 m_k(0)=C_{k,2}.
	\]
Our analysis requires a standard Cauchy-type decrease condition on the step $s_k$:
if $g_k\neq 0$, then $\|s_k\|_2\le \Delta_k$ and
\begin{equation}\label{eq:cauchy-dec}
	\pred_k \ \ge\ \frac12\,\|g_k\|_2 \min\!\left\{\Delta_k,\ \frac{\|g_k\|_2}{\|H_k\|_2}\right\},
\end{equation}
with the convention $\|g_k\|_2/\|H_k\|_2:=+\infty$ when $\|H_k\|_2=0$.
This condition is satisfied by standard trust-region subproblem solvers, including the Cauchy step and
truncated conjugate gradients with negative-curvature detection;
it is recorded formally as Assumption~\ref{asmp:conv-cauchy} in Section~\ref{sec:conv}.

\subsection{Evaluation accounting}\label{subsec:prelim-evals}
We count objective evaluations, since these dominate the cost in the
black-box setting.  After the initial interpolation set has been evaluated,
each main iteration uses at most one trial evaluation.  Additional
evaluations may be spent only when the geometry routine adds repair points.
For the first $K$ main iterations we write
\begin{equation}\label{eq:eval-count}
    N_f(K)=\sum_{k=0}^{K-1}
    \bigl(N_k^{\mathrm{trial}}+N_k^{\mathrm{rep,base}}
          +N_k^{\mathrm{rep,crit}}\bigr),
\end{equation}
where $N_k^{\mathrm{trial}}\in\{0,1\}$, $N_k^{\mathrm{rep,base}}$ counts
repair evaluations during the first model-building pass of iteration $k$,
and $N_k^{\mathrm{rep,crit}}$ counts extra repair evaluations caused by
criticality shrinks.  The geometry routine gives a uniform bound on
$N_k^{\mathrm{rep,base}}$; the criticality contribution is controlled in
Section~\ref{sec:conv} by the usual radius-decrease argument.

	\section{Prior-Regularized Model Completion in BUP-TR}\label{sec:bup}

This section turns the model-completion viewpoint of
Section~\ref{sec:intro} into the BUP-TR construction.  We use the scaled
interpolation system \eqref{eq:Ahat-chat=b} and keep the notation from
Section~\ref{sec:prelim}.  The central operation is a constrained quadratic
minimization in coefficient space: a prior coefficient vector is projected,
in the metric specified by a precision matrix, onto the affine set of
interpolating models.  This metric projection explains the word
``projection'' in the title.  The Schur-complement matrix of the projection
also supplies the geometry test used below.

\subsection{The MAP completion problem and projection formula}\label{subsec:bup-unified}
	Let $\widehat c_k\in\R^q$ be the scaled coefficient vector defined in
	\eqref{eq:chat-def}.  For the current interpolation set $Y_k$, the
	constraints are
	\begin{equation}\label{eq:bup-Achat}
		\widehat A_k \widehat c = b_k,
	\end{equation}
	where $\widehat A_k$ and $b_k$ are the scaled design matrix and sampled
	values introduced in \eqref{eq:Ahat-chat=b}.  In the underdetermined case,
	\eqref{eq:bup-Achat} defines an affine set of feasible coefficient vectors.
	BUP-TR chooses one of them by measuring distance to a prior model.  Given a
	prior mean $\widehat c_k^\pi\in\R^q$ and an SPD precision matrix
	$\widehat W_k\succ0$, we set
	\begin{equation}\label{eq:bup-completion}
		\widehat c_k
		:=\arg\min_{\widehat c\in\R^q}\ \frac12\|\widehat c-\widehat c_k^\pi\|_{\widehat W_k}^2
		\quad\text{s.t.}\quad \widehat A_k\widehat c=b_k,
		\qquad
		\|v\|_{\widehat W}^2:=v^\top \widehat W v.
	\end{equation}
	Thus $\widehat c_k$ is the $\widehat W_k$-orthogonal projection of
	$\widehat c_k^\pi$ onto the affine set \eqref{eq:bup-Achat}.  Classical
	completion rules are recovered by choosing the reference vector and metric
	to reproduce the corresponding minimum-norm or minimum-change criterion.

	The Bayesian interpretation is immediate.  If we place a Gaussian prior on
	the scaled coefficients,
	\begin{equation}\label{eq:bup-prior}
		\widehat c\sim\mathcal{N}(\widehat c_k^\pi,\widehat\Sigma_k),
		\qquad
		\widehat\Sigma_k\succ 0,
		\qquad
		\widehat W_k:=\widehat\Sigma_k^{-1}\succ 0,
	\end{equation}
	and enforce noiseless interpolation $\widehat A_k\widehat c=b_k$, then
	\eqref{eq:bup-completion} is exactly the MAP estimator.
	Figure~\ref{fig:map-projection} illustrates the geometry.

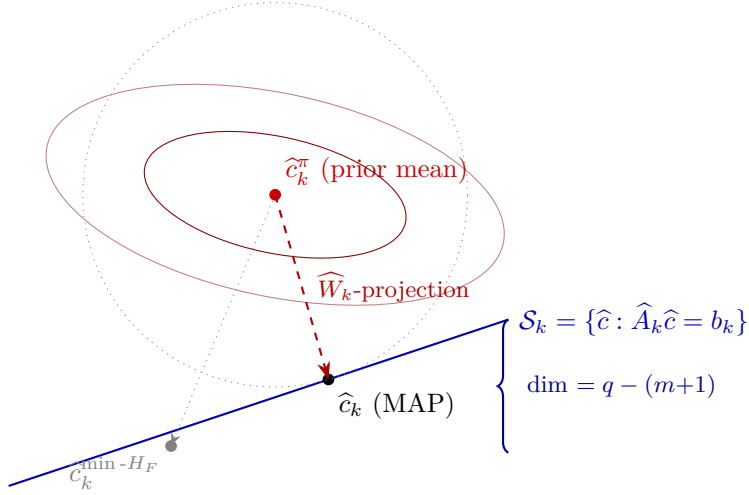
\begin{figure}[t]
\centering
\begin{tikzpicture}[>=Stealth, scale=1.1]
  \coordinate (origin) at (0,0);

  \draw[thick, blue!70!black] (-2.8,-0.9) -- (3.2,1.1)
    node[right,font=\small]{$\mathcal{S}_k=\{\widehat c:\widehat A_k\widehat c=b_k\}$};

  \coordinate (cpi) at (0.4,2.6);
  \fill[red!80!black] (cpi) circle (2pt) node[above right,font=\small]{$\widehat c_k^\pi$ (prior mean)};

  \coordinate (ck) at (1.04,0.38);
  \fill[black] (ck) circle (2pt) node[below right,font=\small,yshift=-1pt]{$\widehat c_k$ (MAP)};

  \draw[->,thick,red!70!black,dashed] (cpi) -- (ck)
    node[midway,right,font=\footnotesize,xshift=2pt]{$\widehat W_k$-projection};

  \begin{scope}[shift={(cpi)},rotate=-12]
    \draw[red!50!black,thin] (0,0) ellipse (1.6cm and 0.7cm);
    \draw[red!50!black,thin,opacity=0.5] (0,0) ellipse (2.8cm and 1.22cm);
  \end{scope}

  \coordinate (cmin) at (-0.85,-0.42);
  \fill[gray] (cmin) circle (2pt) node[below left,font=\small]{$\widehat c_k^{\min\text{-}H_F}$};

  \begin{scope}[shift={(cpi)}]
    \draw[gray,thin,dotted] (0,0) circle (2.31cm);
  \end{scope}
  \draw[->,thin,gray,dotted] (cpi) -- (cmin);

  \draw[decorate,decoration={brace,amplitude=4pt,mirror},thick,blue!60!black]
    (3.2,1.1) -- (3.2,-0.5) node[midway,right,font=\footnotesize,xshift=3pt]{dim $= q-(m{+}1)$};
\end{tikzpicture}
\caption{MAP projection in the scaled quadratic coefficient representation
$\mathbb{R}^q$.  The affine set $\mathcal{S}_k$ contains all
interpolants consistent with data.  The prior mean $\widehat c_k^\pi$
(which may come from a Bayesian surrogate or a recursive accepted-model
prior) defines the projection center, and the precision matrix
$\widehat W_k$ defines the metric (solid ellipses).
The MAP solution $\widehat c_k$ is the closest point on $\mathcal{S}_k$
in the $\widehat W_k$-norm, whereas the classical minimum-curvature
solution $\widehat c_k^{\min\text{-}H_F}$ corresponds to the
Euclidean-closest point (dotted circle) with zero Hessian prior.}
\label{fig:map-projection}
\end{figure}

	To see how classical minimum-curvature completion fits this template,
	partition $\widehat c=[\widehat c^{(0)};\widehat c^{(g)};
	\widehat c^{(H)}]$ as in \eqref{eq:chat-def}
	and consider the precision family
	\begin{equation}\label{eq:bup-W-family}
		\widehat W_k(\tau):=\mathrm{diag}\!\big(\tau I_{1+n},\ I_{q_H}\big),
		\qquad \tau>0.
	\end{equation}
	In the regime $\tau\downarrow 0$ with $\widehat c_k^{\pi,(H)}=0$,
	problem \eqref{eq:bup-completion} increasingly emphasizes minimizing
	$\|\widehat c^{(H)}\|_2$, which corresponds (up to a
	dimension-independent constant) to minimizing $\|H\|_F$ via
	\eqref{eq:bup-mapback}.
	Thus the familiar minimum-Frobenius-norm Hessian completion is
	recovered as a limiting weak-prior case of the prior-regularized
	completion \eqref{eq:bup-prior}--\eqref{eq:bup-completion}.
	
	We now record the block structure of the coefficients and the spectral bounds imposed on the precision matrix.
	
	\phantomsection\label{subsec:bup-prior}
	We partition $\widehat c$ according to constant, gradient, and Hessian components:
	\[
	\widehat c=
	\begin{bmatrix}
		\widehat c^{(0)}\\
		\widehat c^{(g)}\\
		\widehat c^{(H)}
	\end{bmatrix}
	=
	\begin{bmatrix}
		c_{k,0}\\
		\Delta_k c_{k,1}\\
		\Delta_k^2\,\mathrm{vech}(C_{k,2})
	\end{bmatrix},
	\]
	where $\widehat c^{(g)}\in\R^n$ and $\widehat c^{(H)}\in\R^{q_H}$.
	In the analysis (Section~\ref{sec:conv}), we impose explicit spectral bounds on the precision:
	\begin{equation}\label{eq:bup-W-bounds}
		w_{\min}I\ \preceq\ \widehat W_k\ \preceq\ w_{\max}I,
		\qquad 0<w_{\min}\le w_{\max}<\infty,
	\end{equation}
	which can be enforced in practice by eigenvalue clipping (or diagonal clipping in large-scale variants).
	The clipping step is used only to impose the uniform spectral bounds in \eqref{eq:bup-W-bounds}.
	Consequently, the coefficient norm induced by $\widehat W_k$ is uniformly equivalent to the Euclidean norm,
	\[
		w_{\min}\|v\|_2^2\le \|v\|_{\widehat W_k}^2\le w_{\max}\|v\|_2^2,
		\qquad v\in\mathbb{R}^q,
	\]
	with constants independent of the iteration.
	These bounds prevent degeneracy of the completion norm and keep the constants
	in the projection and fully-linear estimates uniform across iterations.
	
	\phantomsection\label{subsec:bup-hard}
	Suppose function values are noiseless and enforce the hard interpolation constraints \eqref{eq:bup-Achat}.
	The hard-MAP estimator is the unique solution of \eqref{eq:bup-completion}:
	\begin{equation}\label{eq:bup-hard-map}
		\widehat c_k
		:=\arg\min_{\widehat c\in\R^q}\ \frac12\|\widehat c-\widehat c_k^\pi\|_{\widehat W_k}^2
		\quad\text{s.t.}\quad \widehat A_k \widehat c=b_k.
	\end{equation}

	Existence and uniqueness follow directly. Since $\widehat W_k\succ 0$, the objective is strictly convex.
	If $\widehat A_k$ has full row rank, then the constraints are feasible for any $b_k\in\R^{m+1}$ and the minimizer is unique.

	The KKT system gives the projection formula. Introduce Lagrange multipliers $\lambda\in\R^{m+1}$ and define
	\[
	\mathcal{L}(\widehat c,\lambda)
	:=\frac12(\widehat c-\widehat c_k^\pi)^\top \widehat W_k(\widehat c-\widehat c_k^\pi)
	+\lambda^\top(\widehat A_k\widehat c-b_k).
	\]
	First-order optimality conditions are
	\begin{align}
		\widehat W_k(\widehat c-\widehat c_k^\pi)+\widehat A_k^\top \lambda&=0, \label{eq:bup-kkt1}\\
		\widehat A_k\widehat c-b_k&=0. \label{eq:bup-kkt2}
	\end{align}
	From \eqref{eq:bup-kkt1},
	\begin{equation}\label{eq:bup-c-lambda}
		\widehat c=\widehat c_k^\pi-\widehat W_k^{-1}\widehat A_k^\top \lambda.
	\end{equation}
	Substituting into \eqref{eq:bup-kkt2} yields the $(m+1)\times(m+1)$ system
	\begin{equation}\label{eq:bup-normal-eq}
		\widehat M_k\,\lambda=\widehat A_k\widehat c_k^\pi-b_k,
		\qquad
		\widehat M_k:=\widehat A_k\widehat W_k^{-1}\widehat A_k^\top.
	\end{equation}
	If $\widehat A_k$ has full row rank, then $\widehat M_k\succ 0$ and is invertible, hence
	\begin{equation}\label{eq:bup-hard-closed}
		\widehat c_k
		=
		\widehat c_k^\pi
		+\widehat W_k^{-1}\widehat A_k^\top \widehat M_k^{-1}\bigl(b_k-\widehat A_k\widehat c_k^\pi\bigr).
	\end{equation}

	The dominant cost in \eqref{eq:bup-hard-closed} is solving the SPD
	system \eqref{eq:bup-normal-eq} of size $(m+1)\times(m+1)$,
	e.g., by Cholesky factorization.
	In the regime $m=\mathcal{O}(n)$, this cost is modest compared to full
	quadratic interpolation, which requires $q=\Theta(n^2)$ samples.

\phantomsection\label{subsec:bup-soft}
For noisy or deliberately smoothed modeling one may replace exact
interpolation by a penalized least-squares term.  With an observation
covariance $R_k\succ0$, the corresponding soft-MAP estimator is
\begin{equation}\label{eq:bup-soft-map}
    \widehat c_k
    :=\arg\min_{\widehat c\in\R^q}
    \frac12\|\widehat A_k\widehat c-b_k\|_{R_k^{-1}}^2
    +\frac12\|\widehat c-\widehat c_k^\pi\|_{\widehat W_k}^2 .
\end{equation}
It has the Woodbury form
\begin{equation}\label{eq:bup-soft-closed-2}
    \widehat c_k
    =
    \widehat c_k^\pi
    +\widehat W_k^{-1}\widehat A_k^\top
    \bigl(\widehat A_k\widehat W_k^{-1}\widehat A_k^\top + R_k\bigr)^{-1}
    \bigl(b_k-\widehat A_k\widehat c_k^\pi\bigr).
\end{equation}
The noiseless hard-MAP model is the object analyzed in
Section~\ref{sec:conv}; the soft-MAP formula is recorded here because it is
used in the discussion of noisy variants.  A derivation and the exact-center
variant are given in Appendix~\ref{app:softmap}.

	\subsubsection{Recovering the quadratic model}\label{subsec:bup-backmap}
	The estimators \eqref{eq:bup-hard-closed}--\eqref{eq:bup-soft-closed-2} produce scaled coefficients $\widehat c_k$.
	To form the displacement model \eqref{eq:quad-model}, write
	\[
	\widehat c_k=
	\begin{bmatrix}
		\widehat c_k^{(0)}\\
		\widehat c_k^{(g)}\\
		\widehat c_k^{(H)}
	\end{bmatrix},
	\]
	and map back via \eqref{eq:chat-def}:
	\begin{equation}\label{eq:bup-mapback}
		c_{k,0}:=\widehat c_k^{(0)},
		\qquad
		c_{k,1}:=\Delta_k^{-1}\widehat c_k^{(g)},
		\qquad
		\mathrm{vech}(C_{k,2}):=\Delta_k^{-2}\widehat c_k^{(H)}.
	\end{equation}
Then $m_k(s)$ is defined by \eqref{eq:quad-model}. Consequently,
	\begin{equation}\label{eq:bup-derivs-origin}
		g_k=\nabla m_k(0)=c_{k,1}=\Delta_k^{-1}\widehat c_k^{(g)},
		\qquad
		H_k=\nabla^2 m_k(0)=C_{k,2}.
	\end{equation}
	
\subsection{Constructing the prior model}\label{subsec:bup-prior-sources}
The completion formula requires a prior mean $\widehat c_k^\pi$ and a
precision matrix $\widehat W_k$.  We describe two sources.  The first uses a
local Gaussian-process surrogate to supply derivative information.  The
second, used in BUP-NEWUOA, transports the most recently accepted model and
therefore remains entirely within the trust-region loop.
	
	\phantomsection\label{subsec:bup-Wchoice}
	Across these constructions, we use a common block-diagonal
	parameterization of the precision matrix:
	\begin{equation}\label{eq:bup-W-block}
		\widehat W_k=\mathrm{diag}\bigl(w_0,\ W_{g,k},\ W_{H,k}\bigr),
	\end{equation}
	where $w_0\in[w_{\min},w_{\max}]$ is the precision for the constant term,
	$W_{g,k}\in\R^{n\times n}$ and $W_{H,k}\in\R^{q_H\times q_H}$ are diagonal
	matrices with entries clipped to $[w_{\min},w_{\max}]$.
	This parameterization makes explicit how the prior acts on the three
	coefficient blocks and, through the Hessian block in
	\eqref{eq:bup-mapback}, how curvature is regularized.
	In both constructions, clipping the diagonal precision entries to
	$[w_{\min},w_{\max}]$ guarantees
	$w_{\min}I \preceq \widehat W_k \preceq w_{\max}I$ by construction.
	
	\subsubsection{Gaussian-process prior}\label{subsec:bup-prior-mean}

	The first prior construction derives prior information from a local
	Gaussian-process (GP) surrogate~\cite{rasmussen2006gp} fitted to past
	function values, without derivative-oracle calls.
	For the prior mean, at the current center $x_k$, derivative moments
	$\nabla\mu_k(x_k)$ and $\nabla^2\mu_k(x_k)$ of the GP posterior mean are
	rescaled to the scaled coefficient basis, yielding
	\begin{equation}\label{eq:bup-prior-gp}
	\widehat c_k^\pi
	:=
	\begin{bmatrix}
	f(x_k)\\
	\Delta_k\,\nabla\mu_k(x_k)\\
	\Delta_k^2\,\mathrm{vech}(\nabla^2\mu_k(x_k))
	\end{bmatrix}
	\in\R^q.
	\end{equation}
	The constant block uses the exact center value so that $m_k(0)=f(x_k)$.

	To define the GP-derived precision, let $\mathcal{D}_k:=\{(x^{(j)},f(x^{(j)}))\}_{j=1}^{N_k}$ denote the
	evaluated data available at iteration $k$ (in practice, the local pool
	used by the GP fit).
	Let $\Sigma_{z,k}=\mathrm{Cov}([\nabla \tilde f(x_k);\,
	\mathrm{vech}(\nabla^2\tilde f(x_k))]\mid\mathcal{D}_k)$ denote
	the posterior covariance of the stacked derivative vector under the
	GP fitted to the $N_{\mathrm{pool}}$ nearest evaluations.
	The scaling $D_k:=\mathrm{diag}(\Delta_k I_n,\Delta_k^2 I_{q_H})$ yields
	a scaled covariance $\Sigma_{\widehat z,k}=D_k\Sigma_{z,k}D_k$
	with gradient-block diagonal $v_g$ and Hessian-block diagonal $v_H$.
	Within the block form \eqref{eq:bup-W-block}, we set
	\[
		W_{g,k}:=\mathrm{diag}\!\bigl(\clip_{[w_{\min},w_{\max}]}(v_g^{-1})\bigr),
		\qquad
		W_{H,k}:=\mathrm{diag}\!\bigl(\clip_{[w_{\min},w_{\max}]}(v_H^{-1})\bigr),
	\]
	with the constant-block precision $w_0\in[w_{\min},w_{\max}]$.
	High posterior uncertainty thus translates to low precision, so poorly predicted coefficient blocks have little influence on the MAP completion.
	For the convergence analysis, this construction is the canonical Bayesian one in our framework: the
	prior mean is taken from GP posterior derivatives, and the precision is
	formed from the corresponding diagonal posterior variances.
	A bridge from GP accuracy to
	Assumption~\ref{asmp:conv-prior} is given in
	Appendix~\ref{app:gp-to-A6}.
	The per-iteration cost is
	$\mathcal{O}(N_{\mathrm{pool}}^3 + n^2 N_{\mathrm{pool}}^2)$.
	Full derivation details are collected in
	Appendix~\ref{app:gp-details}.

\begin{remark}[Sensitivity of the clipping bounds]\label{rem:clip-sensitivity}
The bounds $w_{\min}$ and $w_{\max}$ enter the convergence theory
through the coefficient-error constant
$\kappa_e$ (see~\eqref{eq:conv-ke}): $\kappa_e$ scales with
$1/\sqrt{w_{\min}}$, so a very small $w_{\min}$ inflates the
fully-linear constants.
The scaling in \eqref{eq:chat-def}
makes the diagonal entries comparable across different
$\Delta_k$, a single pair $(w_{\min},w_{\max})$ is typically
effective across the entire run.
Sensitivity of the final performance to
$w_{\max}/w_{\min}$ is examined in the ablation study
(Section~\ref{subsec:exp-results}).
\end{remark}

	\subsubsection{Accepted-model prior}\label{subsec:bup-accepted-map}

	This construction stays entirely within the trust-region loop. It reuses
	the most recently accepted MAP model as prior information for the next
	completion step. We call this the \emph{accepted-model prior}. Since the
	transported model is itself obtained from a previous MAP completion, this
	construction may also be viewed as recursive accepted-MAP completion. It is
	much cheaper than the GP construction and fits the
	$(\widehat c_k^\pi,\widehat W_k)$ interface
	(Remark~\ref{rem:surrogate-agnostic}).

	For the accepted-model prior mean, let $k'$ denote the most recent iteration at which a MAP completion
	was accepted, meaning that the completed model passed the
	acceptability checks and was used in the trust-region subproblem.
	The accepted coefficient vector $\widehat c_{k'}$ is centered at
	$x_{k'}$ and scaled by $\Delta_{k'}$, so it must be transported
	before it can be used at iteration~$k$.
	Define the affine transport
	\begin{equation}\label{eq:bup-transport-map}
		\mathcal T_{k'\to k}:\mathbb{R}^q\to\mathbb{R}^q
	\end{equation}
	as follows: recover the unscaled quadratic model from the input
	coefficient vector using \eqref{eq:bup-mapback} at scale
	$\Delta_{k'}$; translate that quadratic model from center $x_{k'}$ to
	center $x_k$; rescale the gradient and Hessian blocks by
	$\Delta_k$; and set the constant block to the current center value
	$f(x_k)$.
	For an accepted model with gradient $g_{k'}$ and Hessian $H_{k'}$, this
	transport gives
	\begin{equation}\label{eq:bup-accepted-map-shift}
		g_k^\pi = g_{k'} + H_{k'}(x_k - x_{k'}),
		\qquad
		H_k^\pi = H_{k'},
	\end{equation}
	followed by
	\[
		\widehat c_k^{\pi,(g)} = \Delta_k\,g_k^\pi,
		\qquad
		\widehat c_k^{\pi,(H)} = \Delta_k^2\,\mathrm{vech}(H_k^\pi),
		\qquad
		\widehat c_k^{\pi,(0)} = f(x_k).
	\]
	The accepted-model prior is therefore
	\begin{equation}\label{eq:bup-accepted-map-prior}
		\widehat c_k^\pi := \mathcal T_{k'\to k}(\widehat c_{k'}).
	\end{equation}
	We denote by $\mathcal L_{k'\to k}$ the linear difference operator
	associated with this affine transport, obtained by applying the corresponding
	unscaling, shift, and rescaling steps to coefficient
	differences with zero constant reset.
	At the first invocation, when no accepted model is yet available, the
	implementation may use a zero prior, giving a minimum-norm
	completion in the precision metric. With the weak-prior precision family
	\eqref{eq:bup-W-family}, this reduces to the minimum-Frobenius Hessian
	completion. Subsequent theoretical statements use the transported
	accepted-model prior defined above.

	For this recursive construction, the precision is specified directly in
	the block form \eqref{eq:bup-W-block}.
	The gradient block is taken to be uniform,
	$W_{g,k}=w_{\mathrm{base}}I$, while the Hessian block is chosen to be
	diagonal and distance-weighted.
	If the $\ell$-th component of $\mathrm{vech}(H)$ corresponds to the
	upper-triangular entry $(i_u(\ell),j_u(\ell))$, we set
	\begin{equation}\label{eq:bup-accepted-map-W}
		(W_{H,k})_{\ell\ell}
		= \clip_{[w_{\min},w_{\max}]}\!\Big(
		  w_{\mathrm{base}} \cdot \exp\!\big(-\alpha_d\,|i_u(\ell) - j_u(\ell)|\big)
		\Big),
		\qquad \alpha_d > 0,
	\end{equation}
	where $w_{\mathrm{base}}$ is a baseline precision and $\alpha_d$ is a
	decay-rate parameter.
	This design assigns weaker prior pull to Hessian entries coupling more
	distant variable pairs, reflecting lower confidence in long-range
	curvature transfer.
	Since $W_{H,k}$ is diagonal and each entry is clipped to
	$[w_{\min}, w_{\max}]$, the spectral bounds
	\eqref{eq:bup-W-bounds} hold automatically.
	In the default BUP-NEWUOA implementation, local weighted least-squares
	(WLS) curvature statistics are used to calibrate the scale parameter $h_s$
	entering $w_{\mathrm{base}}$; the accepted-model prior mean, the
	center-shift rule, and the clipped block-structured precision remain
	the abstract objects used in the theory.
	In all experiments, we use $\alpha_d = 1.5$ and
	$w_{\mathrm{base}} = h_s \cdot 10^{-6}$, where $h_s$ is a Hessian
	prior scale parameter.  The implementation name
	$\lambda_{\mathrm{decay}}$ corresponds to $\alpha_d$ in
	\eqref{eq:bup-accepted-map-W}.

	In the convergence analysis, this construction is treated as a recursive prior mechanism inside the trust-region loop.
	Under the step-size and model-accuracy conditions stated in
	Lemma~\ref{lem:accepted-map-prior-acc}, the resulting prior satisfies the
	required prior-accuracy condition.

	Storing and updating the accepted-model prior requires $\mathcal{O}(q)$
	memory and $\mathcal{O}(n^2)$ arithmetic for the center-shift
	adjustment \eqref{eq:bup-accepted-map-shift}, compared to
	$\mathcal{O}(n^2 N_{\mathrm{pool}}^2)$ for the GP derivative posterior.
	For $n = 20$ and $N_{\mathrm{pool}} = 100$, this reduces the per-iteration
	overhead by roughly four orders of magnitude.

The two constructions above share the same estimator: once
$(\widehat c_k^\pi,\widehat W_k)$ is available, the model is obtained from
\eqref{eq:bup-hard-closed} and mapped back by \eqref{eq:bup-mapback}.  Thus
the convergence analysis depends on the prior source only through the
spectral bounds on $\widehat W_k$ and the prior-accuracy condition in
Assumption~\ref{asmp:conv-prior}.

\begin{remark}[Generic prior abstraction]\label{rem:surrogate-agnostic}
The default implementation uses the accepted-model prior.  A GP surrogate
is an alternative when the fitting cost is acceptable, and other local
surrogates may be used if they provide a coefficient vector and a positive
definite precision satisfying the assumptions above.
\end{remark}

	\subsection{Geometry of the completion problem}\label{sec:repair}\label{subsec:framework-geometry}
We now formalize the geometry condition previewed in Section~\ref{subsec:geometry-bottleneck}.
The completion rule determines the geometry test: we certify that the
current interpolation set is stable for \eqref{eq:bup-completion} and repair
it only when the certificate fails.  Geometry checks use point locations and
$\widehat W_k$; a function value is queried only after a repair point has
been selected.
	
	\subsubsection{The geometry matrix}\label{subsec:repair-map-geom}

	Let
	\[
	Y=\{y^{(0)},y^{(1)},\ldots,y^{(m)}\}
	\subset B(x_k,\Dgeo{k}),\qquad y^{(0)}=x_k.
	\]
	For this set, define the scaled displacements
	\[
	u^{(i)} := \frac{y^{(i)}-x_k}{\Delta_k},\qquad
	\|u^{(i)}\|_2\le \ctrim,\quad i=0,1,\ldots,m,
	\]
	and the corresponding scaled design matrix
	\begin{equation}\label{eq:repair-AhatY}
		\widehat A(Y)\in\mathbb{R}^{(m+1)\times q},\qquad
		(\widehat A(Y))_{i,:}:=\widehat\phi(u^{(i)})^\top,
	\end{equation}
	where $\widehat\phi(\cdot)$ is the scaled quadratic feature map in \eqref{eq:phihat} and
	$q=\frac{(n+1)(n+2)}{2}$.

Let $\widehat W_k\succ 0$ be the precision used at iteration $k$.  The
matrix used in the geometry test is
	\begin{equation}\label{eq:repair-Mhat}
		\widehat M_k(Y)
		\ :=\
		\widehat A(Y)\,\widehat W_k^{-1}\,\widehat A(Y)^\top
		\ \in\ \mathbb{R}^{(m+1)\times(m+1)}.
	\end{equation}
	We refer to \eqref{eq:repair-Mhat} as the \emph{MAP geometry matrix}.
	In the hard-MAP update \eqref{eq:bup-hard-closed}, $\widehat M_k(Y_k)$ is the matrix inverted in the Schur complement
	system \eqref{eq:bup-normal-eq}. In the soft-MAP update \eqref{eq:bup-soft-closed-2}, the solve involves
	$\widehat M_k(Y_k)+R_k$. The geometry certificate remains based on $\widehat M_k(Y_k)$, which controls the
	conditioning of the MAP completion operator and the constants in the fully-linear analysis in
	Section~\ref{sec:conv}.

We impose a lower bound on the smallest eigenvalue of this matrix.

	\begin{definition}[MAP-poisedness]\label{def:map-poised}
		Fix a threshold $\mu_M>0$. A set $Y\subset B(x_k,\Dgeo{k})$ is called \emph{$\mu_M$-MAP-poised at iteration $k$} if
		\begin{equation}\label{eq:repair-map-poised}
			\lambda_{\min}\big(\widehat M_k(Y)\big)\ \ge\ \mu_M.
		\end{equation}
	\end{definition}
	
	\begin{remark}[Geometry radius]\label{rem:ctrim}
		The interpolation set extends to a geometry radius
		$\Dgeo{k}=\ctrim\Delta_k$ with $\ctrim\ge 1$ fixed (Section~\ref{subsec:prelim-tr}).
		This is the standard setup in model-based trust-region methods for DFO~\cite[Ch.~10]{conn2009introduction}.
		In the fully-linear analysis (Lemma~\ref{lem:conv-FL-BUP}), $\ctrim>1$ increases the
		Taylor-remainder constants by at most $\ctrim^2$, without affecting the convergence
		order or the structure of the proofs.
		In practice we set $\ctrim=1.5$; the fallback set
		(Section~\ref{subsec:repair-fallback}) uses directions of length~$\Delta_k\le\Dgeo{k}$
		and thus always lies within $B(x_k,\Dgeo{k})$.
	\end{remark}

The next lemma records the elementary rank relation behind this condition.

	\begin{lemma}[MAP-poisedness and nonsingularity]\label{lem:repair-rank-spd}
		Let $\widehat W_k\succ 0$ and define $\widehat M_k(Y)$ by \eqref{eq:repair-Mhat}. Set
		\[
		B_k(Y):=\widehat A(Y)\,\widehat W_k^{-1/2}\in\mathbb{R}^{(m+1)\times q},
		\qquad\text{so that}\qquad
		\widehat M_k(Y)=B_k(Y)B_k(Y)^\top.
		\]
		Then $\widehat M_k(Y)\succ 0$ if and only if $\widehat A(Y)$ has full row rank, and
		\[
		\lambda_{\min}\big(\widehat M_k(Y)\big)=\sigma_{\min}\big(B_k(Y)\big)^2,
		\]
		so a uniform lower bound on $\lambda_{\min}(\widehat M_k(Y))$ is equivalent to a uniform lower bound on
		$\sigma_{\min}(B_k(Y))$, i.e., a scale-invariant conditioning bound for the MAP completion.
	\end{lemma}
	
	\begin{proof}
		Since $\widehat W_k^{-1/2}$ is invertible, $\widehat A(Y)$ has full row rank if and only if $B_k(Y)$ has full row rank.
		The identity $\widehat M_k(Y)=B_k(Y)B_k(Y)^\top$ implies $\widehat M_k(Y)\succ 0$ if and only if $B_k(Y)$ has full row rank.
		Finally, the eigenvalues of $B_k(Y)B_k(Y)^\top$ are the squared singular values of $B_k(Y)$, which gives the stated
		identity for $\lambda_{\min}$.
	\end{proof}

For diagnostics one may also monitor the condition number
\begin{equation}\label{eq:repair-cond}
	\kappa_{\mathrm{geo}}(Y):=\frac{\lambda_{\max}(\widehat M_k(Y))}{\lambda_{\min}(\widehat M_k(Y))},
\end{equation}
as a diagnostic indicator.
Since $\lambda_{\max}$ is controlled by the norm of the feature
vectors, $\kappa_{\mathrm{geo}}$ is primarily driven by
$\lambda_{\min}$ approaching zero; it is useful for logging
and for setting adaptive thresholds in practice.
The convergence theory requires only the lower-spectral
condition \eqref{eq:repair-map-poised};
$\kappa_{\mathrm{geo}}$ is used only as a diagnostic quantity.

When $\widehat W_k=I$, we have $\widehat M_k(Y)=\widehat A(Y)\widehat A(Y)^\top$ and
	\eqref{eq:repair-map-poised} reduces to $\sigma_{\min}(\widehat A(Y))^2\ge \mu_M$, i.e., a uniform bound on
	$\|\widehat A(Y)^\dagger\|_2$. This is closely aligned with the classical geometry conditions (e.g., $\Lambda$-poisedness)
	used to guarantee fully-linear model accuracy in model-based trust-region methods for DFO \cite{conn2009introduction}.
	
\begin{proposition}[Spectral link to $\Lambda$-poisedness]\label{prop:lambda-link}
	Let $\widehat W_k$ satisfy the spectral bounds
	\eqref{eq:bup-W-bounds}.  Then
	\begin{equation}\label{eq:lambda-sandwich}
	\frac{1}{w_{\max}}\widehat A(Y)\widehat A(Y)^\top
	\;\preceq\;
	\widehat M_k(Y)
	\;\preceq\;
	\frac{1}{w_{\min}}\widehat A(Y)\widehat A(Y)^\top.
	\end{equation}
	Consequently, $\mu_M$-MAP-poisedness implies
	$\sigma_{\min}(\widehat A(Y))\ge \sqrt{w_{\min}\mu_M}$.
	Conversely, if a classical $\Lambda$-poisedness bound yields
	$\sigma_{\min}(\widehat A(Y)) \ge 1/(\Lambda B_\phi)$ for
	$B_\phi := \sup_{\|u\|\le 1}\|\widehat\phi(u)\|$, then $Y$ is
	$\mu_M$-MAP-poised with
	$\mu_M = 1/(w_{\max}\,\Lambda^2\,B_\phi^2)$.
\end{proposition}

\begin{proof}
	The sandwich \eqref{eq:lambda-sandwich} follows from the definition
	$\widehat M_k(Y)=\widehat A(Y)\widehat W_k^{-1}\widehat A(Y)^\top$
	and the spectral bounds
	$w_{\min}I\preceq\widehat W_k\preceq w_{\max}I$.
	The forward implication follows from the upper sandwich inequality
	\[
	\widehat M_k(Y)\preceq
	w_{\min}^{-1}\widehat A(Y)\widehat A(Y)^\top .
	\]
	Indeed, if $\lambda_{\min}(\widehat M_k(Y))\ge\mu_M$, then
	\[
	\widehat A(Y)\widehat A(Y)^\top
	\succeq w_{\min}\widehat M_k(Y)
	\succeq w_{\min}\mu_M I,
	\]
	and hence
	$\sigma_{\min}(\widehat A(Y))\ge\sqrt{w_{\min}\mu_M}$.
	For the converse, the left inequality in \eqref{eq:lambda-sandwich} yields
	\[
		\lambda_{\min}(\widehat M_k(Y))
		\;\ge\;
		\frac{\sigma_{\min}^2(\widehat A(Y))}{w_{\max}}
		\;\ge\;
		\frac{1}{w_{\max}\,\Lambda^2\,B_\phi^2},
	\]
	which is the claimed MAP-poisedness threshold.
\end{proof}

\begin{remark}
	\cref{prop:lambda-link} shows that MAP-poisedness is the spectral
	counterpart of the bounded-Lagrange-polynomial conditions used in
	classical $\Lambda$-poisedness theory
	\cite[Definition~6.2]{conn2009introduction}.
	Our convergence analysis uses MAP-poisedness directly
	(through $\|\widehat M^{-1}\|_2\le 1/\mu_M$), so
	the proof proceeds without an explicit $\Lambda$-constant.
\end{remark}
	
	The trigger and evaluation budget for geometry repair are specified next.
	
	\phantomsection\label{subsec:repair-when}
	
	At iteration $k$, before forming the MAP model, we check whether the current interpolation set $Y_k$ is MAP-poised:
	\begin{equation}\label{eq:repair-check}
		\text{if }\ \lambda_{\min}\big(\widehat M_k(Y_k)\big) < \mu_M,\ \text{ then invoke geometry repair.}
	\end{equation}
	We allow at most $T_{\mathrm{try}}\ge 1$ incremental repair attempts. If the MAP-poisedness test still fails within
	this budget, we trigger a reset mechanism (Section~\ref{subsec:repair-fallback}). This explicit budget is
	used in the evaluation accounting in Section~\ref{subsec:prelim-evals} and in the evaluation-complexity analysis in
	Section~\ref{sec:conv}.
	
	The convergence analysis in Section~\ref{sec:conv} is stated for the
\emph{baseline certified policy} in which the check
\eqref{eq:repair-check} is applied before every model construction, so
that every trust-region step is computed from a MAP-poised set. The
implementation also supports an \emph{on-reject} schedule: the full
geometry check is performed after rejected trial steps ($\rho_k < \eta_1$),
whereas accepted steps reuse the updated set before the next certification.
This amortized schedule is an implementation variant designed to reduce
repeated eigenvalue checks; the baseline theory analyzes the certified
policy.
	
	\subsubsection{Incremental geometry repair}\label{subsec:repair-acq}
	
	The incremental repair step uses a ranking function only within a geometry-feasible
	candidate subset. Geometry feasibility is enforced by a replace-one test based on
	\eqref{eq:repair-map-poised}.

	Candidate pool.
	We generate a pool of $N_{\mathrm{cand}}$ candidate points
	\begin{equation}\label{eq:repair-candpool}
		C_k\ \subset\ B(x_k,\Delta_k),
		\qquad |C_k|=N_{\mathrm{cand}},
	\end{equation}
by sampling uniformly in the $n$-ball $B(x_k,\Delta_k)$.
No function evaluations are spent at this stage; the theory requires only that candidates are available and
that the fallback reset (Section~\ref{subsec:repair-fallback}) provides a finite safety mechanism.
Alternative candidate strategies are discussed in Appendix~\ref{app:repair-incr}.

	Replace-one feasibility test.
	We keep the interpolation set size fixed at $m+1$. Let $\mathcal{J}:=\{1,2,\ldots,m\}$ index the droppable
	non-center points ($y_k^{(0)}=x_k$ is never dropped).
	For each candidate $y\in C_k$ and each drop index $j\in\mathcal{J}$, form the replacement set
	\[
	Y_k(y,j):=(Y_k\cup\{y\})\setminus\{y_k^{(j)}\},
	\]
	which has size $m+1$, and compute $\lambda_{\min}(\widehat M_k(Y_k(y,j)))$.
	The best drop index for a given candidate is
	\begin{equation}\label{eq:repair-jstar}
		j^\star(y)\ :=\ \arg\max_{j\in\mathcal{J}} \lambda_{\min}\big(\widehat M_k(Y_k(y,j))\big),
	\end{equation}
	with a fixed tie-breaking rule. We declare $y$ \emph{geometry-feasible} if the best replacement achieves
	the threshold:
	\begin{equation}\label{eq:repair-feasible}
		y\in C_k^{\mathrm{geo}}
		\quad:\Longleftrightarrow\quad
		\lambda_{\min}\big(\widehat M_k(Y_k(y,j^\star(y)))\big)\ \ge\ \mu_M.
	\end{equation}
	For each candidate, the scan over all $m$ possible drop positions identifies the replacement with the
	largest $\lambda_{\min}$ and then checks whether it meets the MAP-poisedness threshold.
	The per-candidate cost of this scan is reduced to $\mathcal{O}(m^2)$ by
	a Cholesky rank-one update (Remark~\ref{rem:repair-cost}).

	Among all geometry-feasible candidates $C_k^{\mathrm{geo}}$,
	we select the point that maximizes the resulting minimum eigenvalue:
	\begin{equation}\label{eq:repair-acqselect}
		y_k^{\mathrm{new}}\ :=\ \arg\max_{y\in C_k^{\mathrm{geo}}}\ \lambda_{\min}\!\bigl(\widehat M_k(Y_k(y,j^\star(y)))\bigr).
	\end{equation}
	This directly optimizes the MAP-poisedness quality of the repaired set,
	aligning the repair selection with the geometry certification goal.
	The safety guarantee rests entirely on \eqref{eq:repair-feasible} and the fallback reset;
	the $\lambda_{\min}$-ranking affects the quality of the repair choice while preserving its validity.
	Since the maximization is restricted to $C_k^{\mathrm{geo}}$, the selected point admits a replacement that preserves
	\eqref{eq:repair-map-poised} by construction.

	Reuse of previously evaluated points and candidate-based repair.
	Before spending evaluation budget, the repair mechanism first scans
	$\mathcal{D}_k$ for previously evaluated nearby points satisfying
	\eqref{eq:repair-feasible}, swapping feasible previously evaluated points at
	zero evaluation cost (details in Appendix~\ref{app:repair-incr}).
	If the MAP-poisedness test still fails, each subsequent incremental repair attempt consumes exactly one new function evaluation:
	\begin{enumerate}[leftmargin=2em]
		\item Evaluate $f(y_k^{\mathrm{new}})$ and add it to $\mathcal{D}_k$.
		\item Update the interpolation set:
		\begin{equation}\label{eq:repair-update}
			Y_k \leftarrow Y_k\big(y_k^{\mathrm{new}},\, j^\star(y_k^{\mathrm{new}})\big).
		\end{equation}
		\item Re-check \eqref{eq:repair-check}; if still violated and the attempt counter is below $T_{\mathrm{try}}$,
		generate a fresh candidate pool, filter, select, and repeat.
	\end{enumerate}
	The per-candidate cost of scanning all $m$ drop positions is
	$\mathcal{O}(m^2)$ via Cholesky rank-one updates
	(Appendix~\ref{app:repair-incr}).
	
	\subsubsection{The fallback set}\label{subsec:repair-fallback}
	
Incremental repair may fail if the attempt budget $T_{\mathrm{try}}$ is exhausted (see Appendix~\ref{app:repair-fallback-details} for the counting convention).
To guarantee finite restoration of \eqref{eq:repair-map-poised}, we include a reset mechanism.

	Throughout the baseline algorithm we use the standard choice
	$m=2n$ from \eqref{eq:algo-mchoice}, so the fallback set has
	size $m+1=2n+1$.

	Before resorting to the coordinate-direction fallback below,
	we attempt to assemble a MAP-poised set purely from nearby previously evaluated
	points in $\mathcal{D}_k$, ranked by a surrogate-variance
	score (Appendix~\ref{app:repair-fallback-details}).
	If the resulting set satisfies $\lambda_{\min}(\widehat M_k) \ge \mu_M$,
	the fallback is complete with zero new evaluations;
	otherwise the coordinate-direction reset below is invoked.

	Define the fallback interpolation set
	\begin{equation}\label{eq:repair-fallback-set}
		Y_k^{\mathrm{fb}}
		:=
		\Big\{x_k;\ x_k\pm \Delta_k e_1,\ \ldots,\ x_k\pm \Delta_k e_n\Big\}.
	\end{equation}
	If some points in \eqref{eq:repair-fallback-set} were evaluated previously, their function values can be reused;
	otherwise they are evaluated upon reset.

	If neither reconstruction from previously evaluated points nor incremental repair restores \eqref{eq:repair-map-poised}, we set
	\begin{equation}\label{eq:repair-reset-rule}
		Y_k \leftarrow Y_k^{\mathrm{fb}},
	\end{equation}
	and evaluate any missing function values for points in $Y_k^{\mathrm{fb}}$, adding all newly evaluated pairs to $\mathcal{D}_k$.

	In Section~\ref{sec:conv} (with details in the appendix), we prove that there exists an explicit constant $\mu_0>0$
	such that, whenever $\widehat W_k$ satisfies \eqref{eq:bup-W-bounds}, the fallback set satisfies
	\begin{equation}\label{eq:repair-mu0}
		\lambda_{\min}\big(\widehat M_k(Y_k^{\mathrm{fb}})\big)\ \ge\ \mu_0.
	\end{equation}
	Consequently, choosing $\mu_M\le \mu_0$ makes the reset step \eqref{eq:repair-reset-rule} successful.

	Each incremental repair attempt evaluates exactly one point.
	After at most $T_{\mathrm{try}}$ such attempts, if the MAP-poisedness test still fails, the fallback reset
	\eqref{eq:repair-reset-rule} evaluates at most $m=2n$ new points (the center value is already available):
	\begin{equation}\label{eq:repair-eval-bound-reset}
		N_k^{\mathrm{rep,repair}}\ \le\ T_{\mathrm{try}} + 2n\qquad \text{per repair invocation.}
	\end{equation}
	Since the warm-start set update incurs no new evaluations, the base cost per iteration (single criticality-loop pass)
	equals the repair cost alone:
	$N_k^{\mathrm{rep,base}}\le T_{\mathrm{try}}+2n$ (Lemma~\ref{lem:conv-repair-finite}).
Each additional criticality shrink adds at most $T_{\mathrm{try}}+2n$ evaluations to $N_k^{\mathrm{rep,crit}}$
(the warm-start set update after the radius shrink is evaluation-free; the subsequent repair pass costs at most $T_{\mathrm{try}}+2n$);
the aggregate $\sum_k N_k^{\mathrm{rep,crit}}$ is bounded in Theorem~\ref{thm:conv-complexity}.
	
	\phantomsection\label{subsec:repair-summary}
	The role of the repair mechanism is limited and explicit.  Before each
	model is formed, the current set is tested by \eqref{eq:repair-check}.  A
	failed test first triggers replacements using previously evaluated nearby
	points, then at most $T_{\mathrm{try}}$ new candidate evaluations, and
	finally the coordinate-direction fallback.  Thus every model used for a
	trust-region step is built from a certified interpolation set, while the
	maximum number of non-trial evaluations in one repair pass is bounded by
	$T_{\mathrm{try}}+2n$.  The ranking rule affects only which feasible repair
	point is chosen; the convergence proof uses only feasibility, the attempt
	budget, and the fallback guarantee.

With the completion and geometry ingredients in place, we now assemble
them into a complete trust-region algorithm.
	
	\phantomsection\label{sec:algo}
	\subsection{The BUP-TR algorithm}\label{subsec:framework-algo}
	
	The completion and geometry mechanisms above now give an
	implementable trust-region algorithm.
	We state the baseline hard-MAP variant for noiseless evaluations; the
	soft-MAP variant uses this trust-region skeleton and replaces the model
	construction step with a ridge-type completion.
	Figure~\ref{fig:algo-flow} gives a high-level overview of one iteration.

\begin{figure}[t]
\centering
\resizebox{0.94\textwidth}{!}{%
\begin{tikzpicture}[
    node distance=0.55cm and 0.6cm,
    box/.style={draw, rounded corners=3pt, minimum height=0.7cm,
                minimum width=2.1cm, align=center, font=\small},
    decision/.style={draw, diamond, aspect=2.2, inner sep=1pt,
                     align=center, font=\small},
    arr/.style={-Stealth, thick},
    every node/.style={font=\small}
  ]
  \node[box,fill=blue!8] (gp)
    {Prior source\\[-1pt]$(\widehat c_k^\pi,\widehat W_k)$};
  \node[box,fill=blue!8, right=of gp] (geom)
    {Geometry\\[-1pt]check};
  \node[decision, right=of geom] (ok)
    {MAP-\\poised?};
  \node[box,fill=orange!12, above=0.45cm of ok] (repair)
    {Repair\\[-1pt](candidate + fallback)};
  \node[box,fill=blue!8, right=of ok] (comp)
    {MAP\\[-1pt]completion};
  \node[box,fill=blue!8, right=of comp] (step)
    {TR step\\[-1pt]$s_k$};
  \node[decision, right=of step] (rho)
    {$\rho_k\!\ge\!\eta_1$?};
  \node[box,fill=green!10, above=0.45cm of rho] (accept)
    {Accept;\\[-1pt]expand $\Delta$};
  \node[box,fill=red!8, below=0.45cm of rho] (reject)
    {Reject;\\[-1pt]shrink $\Delta$};

  \draw[arr] (gp) -- (geom);
  \draw[arr] (geom) -- (ok);
  \draw[arr] (ok) -- node[above,font=\scriptsize]{yes} (comp);
  \draw[arr] (ok) -- node[left,font=\scriptsize]{no} (repair);
  \draw[arr] (repair.west) -| ([xshift=-3pt]geom.north);
  \draw[arr] (comp) -- (step);
  \draw[arr] (step) -- (rho);
  \draw[arr] (rho) -- node[left,font=\scriptsize]{yes} (accept);
  \draw[arr] (rho) -- node[left,font=\scriptsize]{no} (reject);

  \coordinate (loop) at ($(accept.east)+(0.3,0)$);
  \draw[arr] (accept.east) -- (loop) |- ([yshift=0.55cm]gp.north) -| (gp.north);
  \draw[arr] (reject.east) -| (loop);
\end{tikzpicture}
}
\caption{One iteration of BUP-TR.
The prior source provides $(\widehat c_k^\pi,\widehat W_k)$;
geometry is certified via $\widehat M_k(Y)$ and repaired if needed;
MAP completion produces the local quadratic model used by the trust-region
step; under the assumptions of Section~\ref{sec:conv}, these models
satisfy the required fully-linear bounds; a standard TR
accept/reject step updates the iterate and radius.
The criticality loop (Repeat--Until block of Algorithm~\ref{alg:buptr})
is omitted for clarity; if the model gradient is too small, the
radius is reduced and the cycle restarts from the geometry check.}
\label{fig:algo-flow}
\end{figure}
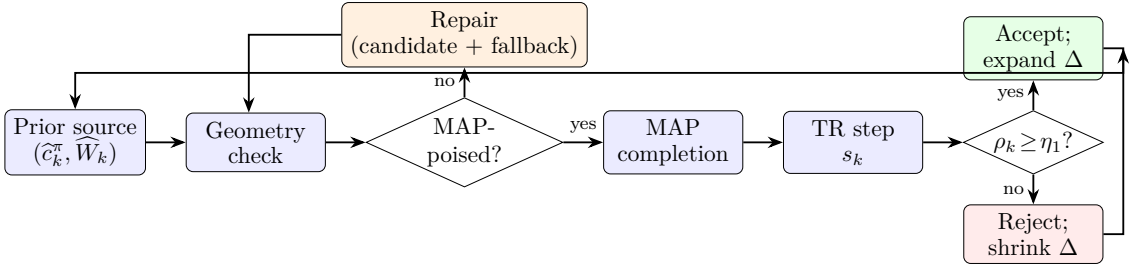
	
	\subsubsection{Design choices}\label{subsec:algo-design}
	
	We fix the interpolation set size to
	\begin{equation}\label{eq:algo-mchoice}
		m:=2n,\qquad |Y_k|=m+1=2n+1,
	\end{equation}
	which is the standard ``linear-plus-symmetric'' regime in model-based trust-region methods for DFO.
	
	The algorithm uses the trust-region parameters $(\eta_1,\eta_2,\gamma_{\mathrm{dec}},\gamma_{\mathrm{inc}})$ as in
	Section~\ref{subsec:prelim-accept}, and maintains MAP-poisedness with a threshold $\mu_M>0$ and an incremental repair budget
	$T_{\mathrm{try}}\ge 1$ as in Subsection~\ref{sec:repair}. The candidate pool size for incremental repair is
	$N_{\mathrm{cand}}$.
	
	Prior information enters through the prior mean $\widehat c_k^\pi$ and the precision $\widehat W_k$, constructed as
	in Sections~\ref{subsec:bup-prior-mean}--\ref{subsec:bup-accepted-map} and clipped to satisfy the uniform bounds
	\eqref{eq:bup-W-bounds}.
	
	\subsubsection{Algorithm statement}\label{subsec:algo-pseudocode}
	
	Algorithm~\ref{alg:buptr} summarizes the baseline hard-MAP BUP-TR method.
	The model-construction step invokes a \textsc{PriorSource} to obtain
	$(\widehat c_k^\pi,\widehat W_k)$ satisfying the spectral bounds
	\eqref{eq:bup-W-bounds}, while geometry is maintained by the explicit
	subroutine \textsc{GeometryRepair}
	(Subsection~\ref{sec:repair}).
	Concrete prior-source examples are recalled in
	Remark~\ref{rem:prior-instances}.
	In implementation, we additionally terminate when the best observed value stagnates over
$W_f$ consecutive iterations with relative variation below a tolerance~$f_{\mathrm{tol}}$;
this practical stopping rule is omitted from the theory.
Theorems~\ref{thm:conv-global}--\ref{thm:conv-complexity}
analyze the generated sequence and bound the first index at which an
iterate satisfies $\|\nabla f(x_k)\|\le\varepsilon$.
	
\begin{algorithm}[t]
	\caption{BUP-TR (hard-MAP version, generic prior source)}\label{alg:buptr}
	\begin{algorithmic}[1]\scriptsize
			\Require Initial center $x_0\in\R^n$, radius $\Delta_0\in(0,\Delta_{\max}]$; set size $m:=2n$.
			\Require TR parameters $(\eta_1,\eta_2,\gamma_{\mathrm{dec}},\gamma_{\mathrm{inc}})$.
			\Require Geometry parameters $(\mu_M,T_{\mathrm{try}},N_{\mathrm{cand}})$.
			\Require Precision bounds $(w_{\min},w_{\max})$; criticality parameter $\kappa_\Delta>0$; termination threshold $\Delta_{\min}\ge 0$.
			\Require A \textsc{PriorSource} that, given $(\mathcal{D}_k, x_k, \Delta_k)$, returns $(\widehat c_k^\pi, \widehat W_k)$ satisfying \eqref{eq:bup-W-bounds}.
			\State Initialize $Y_0\leftarrow \{x_0;\ x_0\pm \Delta_0 e_i\}_{i=1}^n$; evaluate $f$ on $Y_0$; set $\mathcal{D}_0:=\{(y,f(y)):y\in Y_0\}$.
			\For{$k=0,1,2,\ldots$}
			\Statex \quad Model construction and criticality loop
			\Repeat
			\State $(\widehat c_k^\pi,\, \widehat W_k) \leftarrow \textsc{PriorSource}(\mathcal{D}_k, x_k, \Delta_k)$.
			\State Geometry: $Y_k \leftarrow \textsc{GeometryRepair}(Y_k,x_k,\Delta_k,\widehat W_k,\mu_M,T_{\mathrm{try}},N_{\mathrm{cand}})$:
			\Statex \qquad sample $N_{\mathrm{cand}}$ candidates uniformly in $B(x_k,\Delta_k)$;
			\Statex \qquad for up to $T_{\mathrm{try}}$ attempts: filter by replace-one test \eqref{eq:repair-feasible},
			\Statex \qquad\quad select the feasible candidate maximizing \eqref{eq:repair-acqselect}, evaluate, and replace \eqref{eq:repair-update};
			\Statex \qquad if still $\lambda_{\min}(\widehat M_k(Y_k))<\mu_M$: fallback reset $Y_k\leftarrow Y_k^{\mathrm{fb}}$ \eqref{eq:repair-reset-rule},
			\Statex \qquad\quad evaluate any missing points in $Y_k^{\mathrm{fb}}$ and add them to $\mathcal{D}_k$.
			\State Form $\widehat A_k=\widehat A(Y_k)$, $b_k=[f(y_k^{(0)}),\ldots,f(y_k^{(m)})]^\top$.
			\State Hard-MAP: Compute $\widehat c_k$ by \eqref{eq:bup-hard-closed}; map to $(c_{k,0},c_{k,1},C_{k,2})$ via \eqref{eq:bup-mapback}; define $m_k$.
			\If{$\|g_k\|_2\le \kappa_\Delta \Delta_k$ \ \ (with $g_k=\nabla m_k(0)$)}
			\State $\Delta_k\leftarrow\gamma_{\mathrm{dec}}\Delta_k$.
		\State Warm-start $Y_k$: retain points within $B(x_k,\ctrim\Delta_k)$; fill from $\mathcal{D}_k$ sorted by distance; no new evaluations.
			\EndIf
			\Until{$\|g_k\|_2> \kappa_\Delta \Delta_k$ or $\Delta_k\le\Delta_{\min}$ (terminate with approximate stationarity)}
			\Statex \quad Trust-region step
			\State Approximately solve \eqref{eq:tr-subprob} to obtain $s_k$ satisfying \eqref{eq:cauchy-dec}.
			\State $\mathrm{pred}_k\leftarrow m_k(0)-m_k(s_k)$.
			\If{$\mathrm{pred}_k\le 0$} \Comment{no trial evaluation}
			\State $\rho_k\leftarrow -\infty$;\ $x_{k+1}\leftarrow x_k$;\ $\Delta_{k+1}\leftarrow\gamma_{\mathrm{dec}}\Delta_k$.
			\Else
			\State Evaluate $f(x_k+s_k)$; $\mathrm{ared}_k\leftarrow f(x_k)-f(x_k+s_k)$; $\rho_k\leftarrow \mathrm{ared}_k/\mathrm{pred}_k$.
			\State $\mathcal{D}_k \leftarrow \mathcal{D}_k \cup \{(x_k+s_k,\,f(x_k+s_k))\}$.
			\If{$\rho_k\ge \eta_1$}
			\State $x_{k+1}\leftarrow x_k+s_k$ \Comment{successful}
			\Else
			\State $x_{k+1}\leftarrow x_k$ \Comment{unsuccessful}
			\EndIf
			\State Update $\Delta_{k+1}$ by the standard TR rule (Section~\ref{subsec:prelim-accept}).
			\EndIf
			\Statex \quad Interpolation set update (no new evaluations)
			\If{$\rho_k\ge \eta_1$}
			\State Form displacement set $S_k^+=\{y-x_{k+1}: y\in Y_k\}$ relative to the new center $x_{k+1}$.
			\State Set $s_0^+=0$ and replace the farthest noncenter displacement by the previous-center displacement.
			\Statex \qquad $j^\dagger\leftarrow\arg\max_{j\ge1}\|s_j^+\|$,
			\Statex \qquad $S_{k+1}\leftarrow(S_k^+\setminus\{s_{j^\dagger}^+\})\cup\{x_k-x_{k+1}\}$.
			\State Convert back to interpolation points $Y_{k+1}\leftarrow\{x_{k+1}+s: s\in S_{k+1}\}$.
			\Else
			\State Keep the center unchanged and set $Y_{k+1}\leftarrow Y_k$.
			\Statex \qquad If the rejected trial point was evaluated, an implementation may insert it only through a geometry-valid replace-one rule; no previous-center displacement is inserted.
			\EndIf
			\State $\mathcal{D}_{k+1}\leftarrow \mathcal{D}_k$.
			\EndFor
		\end{algorithmic}
	\end{algorithm}

\begin{remark}[Examples of prior sources]\label{rem:prior-instances}
Algorithm~\ref{alg:buptr} is agnostic to the choice of prior source: the
convergence theory (Section~\ref{sec:conv}) depends on
\textsc{PriorSource} only through the spectral bounds on
$\widehat W_k$ (Assumption~\ref{asmp:conv-Wbounds}) and the
prior-accuracy condition (Assumption~\ref{asmp:conv-prior}).
For reference, the two concrete prior sources used in this paper are:
\begin{enumerate}[leftmargin=2em,label=(\alph*),nosep]
\item \emph{GP derivative posterior}
(Section~\ref{subsec:bup-prior-mean}):
$\widehat c_k^\pi$ is set via \eqref{eq:bup-prior-gp}
with derivative blocks from a GP fitted to $\mathcal{D}_k$,
and diagonal $\widehat W_k$ by \eqref{eq:bup-W-block}
(see Appendix~\ref{app:gp-details} for the full GP precision derivation).
\item \emph{Accepted-model prior}
(Section~\ref{subsec:bup-accepted-map}, default in BUP-NEWUOA):
$\widehat c_k^\pi$ reuses the accepted MAP coefficients from
the previous iteration with a center shift, and
$\widehat W_k$ follows the accepted-model template of
Section~\ref{subsec:bup-accepted-map}; in the default implementation,
local WLS curvature statistics are used to calibrate its scale,
together with structured Hessian regularization.
No auxiliary surrogate is needed; per-iteration cost is
$\mathcal{O}(n^2)$.
\end{enumerate}
\end{remark}

\begin{remark}[Interpolation set update and information reuse]\label{rem:warmstart}
The inter-iteration set update (Algorithm~\ref{alg:buptr})
replaces a single point---the farthest from the new center---with
the previous center, analogous to the replacement strategy in
NEWUOA~\cite{powell2006newuoa}.
No new function evaluations are triggered; the intra-criticality
set update similarly reuses only points from~$\mathcal{D}_k$.
All new evaluations are confined to geometry repair (at most
$T_{\mathrm{try}}+2n$ per pass, Lemma~\ref{lem:conv-repair-finite})
and the trial step (at most~$1$).
This incremental replacement preserves good interpolation points
across iterations, avoiding the information loss of a full set rebuild.
Information collected in $\mathcal{D}_k$ is reused through
both the prior source and the set update, while
new evaluations are spent only where geometry certification demands
them.
\end{remark}

\begin{remark}[Practical implementation choices in BUP-NEWUOA]\label{rem:impl-enhancements}
BUP-NEWUOA is the concrete implementation used in the experiments.  Its
main design choices are: (i) the accepted-model prior mean described in
Section~\ref{subsec:bup-accepted-map}; (ii) the structured diagonal
Hessian precision \eqref{eq:bup-accepted-map-W}, with default decay
parameter $\lambda_{\mathrm{decay}}=1.5$; and (iii) local WLS curvature
statistics used to calibrate the precision scale.  The convergence
theory covers the baseline hard-MAP algorithm with clipped precision,
MAP-poisedness repair, and the fallback set.  Additional engineering
heuristics used in the code, including restarts, adaptive geometry-check
skipping, radius safeguards, and recovery from previously evaluated points, are reported in
Appendix~\ref{app:repair-details} and Table~\ref{tab:supp-params-bup};
they are outside the scope of Theorems~\ref{thm:conv-global}--\ref{thm:conv-complexity}.
\end{remark}

	\subsubsection{Soft-MAP variant}\label{subsec:algo-soft}
	
	When evaluations are noisy, one may replace the hard-MAP completion step by the soft-MAP estimator
	\eqref{eq:bup-soft-closed-2} or by a constrained soft-MAP that enforces exact center consistency.
	The trust-region skeleton, MAP-poisedness geometry management, and evaluation accounting remain unchanged; only the model
	construction subroutine differs.
	
\begin{remark}[Convergence of the soft-MAP variant]\label{rem:softmap-conv}
	As $R_k \to 0$, the soft-MAP solution \eqref{eq:bup-soft-closed-2}
	converges to the hard-MAP solution \eqref{eq:bup-hard-closed},
	recovering the full convergence theory of Section~\ref{sec:conv}.
	For fixed $R_k>0$ the residual bound in Lemma~\ref{lem:conv-FL-BUP}
	acquires an additive $O(R_k)$ term, so the fully-linear constants
	depend on the noise level.
	Rigorous treatment then requires a probabilistic
	fully-linear framework \cite{bandeira2014convergence}, yielding
	expected or high-probability complexity bounds; the MAP-poisedness
	mechanism is unaffected since it depends only on point geometry.
\end{remark}
	
	\phantomsection\label{subsec:algo-evals}
	
	The evaluation decomposition follows Section~\ref{subsec:prelim-evals}:
	each main iteration incurs at most one trial evaluation, plus
	$N_k^{\mathrm{rep,base}}+N_k^{\mathrm{rep,crit}}$ non-trial evaluations.
	Deterministic bounds on both components are established next.
	
	\section{Convergence and Evaluation Complexity}\label{sec:conv}
	
	We analyze the baseline hard-MAP BUP-TR method under
	Assumptions~\ref{asmp:conv-smooth}--\ref{asmp:conv-crit}. The results
	cover the noiseless hard-MAP setting with the certified geometry policy
	stated in Section~\ref{subsec:framework-geometry}. Under these
	assumptions, Algorithm~\ref{alg:buptr} converges globally and satisfies a
	worst-case evaluation-complexity bound of order
	$\mathcal{O}(\varepsilon^{-2})$.
	
	The analysis has two parts. First, MAP-poisedness, the precision bounds,
	and the prior-accuracy condition imply that the hard-MAP model is fully
	linear. Second, the standard trust-region decrease argument
	for fully-linear models gives global first-order convergence and the
	$\mathcal{O}(\varepsilon^{-2})$ bound. The constants track the quantities
	$(w_{\min},w_{\max},\mu_M,\bar\kappa_\pi)$ explicitly, and the evaluation
	count includes trial evaluations, incremental repair evaluations, and
	fallback evaluations.
	
	\subsection{Assumptions for the analysis}\label{subsec:conv-assumptions}
	
	At iteration $k$, let $x_k\in\R^n$ be the center, $\Delta_k\in(0,\Delta_{\max}]$ the radius, and
	$Y_k=\{y_k^{(0)}=x_k,y_k^{(1)},\ldots,y_k^{(m)}\}\subset B(x_k,\ctrim\Delta_k)$ the interpolation set.
	Define displacements $s_k^{(i)}:=y_k^{(i)}-x_k$ and scaled displacements $u_k^{(i)}:=s_k^{(i)}/\Delta_k$.
	Let $\widehat A_k:=\widehat A(Y_k)$ be the scaled design matrix \eqref{eq:repair-AhatY}.
	
We first record a simple bound for the scaled feature rows.  For
$\|u\|_2\le 1$, the feature vector $\widehat\phi(u)$ in \eqref{eq:phihat}
satisfies
	\begin{equation}\label{eq:conv-Bphi}
		\|\widehat\phi(u)\|_2^2
		=1+\|u\|_2^2+\|\qvec(u)\|_2^2
		\le 1+1+\frac12\|u\|_2^4
		\le \frac52,
	\end{equation}
	hence
	\begin{equation}\label{eq:conv-Bphi-const}
		B_\phi:=\sup_{\|u\|_2\le 1}\|\widehat\phi(u)\|_2 \le \sqrt{\frac52},
		\qquad
		\|\widehat A(Y)\|_2 \le \sqrt{m+1}\,B_\phi.
	\end{equation}
	Since the interpolation set lives in $B(x_k,\ctrim\Delta_k)$, the scaled
	displacements satisfy $\|u_k^{(i)}\|_2\le\ctrim$. Therefore
	\begin{equation}\label{eq:conv-Bphi-geo}
	\begin{aligned}
		\Bgeo
		&:=\sup_{\|u\|_2\le\ctrim}\|\widehat\phi(u)\|_2
		\le \sqrt{1+\ctrim^2+\tfrac12\ctrim^4},\\
		\|\widehat A(Y)\|_2
		&\le \sqrt{m+1}\,\Bgeo.
	\end{aligned}
	\end{equation}
	
	\begin{assumption}[Smoothness and lower boundedness]\label{asmp:conv-smooth}
		The objective $f$ is continuously differentiable on a neighborhood containing all trust regions visited by the algorithm,
		its gradient is Lipschitz with constant $L_g>0$,
		\[
		\|\nabla f(x)-\nabla f(z)\|_2 \le L_g\|x-z\|_2,
		\]
		and $f$ is bounded below by $f_{\inf}\in\R$.
	\end{assumption}
	
	\begin{assumption}[Trust-region parameters]\label{asmp:conv-trparams}
		The trust-region update uses fixed constants $0<\eta_1\le \eta_2<1$ and $0<\gamma_{\mathrm{dec}}<1<\gamma_{\mathrm{inc}}$
		as in Section~\ref{subsec:prelim-accept}.
	\end{assumption}
	
	\begin{assumption}[Subproblem decrease]\label{asmp:conv-cauchy}
		Let $m_k(s)=c_{k,0}+c_{k,1}^\top s+\frac12 s^\top C_{k,2}s$ be the (hard-MAP) BUP model.
		Let $g_k:=\nabla m_k(0)=c_{k,1}$ and $H_k:=\nabla^2 m_k(0)=C_{k,2}$.
		If $g_k\neq 0$, the computed step $s_k$ satisfies $\|s_k\|_2\le \Delta_k$ and
		\begin{equation}\label{eq:conv-cauchy}
			\pred_k:=m_k(0)-m_k(s_k)
			\ \ge\ \frac12\,\|g_k\|_2\min\!\left\{\Delta_k,\ \frac{\|g_k\|_2}{\|H_k\|_2}\right\},
		\end{equation}
		with the convention $\|g_k\|_2/\|H_k\|_2:=+\infty$ when $\|H_k\|_2=0$.
		If $g_k=0$, the algorithm either returns $s_k=0$ or a boundary step along a detected negative-curvature direction such that
		$\pred_k>0$.
	\end{assumption}
	
	\begin{assumption}[Precision bounds]\label{asmp:conv-Wbounds}
		The precision matrix $\widehat W_k$ satisfies the uniform bounds
		\begin{equation}\label{eq:conv-Wbounds}
			w_{\min}I \preceq \widehat W_k \preceq w_{\max}I,
			\qquad 0<w_{\min}\le w_{\max}<\infty,
			\qquad \forall k.
		\end{equation}
		The diagonal block structure \eqref{eq:bup-W-block} satisfies this by construction when the diagonal precision entries are clipped
		to $[w_{\min},w_{\max}]$.
	\end{assumption}
	
	\begin{assumption}[MAP-poisedness]\label{asmp:conv-mappoised}
		After geometry management, $Y_k$ is $\mu_M$-MAP-poised, i.e.,
		\begin{equation}\label{eq:conv-mappoised}
			\lambda_{\min}\!\left(\widehat M_k(Y_k)\right)\ge \mu_M,\qquad \forall k,
		\end{equation}
		where $\widehat M_k(Y)=\widehat A(Y)\widehat W_k^{-1}\widehat A(Y)^\top$ is defined in \eqref{eq:repair-Mhat}.
	\end{assumption}
	
\begin{assumption}[Prior accuracy at the trust-region scale]\label{asmp:conv-prior}
	Let $\ell_k(s)=f(x_k)+\nabla f(x_k)^\top s$ and define its scaled coefficient vector
	\[
	\widehat c_k^\ell=\begin{bmatrix} f(x_k)\\ \Delta_k\nabla f(x_k)\\ 0\end{bmatrix}.
	\]
	The prior mean $\widehat c_k^\pi$ satisfies
	\begin{equation}\label{eq:conv-prior-acc}
		\|\widehat c_k^\pi-\widehat c_k^\ell\|_{\widehat W_k} \le \bar\kappa_\pi\,\Delta_k^2,\qquad \forall k,
	\end{equation}
	for some constant $\bar\kappa_\pi\ge 0$,
	where $\|v\|_{\widehat W_k}:=\sqrt{v^\top\widehat W_k\,v}$ denotes the
	$\widehat W_k$-norm.
\end{assumption}
	
	\begin{assumption}[Criticality safeguard]\label{asmp:conv-crit}
		There exists $\kappa_\Delta>0$ such that whenever $\|g_k\|_2\le \kappa_\Delta \Delta_k$, the algorithm reduces $\Delta_k$
		and rebuilds a MAP-poised set and BUP model, repeating until $\|g_k\|_2>\kappa_\Delta\Delta_k$ (or a stopping test is met).
	\end{assumption}
	
	\begin{definition}[Fully-linear model]\label{def:conv-fullylinear}
		A model $m_k$ is $(\kappa_f,\kappa_g)$-\emph{fully linear} on $B(x_k,\Delta_k)$ if for all $\|s\|_2\le \Delta_k$,
		\begin{align}
			|f(x_k+s)-m_k(s)| &\le \kappa_f \Delta_k^2,\label{eq:conv-FL1}\\
			\|\nabla f(x_k+s)-\nabla m_k(s)\|_2 &\le \kappa_g \Delta_k.\label{eq:conv-FL2}
		\end{align}
	\end{definition}

\begin{remark}[Scope of Assumption~\ref{asmp:conv-prior}]\label{rem:scope-prior}
	Assumption~\ref{asmp:conv-prior} isolates the accuracy required from the
	prior source.  The geometry mechanism enforces MAP-poisedness, and the
	clipping step enforces the precision bounds; the prior-accuracy condition
	depends on how the prior mean is constructed.  It requires the prior mean
	to approximate the linear Taylor coefficients of $f$ at the trust-region
	scale, measured by the $\widehat W_k$-norm.  Coefficient
	directions assigned larger precision therefore require correspondingly
	more accurate prior information.
	A quantitative sufficient condition for surrogate-based priors is given in
	Appendix~\ref{app:gp-to-A6}.  Other prior families can be used in the
	analysis whenever they satisfy
	Assumption~\ref{asmp:conv-prior}.
\end{remark}

\begin{lemma}[Accuracy of the transported accepted model]\label{lem:accepted-map-prior-acc}
	Let $k'\le k$ be the most recent iteration at which a hard-MAP
	completion was accepted, and define
	$\widehat c_k^\pi:=\mathcal T_{k'\to k}(\widehat c_{k'})$ by
	\eqref{eq:bup-transport-map}--\eqref{eq:bup-accepted-map-prior}.
	Suppose the model at~$k'$ is fully linear with scaled coefficient
	error
	$\|\widehat c_{k'}-\widehat c_{k'}^\ell\|_2
	\le \kappa_e\,\Delta_{k'}^2$, and suppose the displacement and radius
	comparison satisfy
	\begin{equation}\label{eq:accepted-map-prior-cond}
		\|x_k - x_{k'}\| \le C_s\,\Delta_k,
		\qquad
		\Delta_{k'} \le C_\Delta\,\Delta_k.
	\end{equation}
	Then the accepted-model prior satisfies Assumption~\ref{asmp:conv-prior}
	with
	\begin{equation}\label{eq:accepted-map-prior-kpi}
		\bar\kappa_\pi
		\;\le\;
		\sqrt{w_{\max}}\,
		\bigl(C_T\,\kappa_e + C_s\,L_g\bigr),
	\end{equation}
	where $L_g$ is the gradient Lipschitz constant from
	Assumption~\ref{asmp:conv-smooth}, and one admissible choice is
	$C_T=C_\Delta+\sqrt{2}C_s+1$.
\end{lemma}

\begin{proof}
	Write the transported prior error as
	\begin{align*}
	\widehat c_k^\pi-\widehat c_k^\ell
	&=\mathcal T_{k'\to k}(\widehat c_{k'})-\widehat c_k^\ell \\
	&=\mathcal L_{k'\to k}(\widehat c_{k'}-\widehat c_{k'}^\ell)
	+\bigl(\mathcal T_{k'\to k}(\widehat c_{k'}^\ell)-\widehat c_k^\ell\bigr),
	\end{align*}
	where the constant blocks cancel because the affine transport and
	$\widehat c_k^\ell$ both use $f(x_k)$ in the constant block.
	The constant block of $\widehat c_{k'}-\widehat c_{k'}^\ell$ is zero
	because both the accepted hard-MAP model and the linear Taylor model
	interpolate $f(x_{k'})$ at the center.
	For the first term, write the coefficient error at $k'$ as scaled
	gradient and Hessian blocks $(e_g,e_H)$, and define the corresponding
	unscaled differences
	\[
	\delta g := \Delta_{k'}^{-1}e_g,
	\qquad
	\delta H := \Delta_{k'}^{-2}\operatorname{mat}(e_H),
	\]
	where $\operatorname{mat}(e_H)$ is the symmetric matrix whose vech is
	$e_H$.  With $d=x_k-x_{k'}$, the transported gradient block is
	\[
	\Delta_k(\delta g+\delta H d)
	=
	\frac{\Delta_k}{\Delta_{k'}}e_g
	+
	\frac{\Delta_k}{\Delta_{k'}^2}\operatorname{mat}(e_H)d,
	\]
	and the transported Hessian block is
	\[
	\Delta_k^2\operatorname{vech}(\delta H)
	=
	\left(\frac{\Delta_k}{\Delta_{k'}}\right)^2e_H.
	\]
	Using \eqref{eq:accepted-map-prior-cond} and
	$\|\operatorname{mat}(e_H)\|_2\le\|\operatorname{mat}(e_H)\|_F\le\sqrt{2}\|e_H\|_2$
	yields
	\[
	\|\mathcal L_{k'\to k}(\widehat c_{k'}-\widehat c_{k'}^\ell)\|_2
	\le (C_\Delta+\sqrt{2}C_s+1)
	\,\kappa_e\,\Delta_k^2.
	\]
	For the second term, transporting the linear Taylor coefficient vector
	from $x_{k'}$ to $x_k$ changes only the gradient block after the constant
	reset.  By Assumption~\ref{asmp:conv-smooth},
	\[
	\|\nabla f(x_{k'})-\nabla f(x_k)\|_2
	\le L_g\|x_{k'}-x_k\|_2
	\le L_g C_s\Delta_k,
	\]
	and scaling by $\Delta_k$ yields
	\[
	\|\mathcal T_{k'\to k}(\widehat c_{k'}^\ell)-\widehat c_k^\ell\|_2
	\le C_s L_g\Delta_k^2.
	\]
	Since $\widehat W_k\preceq w_{\max}I$, the triangle inequality gives
	\[
	\|\widehat c_k^\pi-\widehat c_k^\ell\|_{\widehat W_k}
	\le \sqrt{w_{\max}}\,(C_T\kappa_e+C_sL_g)\Delta_k^2,
	\]
	which is \eqref{eq:accepted-map-prior-kpi}.

	\emph{Initialization and restarts.}
	The transported-prior argument applies after an accepted model is available.
	Before that point, Assumption~\ref{asmp:conv-prior} must be verified for the
	chosen initialization prior, or imposed as part of the initialization
	hypothesis.  Since this affects only finitely many initial model constructions,
	its cost can be absorbed into the finite initialization overhead in the
	evaluation-complexity bound.
	If implementation-level restarts are used, the same argument applies on each
	accepted-model segment after the restart; these restarts are not part of the
	baseline algorithm analyzed in
	Theorems~\ref{thm:conv-global}--\ref{thm:conv-complexity}.
\end{proof}

\begin{remark}[When the transport condition fails]\label{rem:transport-failure}
	The displacement bound~\eqref{eq:accepted-map-prior-cond} may be violated when
	many consecutive rejected steps cause $\|x_k-x_{k'}\|\gg C_s\Delta_k$, or
	when the radius shrinks substantially between iterations $k'$ and~$k$.
	Lemma~\ref{lem:accepted-map-prior-acc} applies only on iterations for which
	the admissibility conditions in \eqref{eq:accepted-map-prior-cond} hold.
	If an implementation falls back to a zero-Hessian or minimum-norm prior when
	these conditions fail, that fallback is a practical safeguard and is not
	covered by Lemma~\ref{lem:accepted-map-prior-acc} unless it is separately
	constructed or verified to satisfy Assumption~\ref{asmp:conv-prior}.  The
	cascade ablation in Table~\ref{tab:ablation} is therefore interpreted as an
	implementation study rather than as an additional theoretical guarantee.
\end{remark}
	
	\subsection{Fully-linear accuracy of the MAP model}\label{subsec:conv-fl}
	
	\begin{lemma}[Fully-linear error bound for hard-MAP models]\label{lem:conv-FL-BUP}
		Suppose Assumptions~\ref{asmp:conv-smooth}, \ref{asmp:conv-Wbounds}, \ref{asmp:conv-mappoised}, and \ref{asmp:conv-prior} hold.
		Let $\widehat c_k$ be the hard-MAP solution \eqref{eq:bup-hard-closed} on $Y_k$, and let $m_k$ be the corresponding model.
		Define
	\begin{equation}\label{eq:conv-ke}
		\kappa_e
		:=
		\frac{\bar\kappa_\pi}{\sqrt{w_{\min}}}
		+\frac{\ctrim^2\,\Bgeo\,L_g(m+1)}{2w_{\min}\mu_M}.
	\end{equation}
		Then $m_k$ is fully linear on $B(x_k,\Delta_k)$ with constants
		\begin{equation}\label{eq:conv-kappas}
		\kappa_f := \frac{L_g}{2} + B_\phi\,\kappa_e,
		\qquad
		\kappa_g := L_g + (1+\sqrt{2})\,\kappa_e,
		\end{equation}
		and the model Hessian is uniformly bounded:
		\begin{equation}\label{eq:conv-Hmax}
			\|H_k\|_2 \le H_{\max}:=\sqrt{2}\,\kappa_e,\qquad \forall k.
		\end{equation}
	\end{lemma}
	
	\begin{proof}
		See Appendix~\ref{app:proof-FL}.
	\end{proof}
	
	The lemma above applies to any single MAP-poised iteration.
	It remains to verify that the geometry-repair mechanism restores MAP-poisedness within
	a bounded number of evaluations, making the fully-linear guarantee available at
	every iteration of Algorithm~\ref{alg:buptr}.
	
	\phantomsection\label{subsec:conv-repair}
	
	\begin{lemma}[MAP-poisedness of the fallback set]\label{lem:conv-fallback}
		Suppose Assumption~\ref{asmp:conv-Wbounds} holds and set $m=2n$.
		Let $Y_k^{\mathrm{fb}}=\{x_k;\ x_k\pm \Delta_k e_1,\ldots,x_k\pm \Delta_k e_n\}$ be the fallback set \eqref{eq:repair-fallback-set}.
		Then
		\begin{equation}\label{eq:conv-mu0}
			\lambda_{\min}\!\left(\widehat M_k(Y_k^{\mathrm{fb}})\right)\ \ge\
			\mu_0:=\frac{1}{w_{\max}}\cdot\frac{1}{4n+3}.
		\end{equation}
	\end{lemma}
	
	\begin{proof}
		See Appendix~\ref{app:fallback}. The bound follows by combining
		$\widehat W_k^{-1}\succeq (1/w_{\max})I$ with an explicit eigenvalue lower bound for
		$\widehat A(Y_k^{\mathrm{fb}})\widehat A(Y_k^{\mathrm{fb}})^\top$.
	\end{proof}
	
	\begin{lemma}[Uniform repair-evaluation bound]\label{lem:conv-repair-finite}
		If the geometry mechanism of Subsection~\ref{sec:repair} is used with attempt budget $T_{\mathrm{try}}$ and threshold
		$\mu_M\le \mu_0$, where $\mu_0$ is defined in \eqref{eq:conv-mu0}, then at every main iteration $k$, the
		base non-trial evaluation count (for a single criticality-loop pass) satisfies
		\begin{equation}\label{eq:conv-repair-evals}
			N_k^{\mathrm{rep,base}} \le T_{\mathrm{try}} + 2n.
		\end{equation}
	Each additional criticality shrink within the same iteration adds at most $T_{\mathrm{try}}+2n$ evaluations to
	$N_k^{\mathrm{rep,crit}}$ (from the subsequent repair pass; the warm-start set update after the radius shrink incurs no new evaluations).
\end{lemma}
	
	\begin{proof}
		The base non-trial evaluations at a single main iteration $k$ (one pass through the outer \textbf{for} loop of
		Algorithm~\ref{alg:buptr}, without extra criticality shrinks) decompose into two groups:
		\begin{enumerate}[leftmargin=2em,label=(\alph*)]
		\item \emph{Pre-model repair} (incremental attempts): at most $T_{\mathrm{try}}$ (one evaluation per attempt).
		\item \emph{Pre-model fallback reset} (if incremental repair fails): at most $2n$ new evaluations
		(the $2n$ coordinate-direction points in $Y_k^{\mathrm{fb}}$ minus any already cached in $\mathcal{D}_k$;
		at most $2n$ in the worst case).
		\end{enumerate}
		The warm-start set update at the end of the iteration reuses only previously evaluated points from $\mathcal{D}_k$
		and incurs no new evaluations.
		Summing (a)--(b) gives $N_k^{\mathrm{rep,base}}\le T_{\mathrm{try}}+2n$.
		
	Each additional criticality shrink first performs a warm-start set update (no new evaluations)
	and then re-enters the model-building phase (GP fit, geometry check, repair, MAP build),
	where the repair costs at most $T_{\mathrm{try}}$ (incremental) plus $2n$ (fallback).
	Hence each extra shrink adds at most $T_{\mathrm{try}}+2n$ evaluations, giving
	$N_k^{\mathrm{rep,crit}}\le L_k(T_{\mathrm{try}}+2n)$ where $L_k\ge 0$ is the number of extra criticality
	shrinks at iteration~$k$.
		Lemma~\ref{lem:conv-fallback} guarantees that every fallback reset restores MAP-poisedness for $\mu_M\le \mu_0$.
	\end{proof}
	
	\begin{corollary}[Geometry guarantee after repair]\label{cor:conv-mappoised-guaranteed}
		Under the geometry mechanism of Subsection~\ref{sec:repair} with
		$\mu_M \le \mu_0$, Algorithm~\ref{alg:buptr} ensures that
		Assumption~\ref{asmp:conv-mappoised} holds at every iteration.
	\end{corollary}
	\begin{proof}
		At each model-building phase, if the pre-model check
		\eqref{eq:repair-check} detects
		$\lambda_{\min}(\widehat M_k(Y_k)) < \mu_M$,
		incremental repair is attempted (at most $T_{\mathrm{try}}$ times).
		If the MAP-poisedness test still fails, the fallback reset
		\eqref{eq:repair-reset-rule} sets $Y_k \leftarrow Y_k^{\mathrm{fb}}$,
		and Lemma~\ref{lem:conv-fallback} guarantees
		$\lambda_{\min}(\widehat M_k(Y_k^{\mathrm{fb}})) \ge \mu_0 \ge \mu_M$.
		Hence \eqref{eq:conv-mappoised} holds upon completion of the
		geometry phase at every iteration.
	\end{proof}
	
	\subsection{Global convergence and evaluation complexity}\label{subsec:conv-accept}
	
	In our construction, the hard-MAP model satisfies $m_k(0)=f(x_k)$ by design: the constant block of the prior mean
	$\widehat c_k^\pi$ is set to $f(x_k)$ in \eqref{eq:bup-prior-gp}, and the interpolation constraint
	$\widehat A_k\widehat c=b_k$ preserves this value at the center.
	
	\begin{lemma}[Actual--predicted reduction error]\label{lem:conv-aredpred}
		If $m_k(0)=f(x_k)$ and $m_k$ is $(\kappa_f,\kappa_g)$-fully linear on $B(x_k,\Delta_k)$, then for any $\|s_k\|_2\le \Delta_k$,
		\begin{equation}\label{eq:conv-aredpred}
			|\ared_k-\pred_k|\le \kappa_f \Delta_k^2.
		\end{equation}
	\end{lemma}
	
	\begin{lemma}[Acceptance at sufficiently small radius]\label{lem:conv-succ-small}
		Suppose $m_k(0)=f(x_k)$ and $m_k$ is $(\kappa_f,\kappa_g)$-fully linear on $B(x_k,\Delta_k)$.
		If $\pred_k>0$ and
		\begin{equation}\label{eq:conv-delta-thr}
			\Delta_k \le \sqrt{\frac{(1-\eta_1)\pred_k}{\kappa_f}},
		\end{equation}
		then $\rho_k\ge \eta_1$ and the iteration is successful.
	\end{lemma}
	
	\begin{lemma}[Successful step above a gradient threshold]\label{lem:conv-succ-eps}
		Suppose $m_k(0)=f(x_k)$, $m_k$ is $(\kappa_f,\kappa_g)$-fully linear, and $\|H_k\|_2\le H_{\max}$.
		Fix $\varepsilon\in(0,1]$. If
		\begin{equation}\label{eq:conv-Delta-succ}
			\|\nabla f(x_k)\|_2 \ge \varepsilon
			\quad\text{and}\quad
			\Delta_k \le \Delta_{\mathrm{succ}}(\varepsilon)
			:=
			\min\left\{
			\frac{\varepsilon}{2\kappa_g},\ \frac{\varepsilon}{2H_{\max}},\ \frac{1-\eta_1}{4\kappa_f}\varepsilon,\ 1
			\right\},
		\end{equation}
		then $\rho_k\ge \eta_1$ and the iteration is successful.
	\end{lemma}
	
	\begin{lemma}[Finite criticality loop]\label{lem:conv-crit-exit}
		Suppose $m_k(0)=f(x_k)$ and $m_k$ is $(\kappa_f,\kappa_g)$-fully linear on $B(x_k,\Delta_k)$.
		Fix $\varepsilon>0$.
		If $\|\nabla f(x_k)\|_2\ge \varepsilon$ and
		\[
		\Delta_k \le \Delta_{\mathrm{crit}}(\varepsilon):=\min\!\left\{\frac{\varepsilon}{4\kappa_g},\ \frac{\varepsilon}{2\kappa_\Delta}\right\},
		\]
		then $\|g_k\|_2>\kappa_\Delta\Delta_k$, so the criticality safeguard is inactive.
	\end{lemma}

Lemmas~\ref{lem:conv-aredpred}--\ref{lem:conv-crit-exit} are standard
trust-region acceptance and criticality results
\cite[Ch.\,10]{conn2009introduction}; their short proofs follow directly
from the fully-linear bounds
\eqref{eq:conv-FL1}--\eqref{eq:conv-FL2} and the Cauchy decrease condition
(Assumption~\ref{asmp:conv-cauchy}); complete proofs are given in
Appendix~\ref{app:proof-accept}.
	
	With fully-linear accuracy (Section~\ref{subsec:conv-fl}), bounded repair overhead
	(Section~\ref{subsec:conv-repair}), and the acceptance/criticality lemmas above
	in hand, we can now state the main convergence and complexity results.
	
	\phantomsection\label{subsec:conv-main}
	
	\begin{theorem}[Global first-order convergence]\label{thm:conv-global}
		Suppose Assumptions~\ref{asmp:conv-smooth}--\ref{asmp:conv-crit} hold, and the hard-MAP model is used.
		Then
		\[
		\liminf_{k\to\infty}\|\nabla f(x_k)\|_2=0.
		\]
	\end{theorem}
	
	\begin{proof}
		See Appendix~\ref{app:proof-global}.
	\end{proof}
	\begin{theorem}[Worst-case evaluation complexity]\label{thm:conv-complexity}
		Suppose Assumptions~\ref{asmp:conv-smooth}--\ref{asmp:conv-crit} hold, the hard-MAP model is used, and the geometry mechanism satisfies Lemma~\ref{lem:conv-repair-finite}.
		Then there exists a constant $C_{\mathrm{eval}}>0$, independent of $\varepsilon$, such that for any $\varepsilon\in(0,1]$,
		there exists a generated iterate $x_j$ satisfying $\|\nabla f(x_j)\|_2\le \varepsilon$ among the iterates obtained after at most
		\[
		C_{\mathrm{eval}}\,\varepsilon^{-2}
		\]
		objective evaluations.
	\end{theorem}
	
	\begin{proof}
		See Appendix~\ref{app:proof-complexity}.
	\end{proof}
	
	\begin{remark}[Dependence of $C_{\mathrm{eval}}$ on problem data]\label{rem:dim-dependence}
		The constant $C_{\mathrm{eval}}$ depends on $n$ through $m=2n$,
		$\Bgeo$, $\kappa_e$, and the per-iteration repair bound
		$T_{\mathrm{try}}+2n$.  It also depends on the precision-matrix bounds
		$w_{\min}$, $w_{\max}$ (via $\mu_M$ and $\bar\kappa_\pi$),
		the trust-region parameters $\eta_1$, $\gamma_{\mathrm{dec}}$,
		$\gamma_{\mathrm{inc}}$, and the initial optimality gap
		$f(x_0)-f_{\mathrm{low}}$.
		Tracking the dimension dependence explicitly yields
		$C_{\mathrm{eval}} = \mathcal{O}(n^2/\mu_M^2)$; the dominant factor
		is the $n/\mu_M$ scaling of $\kappa_e$ in \eqref{eq:conv-ke}, which
		enters quadratically through the decrease bound.
		This matches the known $\mathcal{O}(n^2\varepsilon^{-2})$ complexity
		of classical fully-linear trust-region methods for derivative-free
		optimization
		\cite[Ch.\,10]{conn2009introduction}.
	\end{remark}
	
	The analysis above shows that BUP-TR converges at the classical
	$\mathcal{O}(\varepsilon^{-2})$ rate despite the prior-regularized
	completion and geometry repair guided by $\lambda_{\min}$.
	We examine how these structural changes translate into
	fixed-budget accuracy, efficiency, and robustness.
	
	\section{Numerical Experiments}\label{sec:experiments}

	We assess the method through BUP-NEWUOA, a NEWUOA-style implementation
	using the recursive accepted-model prior
	(Section~\ref{subsec:bup-accepted-map}).
	BUP-NEWUOA and NEWUOA use the same core Powell-type trust-region
	skeleton, including the acceptance rule and evaluation budget.  The
	comparison is designed to assess the combined effect of the BUP completion
	and geometry mechanism within a NEWUOA-style implementation.
	The cascade ablation in Table~\ref{tab:ablation} separates the main
	implementation components, while implementation-only enhancements such as
	the bounded restart heuristic are reported explicitly in
	Tables~\ref{tab:exp-params} and~\ref{tab:supp-params-bup}.
	This comparison with NEWUOA is the primary empirical test in the
	paper; the remaining solvers provide broader contextual baselines.
	We start with aggregate summaries under the prescribed evaluation
	budget, then move to problem-level convergence views, and finally to
	ablation and noisy experiments.

	\subsection{Experimental setup}\label{subsec:exp-setup}

	The benchmark comprises 17 unconstrained test functions drawn from
	the Mor{\'e}--Garbow--Hillstrom collection~\cite{more2009benchmarking},
	the CUTEst library, and custom engineering-inspired problems,
	spanning narrow-valley, quartic/polynomial, mixed-scale, and
	chain-coupled structures (Appendix~\ref{app:supp-problems}).
	Each function is tested at $n\in\{5,10,20,30,50\}$, yielding
	85 problem--dimension pairs.
	The evaluation budget is $N_{\max}=500(n+1)$ following the PDFO
	convention~\cite{ragonneau2024pdfo}; each solver--problem pair is
	run with five independent seeds $\{42,123,7,256,999\}$.

	The solver \emph{BUP-NEWUOA} uses the recursive accepted-model prior mean,
	local WLS curvature statistics, structured Hessian
	decay ($\alpha_d=1.5$, implemented as $\lambda_{\mathrm{decay}}$),
	and at most two implementation-level restarts.  These restarts are
	implementation-level safeguards, not part of the baseline theoretical
	algorithm.  Because the bounded restart heuristic is enabled only for
	BUP-NEWUOA, we report it explicitly as an implementation safeguard and
	separate its effect in the ablation study whenever the corresponding
	no-restart runs are available.  The shared algorithm parameters are summarized in
	Table~\ref{tab:exp-params}, and the full BUP-NEWUOA-specific settings
	are listed in Appendix~\ref{app:supp-params}.
	\emph{NEWUOA} uses this trust-region skeleton with Powell's least-change
	quadratic model, often described as a minimum-Frobenius-change
	update~\cite{powell2006newuoa}.
	The additional derivative-free baselines are
	\emph{UOBYQA}~\cite{powell2002uobyqa},
	\emph{Nelder--Mead}~\cite{nelder1965simplex}, and
	\emph{CMA-ES}~\cite{hansen2001cmaes}; all use library defaults and
	the common budget.
	The median wall-clock time per run is 3.41\,s for BUP-NEWUOA vs.\
	2.52\,s for NEWUOA ($1.35{\times}$ overhead), a modest premium
	relative to the evaluation-count gains reported below and typically
	secondary for expensive black-box objectives.

	\begin{table}[t]
	\centering
	\caption{Algorithm parameters.}\label{tab:exp-params}
	\small
	\begin{tabular}{@{}llll@{}}
	\toprule
	Category & Parameter & Symbol & Value \\
	\midrule
	\multicolumn{4}{@{}l}{\emph{Shared trust-region parameters}}\\
	& Acceptance / expansion & $\eta_1,\eta_2$ & $0.10,\;0.70$ \\
	& Shrink / expand & $\gamma_{\mathrm{dec}},\gamma_{\mathrm{inc}}$ & $0.5,\;2.0$ \\
	& Interpolation parameter & $m$ & $2n$ \\
	& Set size & $|Y_k|$ & $2n+1$ \\
	& Initial / final radius & $\rho_{\mathrm{beg}},\rho_{\mathrm{end}}$ & $1.0,\;10^{-8}$ \\
	& Budget & $N_{\max}$ & $500(n+1)$ \\
	\midrule
	\multicolumn{4}{@{}l}{\emph{BUP completion (BUP-NEWUOA only)}}\\
	& Prior source & --- & accepted-model \\
	& Local curvature statistics & --- & distance-weighted WLS \\
	& Hessian decay & $\alpha_d$ (code: $\lambda_{\mathrm{decay}}$) & $1.5$ \\
	& Prior transfer & --- & accepted-model center shift \\
	& Restart & --- & True (max 2) \\
	\bottomrule
	\end{tabular}
	\end{table}

	The primary metric is the \emph{success rate under the prescribed evaluation budget}: the
	fraction of test instances (problem--dimension--seed runs) achieving
	$f_{\mathrm{rel}}:=|f(x_{\mathrm{best}})-f^\star|/(|f(x_0)-f^\star|+10^{-16})<\tau$
	within $N_{\max}$, at
	$\tau\in\{10^{-1},10^{-3},10^{-5},10^{-7}\}$.
	We refer to this as the fixed-budget success rate below.
	We also report Dolan--Mor{\'e} performance
	profiles~\cite{dolan2002benchmarking} and More--Wild data
	profiles~\cite{more2009benchmarking}.

	\subsection{Fixed-budget benchmark results}\label{subsec:exp-results}

	Table~\ref{tab:exp-success} reports run-level success rates across all
	85 problem--dimension pairs and five seeds, giving \(425\) runs per solver
	and \(2,125\) solver-runs in total across the five solvers.

	\begin{table}[t]
	\centering
	\caption{Run-level success rate (\%) and median relative error over
	85 problem--dimension pairs, 5~seeds each
	($N_{\max}=500(n+1)$).}\label{tab:exp-success}
	\footnotesize
	\setlength{\tabcolsep}{4.0pt}
	\begin{tabular}{lccccc}
	\toprule
	 & BUP-NEWUOA & NEWUOA & UOBYQA & Nelder--Mead & CMA-ES \\
	\midrule
	$\tau{=}10^{-1}$ & 96.2 & 96.5 & 96.9 & 77.9 & \textbf{98.4} \\
	$\tau{=}10^{-3}$ & \textbf{87.5} & 84.9 & 80.7 & 38.4 & 83.1 \\
	$\tau{=}10^{-5}$ & \textbf{74.8} & 67.1 & 65.9 & 28.0 & 67.1 \\
	$\tau{=}10^{-7}$ & \textbf{72.0} & 61.2 & 59.5 & 24.9 & 64.5 \\
	\midrule
	Med.\ $f_{\mathrm{rel}}$ & $\mathbf{4.2{\times}10^{-15}}$
	  & $3.6{\times}10^{-10}$ & $2.7{\times}10^{-11}$
	  & $5.8{\times}10^{-3}$ & $9.0{\times}10^{-13}$ \\
	\bottomrule
	\end{tabular}
	\end{table}

	\noindent
	At the stricter tolerances $\tau\in\{10^{-3},10^{-5},10^{-7}\}$,
	BUP-NEWUOA achieves the highest success rate among the tested solvers
	in the noiseless benchmark, improving over
	NEWUOA by 2.6, 7.8, and 10.8 percentage points at $\tau\in\{10^{-3},
	10^{-5},10^{-7}\}$.
	Bootstrap 95\% confidence intervals (CIs; $B=10{,}000$), obtained by
	cluster resampling the 85 problem--dimension pairs while retaining all
	5 seeds in each sampled cluster, for the BUP--NEWUOA success-rate
	difference are
	$[3.1,12.7]$ percentage points at $\tau=10^{-5}$ and
	$[5.4,16.9]$ percentage points at $\tau=10^{-7}$.
	BUP-NEWUOA also attains the smallest median final relative error
	($4.2\times10^{-15}$, last row of Table~\ref{tab:exp-success})
	across its 425 runs, five orders of magnitude below NEWUOA
	($3.6\times10^{-10}$).

	We also compare final-error distributions using paired Wilcoxon
	signed-rank tests.  The pairing unit is a problem--dimension pair, and
	the 5~seeds are aggregated by the median
	$\log_{10} f_{\mathrm{rel}}$.  The tests indicate a statistically
	significant advantage for BUP-NEWUOA over NEWUOA
	($p=4.5\times10^{-3}$, Cliff's delta $=-0.16$),
	CMA-ES ($p=1.4\times10^{-4}$, effect size $=-0.61$), and
	UOBYQA ($p=3.3\times10^{-2}$, effect size $=-0.19$).
	The negative signs reflect lower errors for BUP-NEWUOA under our
	difference convention.
	These tests quantify distributional differences in the
	problem--dimension median error metric; they are complementary to,
	and distinct from the run-level success-rate improvements reported
	above.  All three comparisons remain significant after
	Holm--Bonferroni correction at the 5\% level.

	Figure~\ref{fig:pairwise-bup-newuoa} complements the aggregate success
	rates by plotting the median final relative error of BUP-NEWUOA against
	that of NEWUOA on each problem--dimension pair.  The figure shows that
	49 of the 85 pairs fall below the diagonal, 32 above it, and 4 are
	exact ties.  Thus the median-error advantage is visible on a majority of
	the pairs, while NEWUOA remains better on a substantial minority of the
	benchmark.

	\begin{figure}[t]
	\centering
	\includegraphics[width=0.58\textwidth]{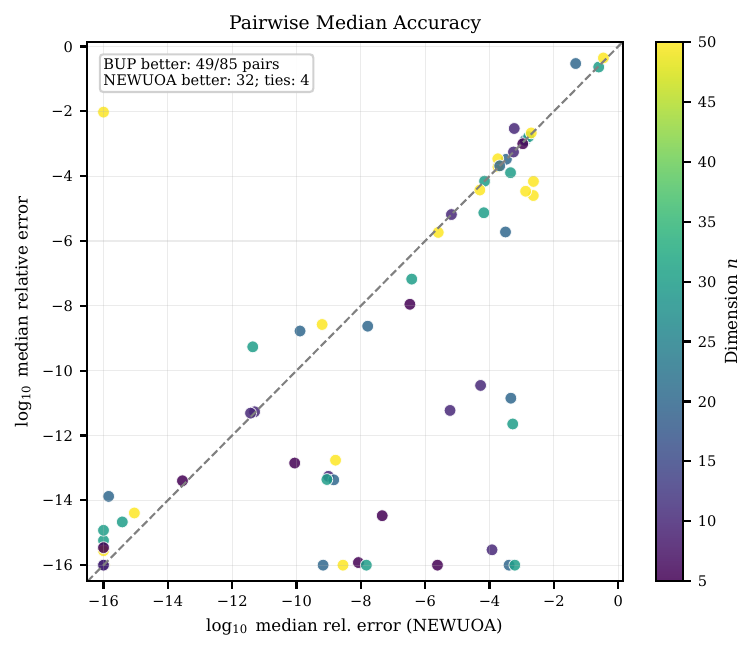}
	\caption{Pairwise comparison of median final relative error on each
	problem--dimension pair.  Points below the diagonal favor
	BUP-NEWUOA; the counts are 49 BUP-NEWUOA wins, 32 NEWUOA wins,
	and 4 ties.}\label{fig:pairwise-bup-newuoa}
	\end{figure}

	Figure~\ref{fig:perf-profiles} shows performance profiles at four
	tolerances, using the first evaluation count at which a run reaches the
	target accuracy.
	Because Dolan--Mor{\'e} profiles condition on solved instances, they
	emphasize efficiency instead of aggregate success.
	Among the derivative-free solvers in our comparison, BUP-NEWUOA
	lies above the remaining baselines over most of the profile range at
	strict tolerances.
	At $\tau=10^{-7}$, BUP-NEWUOA reaches about 79\% of solvable
	instances within twice the best-solver budget, versus about 44\% for
	NEWUOA, 27\% for UOBYQA, and 14\% for CMA-ES.

	\begin{figure}[tp]
	\centering
	\includegraphics[width=0.92\textwidth]{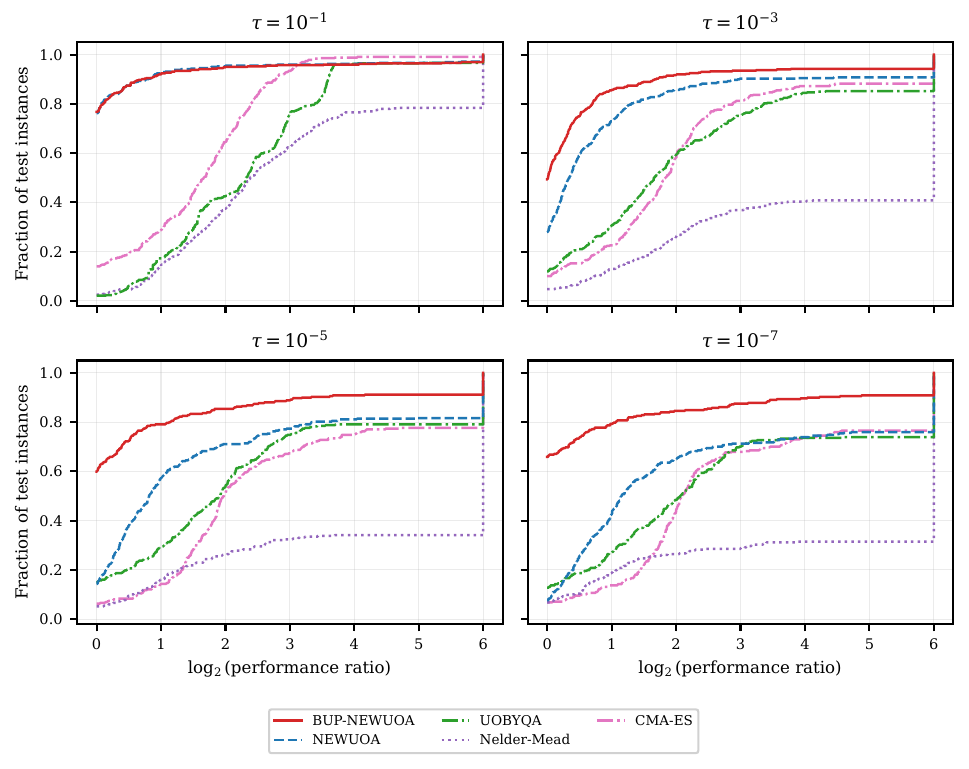}
	\caption{Dolan--Mor{\'e} performance profiles at four accuracy
	levels.  Each panel shows the fraction of test instances solved
	within a given ratio of the best solver's evaluation count.}%
	\label{fig:perf-profiles}
	\end{figure}

	Figure~\ref{fig:data-profiles} shows data profiles.
	At $\tau=10^{-5}$, BUP-NEWUOA reaches about 63\% success within
	$\alpha=100$ simplex gradients, while NEWUOA remains below 57\%;
	NEWUOA needs nearly $\alpha=200$ to exceed 60\%.
	At $\tau=10^{-7}$, the gap widens further and persists through the
	full budget.

	\begin{figure}[tp]
	\centering
	\includegraphics[width=0.92\textwidth]{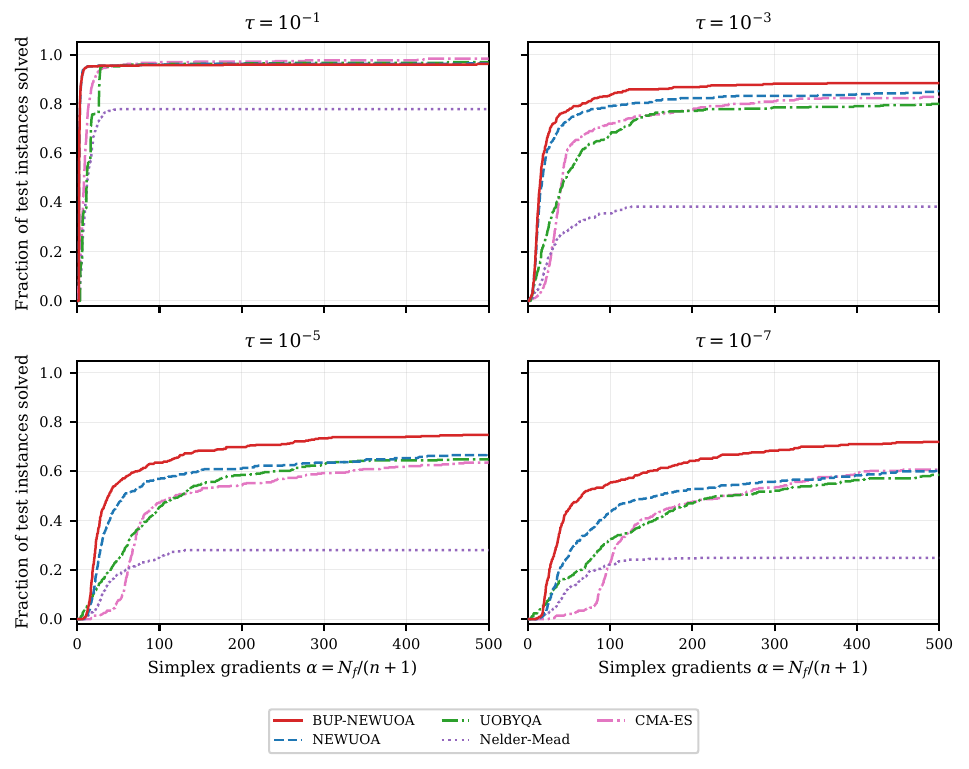}
	\caption{More--Wild data profiles.  Each panel shows the fraction
	of test instances solved as a function of the evaluation budget
	(measured in simplex gradients $\alpha=N_f/(n{+}1)$, where $N_f$
	denotes the number of function evaluations) for one accuracy level
	$\tau$.}\label{fig:data-profiles}
	\end{figure}

	Figure~\ref{fig:conv-overview-n20} shows the median best-so-far
	relative-error curves for all 17 benchmark functions at the fixed
	dimension $n=20$.  Each panel reports the median over 5~seeds at the
	standard evaluation checkpoints.  BUP-NEWUOA attains the lowest final
	checkpoint median on 11 of the 17 functions.  The largest gains appear on
	Rosenbrock, Genrose, SeqProcess, Nondquar, and ScaledRosen, while
	near-ties remain on DixonPrice, Cragglvy, and Fletchcr.

	\begin{figure}[p]
	\centering
	\includegraphics[width=0.95\textwidth]{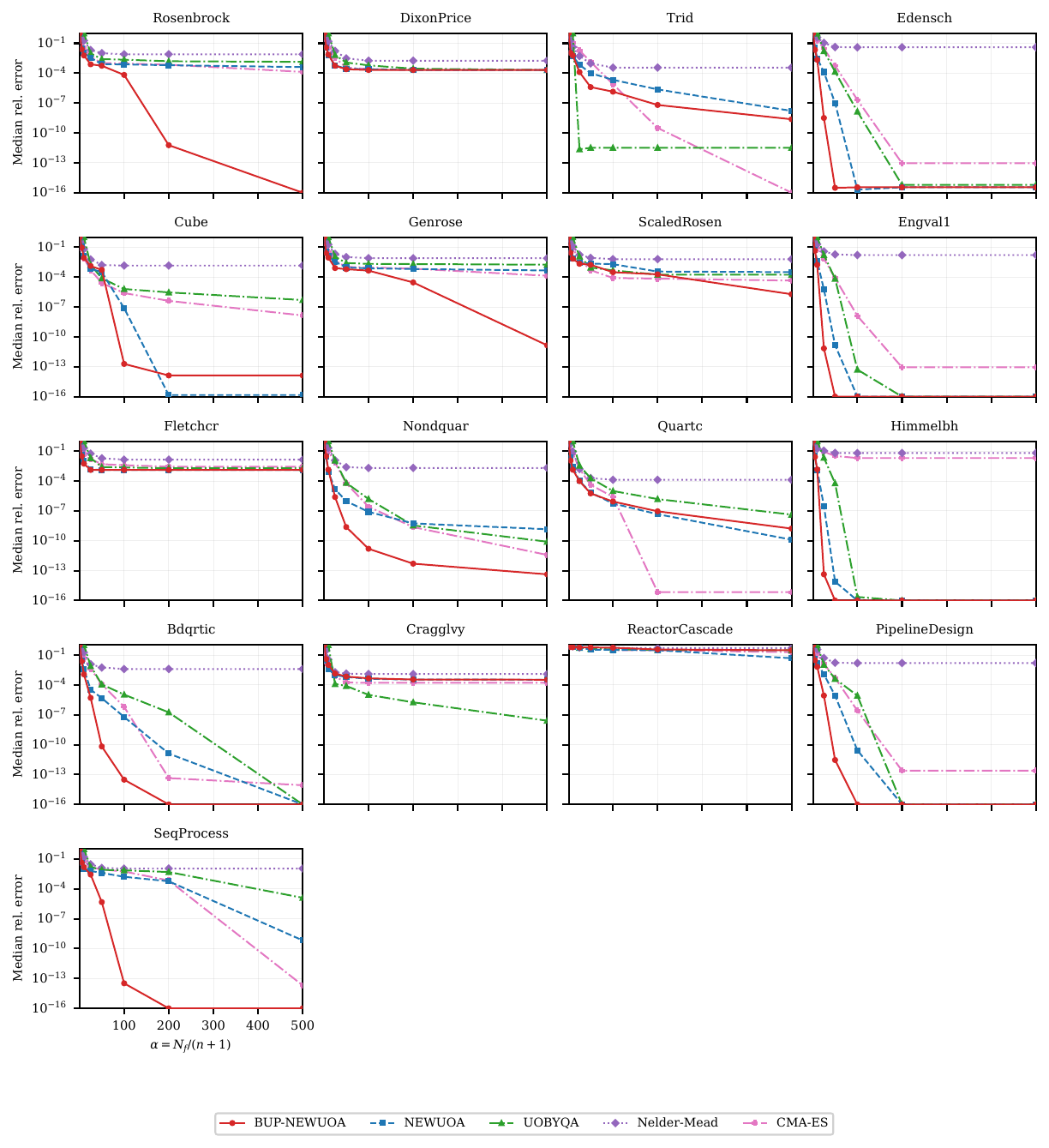}
	\caption{Convergence overview at the fixed dimension $n=20$.
	Each panel shows the median best-so-far relative error over 5~seeds for
	one test function.}
	\label{fig:conv-overview-n20}
	\end{figure}

	We then examine the main sources of the observed gains through
	ablation.
	
	Table~\ref{tab:ablation} reports an incremental cascade ablation on
	the full 85-pair suite.
	The accepted-model prior chain contributes the largest single
	gain (6.6 percentage points at $\tau=10^{-5}$),
	while the structured Hessian decay contributes 1.4 percentage points
	and WLS pooling is nearly neutral ($-0.2$ percentage points at
	$\tau=10^{-5}$).
	Although WLS pooling is nearly neutral in the aggregate
	success rate, we keep it in the default configuration because it is inexpensive,
	preserves the diagonal precision structure, and provides a mild safeguard
	against over-transferring long-range curvature entries in the accepted-model
	prior.
	The cumulative gain from NEWUOA to full BUP-NEWUOA is 7.8 percentage points.

	\begin{table}[t]
	\centering
	\caption{Cascade ablation: success rate (\%) at four accuracy
	levels.  Each row adds one component to the previous
	configuration.}\label{tab:ablation}
	\footnotesize
	\setlength{\tabcolsep}{4.5pt}
	\begin{tabular}{lcccc}
	\toprule
	Configuration & $\tau{=}10^{-1}$ & $\tau{=}10^{-3}$ & $\tau{=}10^{-5}$ & $\tau{=}10^{-7}$ \\
	\midrule
	NEWUOA (baseline)           & 96.5 & 84.9 & 67.1 & 61.2 \\
	\quad + WLS pooling         & 96.2 & 84.9 & 66.8 & 60.5 \\
	\quad + Struct.\ decay      & 96.2 & 85.6 & 68.2 & 62.8 \\
	\quad + Accepted-MAP (full) & \textbf{96.2} & \textbf{87.5} & \textbf{74.8} & \textbf{72.0} \\
	\midrule
	$\Delta$ (cumulative)       & $-0.2$ & $+2.6$ & $+7.8$ & $+10.8$ \\
	\bottomrule
	\end{tabular}
	\end{table}

	Figure~\ref{fig:ablation-profile} presents the same cascade in profile
	form, showing how the four configurations separate as the evaluation
	budget increases.  The accepted-model prior is the only component that
	produces a large separation at the two strictest tolerances, consistent
	with the aggregate gains in Table~\ref{tab:ablation}.

	\begin{figure}[!htbp]
	\centering
	\includegraphics[width=0.80\textwidth]{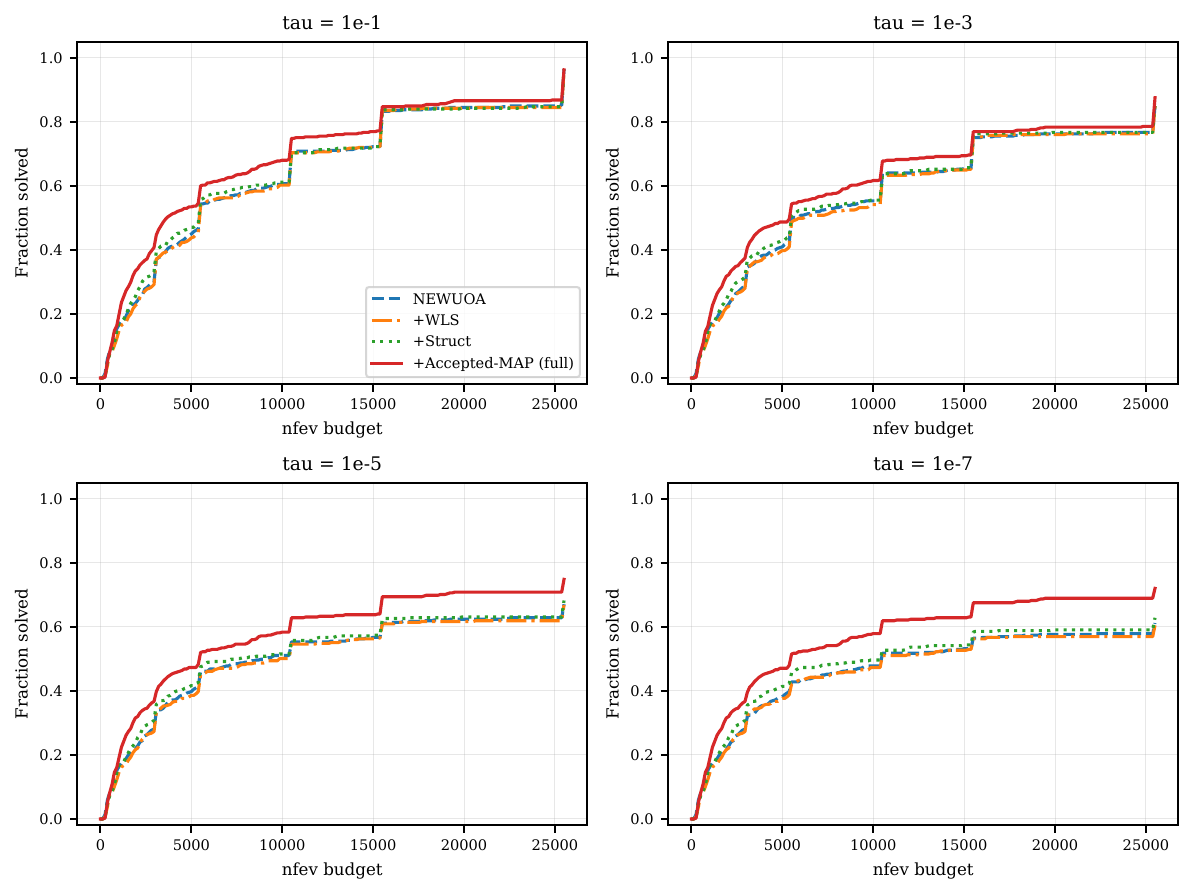}
	\caption{Cascade-ablation profiles across evaluation budgets.  Each
	panel shows the fraction of runs solved at one target accuracy; the
	curves correspond to the rows of Table~\ref{tab:ablation}.}
	\label{fig:ablation-profile}
	\end{figure}

	Across the suite, BUP-NEWUOA uses 2--8\% more evaluations than
	NEWUOA due to geometry repair triggers.  With the diagonal precision
	used here, forming the MAP Schur complement costs
	$\mathcal{O}(m^2q)$, followed by an $\mathcal{O}(m^3)$ Cholesky
	solve; for $m=2n$ and $q=\mathcal{O}(n^2)$, this remains modest for
	the tested dimensions.
	At $\tau=10^{-7}$, BUP-NEWUOA fails on about 28\% of the
	problem--dimension--seed runs.
	Failures concentrate on high-dimensional instances ($n\ge30$,
	budget-limited) and narrow curved valleys (Fletchcr, Quartc)
	where the quadratic model class is inherently limiting.

	\FloatBarrier
	
	\subsection{Robustness under evaluation noise}\label{subsec:exp-noisy}

	We test all solvers under homoscedastic additive Gaussian noise
	$f_{\mathrm{obs}}(x)=f(x)+\sigma\xi$, $\xi\sim\mathcal{N}(0,1)$,
	with $\sigma\in\{0,10^{-4},10^{-3},10^{-2},10^{-1}\}$ on the
	full 17-problem suite at $n\in\{10,20,30\}$, 5~seeds
	(255 runs per solver per noise level).
	The solvers are given only noisy observations, whereas success is assessed using
	the underlying noiseless objective.

	Figure~\ref{fig:noisy-main}(a) and Table~\ref{tab:noisy-summary}
	show the success rate at $\tau=10^{-3}$.
	BUP-NEWUOA deteriorates more gradually: at $\sigma=10^{-2}$ it retains
	68.2\% success (Wilson 95\% CI $[62.3,73.6]$) versus 11.0\% for
	NEWUOA ($[7.7,15.4]$).
	The prior provides implicit regularization that stabilizes model
	coefficients, whereas Powell's least-change quadratic model can transmit
	noisy interpolation data directly into the Hessian.
	Figure~\ref{fig:noisy-by-dim} shows that this pattern persists
	across all tested dimensions: at $\sigma=10^{-2}$, BUP-NEWUOA
	achieves 76.5\%, 67.1\%, and 61.2\% success for
	$n=10,20,30$, versus 22.4\%, 7.1\%, and 3.5\% for NEWUOA.
	The heatmap in Figure~\ref{fig:noisy-main}(b) also shows that UOBYQA
	and Nelder--Mead deteriorate steadily as noise increases, with
	success dropping to 17.6\% and 19.2\% respectively at $\sigma=10^{-2}$.

	CMA-ES is the most noise-robust solver in our comparison: it
	retains 83--85\% success at $\sigma\in[10^{-4},10^{-2}]$ and
	still achieves 62.4\% at $\sigma=10^{-1}$, consistently
	surpassing BUP-NEWUOA at all noise levels
	$\sigma\ge10^{-4}$.
	This advantage has two main sources: (i)~CMA-ES aggregates a
	population of $\lambda\approx 4+\lfloor 3\ln n\rfloor$ evaluations
	per generation, so rank-based selection implicitly averages out
	per-evaluation noise; (ii)~its step-size adaptation (cumulative path)
	is inherently smoothed, whereas our hard-MAP interpolation
	incorporates each noisy evaluation directly into the model.
	Within the quadratic-interpolation methods in this comparison,
	BUP-NEWUOA performs best in these noisy tests and remains closer to
	CMA-ES at moderate noise levels ($\sigma\le 10^{-3}$).
	These results also indicate a limitation of hard interpolation under
	noise.
	The soft-MAP variant (Section~\ref{subsec:algo-soft}), which replaces exact
	interpolation by a penalized fit, is the natural next step for closing the gap
	between interpolation-based BUP-TR and population-based methods such as
	CMA-ES in noisy regimes.

	\begin{figure}[tp]
	\centering
	\begin{subfigure}[t]{0.68\textwidth}
		\centering
		\includegraphics[width=\textwidth]{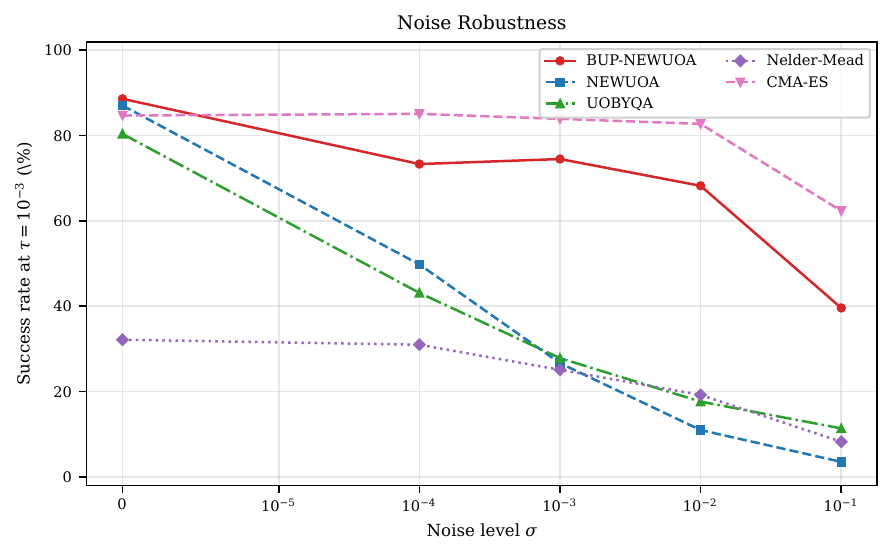}
		\caption{Success-rate degradation vs.\ $\sigma$.}
	\end{subfigure}
	\medskip

	\begin{subfigure}[t]{0.68\textwidth}
		\centering
		\includegraphics[width=\textwidth]{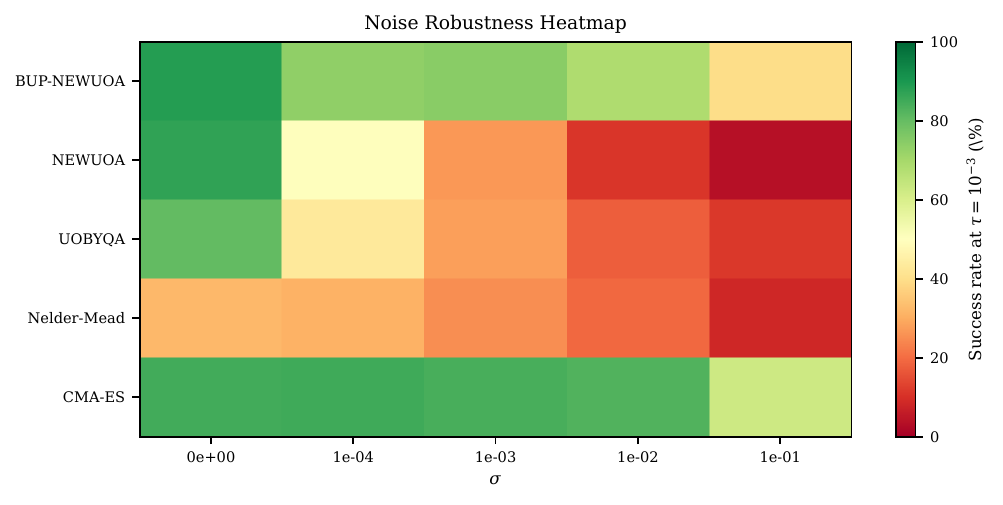}
		\caption{Success-rate heatmap (solver $\times$ $\sigma$).}
	\end{subfigure}
	\caption{Noisy benchmark: success rate at $\tau=10^{-3}$ under
	increasing noise.}\label{fig:noisy-main}
	\end{figure}

	\begin{figure}[tp]
	\centering
	\includegraphics[width=\textwidth]{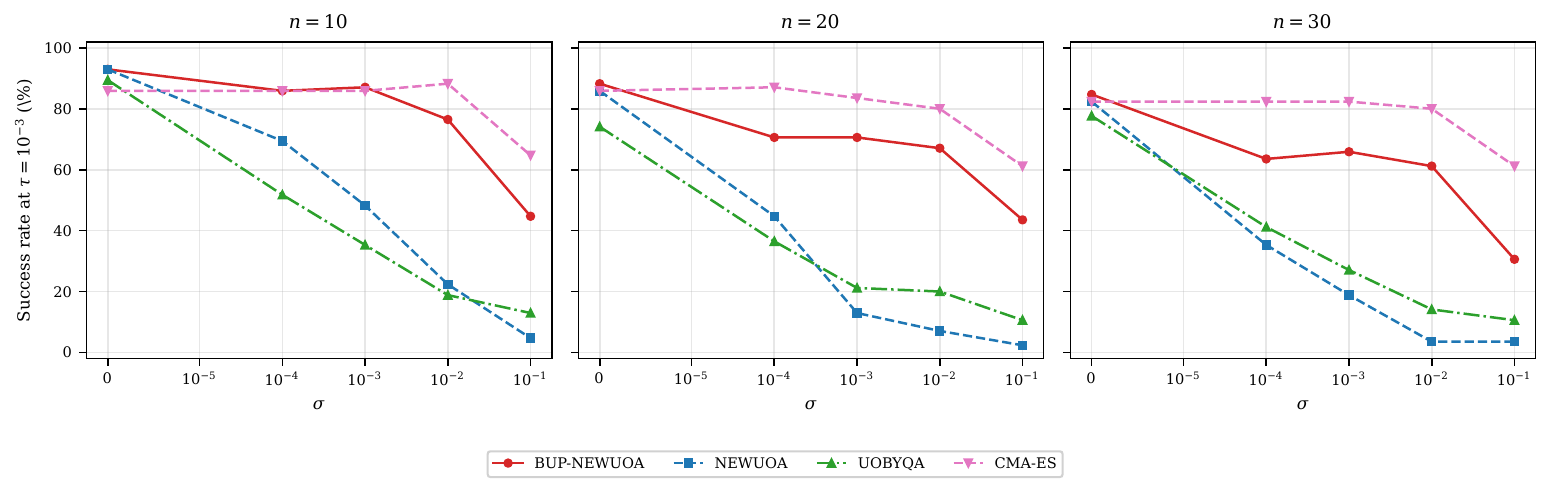}
	\caption{Success rate at $\tau=10^{-3}$ under noise, stratified by
	dimension.  Across the tested dimensions, BUP-NEWUOA remains the
	highest-performing quadratic-interpolation solver in this comparison.}\label{fig:noisy-by-dim}
	\end{figure}

	\begin{table}[t]
	\centering
	\caption{Noisy benchmark: success rate (\%) at $\tau=10^{-3}$
	with Wilson 95\% CIs (255 runs per cell).}\label{tab:noisy-summary}
	\footnotesize
	\setlength{\tabcolsep}{4.0pt}
	\begin{tabular}{rccc}
	\toprule
	$\sigma$ & BUP-NEWUOA & NEWUOA & CMA-ES \\
	\midrule
	0              & 88.6~[84.1, 92.0] & 87.1~[82.4, 90.6]
	               & 84.7~[79.8, 88.6] \\
	$10^{-4}$      & 73.3~[67.6, 78.4] & 49.8~[43.7, 55.9]
	               & 85.1~[80.2, 88.9] \\
	$10^{-3}$      & 74.5~[68.8, 79.5] & 26.7~[21.6, 32.4]
	               & 83.9~[78.9, 87.9] \\
	$10^{-2}$      & 68.2~[62.3, 73.6] & 11.0~[7.7, 15.4]
	               & 82.7~[77.6, 86.9] \\
	$10^{-1}$      & 39.6~[33.8, 45.7] &  3.5~[1.9, 6.6]
	               & 62.4~[56.3, 68.1] \\
	\bottomrule
	\end{tabular}
	\end{table}

	\noindent
	The noisy benchmark shows that the present hard-MAP construction
	continues to outperform NEWUOA and remains the
	strongest quadratic-interpolation method tested here in this regime, although
	CMA-ES is more noise-robust overall.
	
	\FloatBarrier


\section{Conclusion}\label{sec:concl}

We have introduced BUP-TR, a prior-regularized completion rule for
underdetermined quadratic interpolation in derivative-free trust-region
methods.  The method extends minimum-norm and minimum-change completions by
letting a prior provide both the reference model and the metric in the
completion problem.  In coefficient space, the hard-MAP model is the metric
projection of the prior mean onto the affine interpolation set.

The same completion system gives MAP-poisedness, a spectral condition used
to certify and repair interpolation geometry.  Under smoothness,
MAP-poisedness, bounded precision, and prior-accuracy assumptions, the
hard-MAP model is fully linear.  The standard trust-region decrease
argument then yields global first-order convergence and an
$\mathcal{O}(\varepsilon^{-2})$ evaluation-complexity bound that includes
repair evaluations.

The NEWUOA-style implementation BUP-NEWUOA improves fixed-budget accuracy on
the benchmark suite considered here, especially at stricter tolerances.  The
experiments indicate that previously evaluated values can be used
constructively in the completion metric and in the associated geometry test.
Future work includes probabilistic fully-linear theory and parameter selection
for soft-MAP variants in noisy settings, sharper rules for constructing the
precision matrix, and large-scale implementations with structured or low-rank
precision matrices.

\section*{Code availability}
The Python implementation of BUP-NEWUOA and all scripts needed to
reproduce the numerical experiments are available at
\url{https://github.com/huwei0121/BUPTR}.

\section*{Acknowledgments}
The work of W.~Hu, Y.-X.~Yuan, and L.~Zhang was supported in part by
NSFC and the Chinese Academy of Sciences.  The work of P.~Xie was
supported in part by the U.S.\ Department of Energy, Office of Science.


	\appendix
	
	\section*{Technical Appendices}
	\addcontentsline{toc}{section}{Technical Appendices}
	The appendices are split into two parts.
	Sections~\ref{app:norm}--\ref{app:repair-details} collect the
	technical derivations that support the main theory and algorithm,
	including coefficient scaling, closed-form MAP projections,
	fallback MAP-poisedness, surrogate-to-assumption bridges, full
	convergence proofs, and repair implementation details.
	The supplementary-material block that follows records benchmark
	definitions and parameter tables.
	
	\section{Scaled Quadratic Representation}\label{app:norm}
	
	This appendix records the scaled quadratic basis and the mapping between scaled coefficients and the
	original (unscaled) quadratic model. It also provides basic norm bounds used in Section~\ref{sec:conv}.
	
	\subsection{Quadratic basis and scaling}\label{app:norm-basis}
	
	Let $s\in\mathbb{R}^n$ and define the (unscaled) quadratic feature vector
	\begin{equation}\label{eq:app-phi}
		\phi(s)
		:=
		\begin{bmatrix}
			1\\
			s\\
			\qvec(s)
		\end{bmatrix}
		\in\mathbb{R}^q,
		\qquad
		q=\frac{(n+1)(n+2)}{2}.
	\end{equation}
	Let $\Delta>0$ and set the scaled coordinate $u:=s/\Delta$.
	Define the scaled feature vector
	\begin{equation}\label{eq:app-phihat}
		\widehat\phi(u)
		:=
		\begin{bmatrix}
			1\\
			u\\
			\qvec(u)
		\end{bmatrix}
		\in\mathbb{R}^q.
	\end{equation}
	
	\subsection{Coefficient mapping}\label{app:norm-map}
	
	A quadratic model is parameterized as
	\begin{equation}\label{eq:app-model}
		m(s)=c_0 + g^\top s + \frac12 s^\top H s,
		\qquad
		H=H^\top.
	\end{equation}
	Let the coefficient vector be
	\begin{equation}\label{eq:app-cvec}
		c:=
		\begin{bmatrix}
			c_0\\
			g\\
			\mathrm{vech}(H)
		\end{bmatrix}
		\in\mathbb{R}^q.
	\end{equation}
	Define the scaled coefficient vector
	\begin{equation}\label{eq:app-chat}
		\widehat c
		:=
		\begin{bmatrix}
			c_0\\
			\Delta g\\
			\Delta^2 \mathrm{vech}(H)
		\end{bmatrix}.
	\end{equation}
	Then, for any $s=\Delta u$,
	\begin{equation}\label{eq:app-feature-identity}
		m(\Delta u)
		=
		\phi(\Delta u)^\top c
		=
		\widehat\phi(u)^\top \widehat c.
	\end{equation}
	Proof.
	The constant term is unchanged. The linear term satisfies $g^\top(\Delta u)=(\Delta g)^\top u$.
	For the quadratic term,
	\[
	\frac12(\Delta u)^\top H(\Delta u)
	=\Delta^2 \cdot \frac12 u^\top H u
	=\Delta^2\,\qvec(u)^\top \mathrm{vech}(H)
	=\qvec(u)^\top \bigl(\Delta^2\mathrm{vech}(H)\bigr),
	\]
	where we used \eqref{eq:qvec-identity} with $s=u$.
	Collecting the three blocks yields \eqref{eq:app-feature-identity}. \hfill$\square$
	
	\subsection{Gradient/Hessian mapping and basic bounds}\label{app:norm-deriv}
	
	From \eqref{eq:app-model},
	\begin{equation}\label{eq:app-grad-hess}
		\nabla m(s)=g+Hs,
		\qquad
		\nabla^2 m(s)=H.
	\end{equation}
	In particular,
	\begin{equation}\label{eq:app-gh-at0}
		\nabla m(0)=g,
		\qquad
		\nabla^2 m(0)=H.
	\end{equation}
	From \eqref{eq:app-chat}, the mappings are
	\begin{equation}\label{eq:app-mapback}
		g=\Delta^{-1}\widehat c^{(g)},
		\qquad
		\mathrm{vech}(H)=\Delta^{-2}\widehat c^{(H)}.
	\end{equation}
	
	Bounded features on the unit ball.
	For $\|u\|_2\le 1$,
	\begin{align}
		\|\widehat\phi(u)\|_2^2
		&= 1 + \|u\|_2^2 + \|\qvec(u)\|_2^2 \nonumber\\
		&\le 1 + 1 + \frac12\|u\|_2^4 \label{eq:app-phihat-bound}\\
		&\le \frac52. \nonumber
	\end{align}
	The bound $\|\qvec(u)\|_2^2\le \frac12\|u\|_2^4$ follows by expanding
	$\|\qvec(u)\|_2^2=\frac14\sum_i u_i^4+\sum_{i<j}u_i^2u_j^2
	=\frac12\|u\|_2^4-\frac14\sum_i u_i^4 \le \frac12\|u\|_2^4$.
	Hence,
	\begin{equation}\label{eq:app-Bphi}
		\sup_{\|u\|_2\le 1}\|\widehat\phi(u)\|_2 \le \sqrt{\frac52}.
	\end{equation}
	
	From vech to operator norm (symmetric case).
	Let $H=H^\top$ and $v=\mathrm{vech}(H)$. Then
	\begin{equation}\label{eq:app-vech-to-op}
		\|H\|_2 \le \|H\|_F \le \sqrt{2}\,\|v\|_2.
	\end{equation}
	Proof.
	Write $\|H\|_F^2=\sum_i H_{ii}^2 + 2\sum_{i<j}H_{ij}^2
	\le 2\big(\sum_i H_{ii}^2 + \sum_{i<j}H_{ij}^2\big)=2\|v\|_2^2$.
	Then $\|H\|_2\le \|H\|_F\le \sqrt{2}\|v\|_2$. \hfill$\square$
	
	\section{Weighted Projection Formula for Hard-MAP Completion}\label{app:hardmap}
	
	This appendix derives the closed form of the hard-MAP estimator and proves the projector properties used
	in Lemma~\ref{lem:conv-FL-BUP} (nonexpansiveness in the $\widehat W$-norm).
	
	\subsection{KKT derivation of the projection formula}\label{app:hardmap-kkt}
	
	Fix $k$ and suppress the index. Let $\widehat A\in\mathbb{R}^{(m+1)\times q}$ and $b\in\mathbb{R}^{m+1}$.
	Given $\widehat c^\pi\in\mathbb{R}^q$ and $\widehat W\succ 0$, the hard-MAP problem is
	\begin{equation}\label{eq:app-hardmap-prob}
		\min_{\widehat c\in\mathbb{R}^q}\ \frac12\|\widehat c-\widehat c^\pi\|_{\widehat W}^2
		\quad\text{s.t.}\quad
		\widehat A\widehat c=b.
	\end{equation}
	The Lagrangian is
	\[
	\mathcal{L}(\widehat c,\lambda)=\frac12(\widehat c-\widehat c^\pi)^\top \widehat W(\widehat c-\widehat c^\pi)
	+\lambda^\top(\widehat A\widehat c-b),
	\]
	where $\lambda\in\mathbb{R}^{m+1}$.
	First-order KKT conditions are
	\begin{align}
		\nabla_{\widehat c}\mathcal{L}= \widehat W(\widehat c-\widehat c^\pi)+\widehat A^\top\lambda &=0, \label{eq:app-kkt1}\\
		\widehat A\widehat c-b &=0. \label{eq:app-kkt2}
	\end{align}
	From \eqref{eq:app-kkt1},
	\begin{equation}\label{eq:app-c-sol1}
		\widehat c=\widehat c^\pi-\widehat W^{-1}\widehat A^\top\lambda.
	\end{equation}
	Substitute into \eqref{eq:app-kkt2}:
	\[
	\widehat A\widehat c^\pi-\widehat A\widehat W^{-1}\widehat A^\top\lambda=b,
	\]
	hence
	\begin{equation}\label{eq:app-lam-sol}
		\lambda
		=
		\big(\widehat A\widehat W^{-1}\widehat A^\top\big)^{-1}(\widehat A\widehat c^\pi-b),
	\end{equation}
	provided $\widehat A\widehat W^{-1}\widehat A^\top$ is invertible.
	Plugging \eqref{eq:app-lam-sol} into \eqref{eq:app-c-sol1} yields
	\begin{equation}\label{eq:app-hardmap-closed}
		\widehat c
		=
		\widehat c^\pi
		+
		\widehat W^{-1}\widehat A^\top
		\big(\widehat A\widehat W^{-1}\widehat A^\top\big)^{-1}
		\big(b-\widehat A\widehat c^\pi\big),
	\end{equation}
	which is the formula used in Section~\ref{sec:bup}.
	
	\subsection{Weighted projection and nonexpansiveness}\label{app:hardmap-proj}
	
	Define
	\begin{equation}\label{eq:app-Mdef}
		\widehat M := \widehat A\widehat W^{-1}\widehat A^\top,
	\end{equation}
	and
	\begin{equation}\label{eq:app-Pdef}
		P := I - \widehat W^{-1}\widehat A^\top \widehat M^{-1}\widehat A.
	\end{equation}
	
	\begin{lemma}[Weighted projector formula]\label{lem:app-projector}
		If $\widehat M$ is invertible, then:
		\begin{enumerate}[leftmargin=2em]
			\item $P^2=P$ (idempotence).
			\item $\mathrm{Range}(P)=\mathrm{Null}(\widehat A)$ and $\mathrm{Null}(P)=\mathrm{Range}(\widehat W^{-1}\widehat A^\top)$.
			\item $P$ is $\widehat W$-self-adjoint: $\widehat W P = P^\top \widehat W$.
		\end{enumerate}
	\end{lemma}
	
	\begin{proof}
		(1) Idempotence.
		Compute
		\[
		P^2
		=
		\Big(I-\widehat W^{-1}\widehat A^\top \widehat M^{-1}\widehat A\Big)^2
		=
		I - 2\widehat W^{-1}\widehat A^\top \widehat M^{-1}\widehat A
		+ \widehat W^{-1}\widehat A^\top \widehat M^{-1}\widehat A\widehat W^{-1}\widehat A^\top \widehat M^{-1}\widehat A.
		\]
		Using $\widehat A\widehat W^{-1}\widehat A^\top=\widehat M$,
		\[
		\widehat A\widehat W^{-1}\widehat A^\top \widehat M^{-1} = I,
		\]
		hence the last term reduces to $\widehat W^{-1}\widehat A^\top \widehat M^{-1}\widehat A$, so $P^2=P$.
		
		(2) Range/Null.
		First, $\widehat A P = \widehat A - \widehat M \widehat M^{-1}\widehat A = 0$, hence $\mathrm{Range}(P)\subseteq \mathrm{Null}(\widehat A)$.
		Conversely, if $\widehat A v=0$, then $Pv=v$, so $v\in \mathrm{Range}(P)$, proving $\mathrm{Range}(P)=\mathrm{Null}(\widehat A)$.
		
		Also, $Pv=0$ iff $v\in \mathrm{Range}(\widehat W^{-1}\widehat A^\top)$, and the converse inclusion follows by direct substitution,
		so $\mathrm{Null}(P)=\mathrm{Range}(\widehat W^{-1}\widehat A^\top)$.
		
		(3) $\widehat W$-self-adjointness.
		Compute $P^\top = I - \widehat A^\top \widehat M^{-1}\widehat A \widehat W^{-1}$, hence
		$P^\top \widehat W = \widehat W - \widehat A^\top \widehat M^{-1}\widehat A = \widehat W P$.
	\end{proof}
	
	\begin{lemma}[Nonexpansiveness in $\widehat W$-norm]\label{lem:app-nonexp}
		If $\widehat M$ is invertible, then for all $v\in\mathbb{R}^q$,
		\begin{equation}\label{eq:app-nonexp}
			\|Pv\|_{\widehat W} \le \|v\|_{\widehat W}.
		\end{equation}
	\end{lemma}
	
	\begin{proof}
		By Lemma~\ref{lem:app-projector}(3), $P$ is self-adjoint under $\langle a,b\rangle_{\widehat W}:=a^\top \widehat W b$.
		By Lemma~\ref{lem:app-projector}(1), $P$ is an orthogonal projector in this inner product.
		Hence $v=Pv+(I-P)v$ with $\widehat W$-orthogonality, so
		$\|v\|_{\widehat W}^2 = \|Pv\|_{\widehat W}^2 + \|(I-P)v\|_{\widehat W}^2 \ge \|Pv\|_{\widehat W}^2$,
		which gives \eqref{eq:app-nonexp}.
	\end{proof}
	
	\section{MAP-Poisedness of the Fallback Set}\label{app:fallback}
	
	This appendix proves the MAP-poisedness guarantee of the fallback set
	$Y^{\mathrm{fb}}=\{x;\ x\pm \Delta e_i\}_{i=1}^n$ used in Subsection~\ref{sec:repair}.
	The proof is fully explicit and uses a symmetry-based eigen-decomposition.
	
	\subsection{Setup}\label{app:fallback-setup}
	
	Fix $x\in\mathbb{R}^n$ and $\Delta>0$. Consider scaled points
	\[
	u^{(0)}=0,\qquad u^{(+i)}=e_i,\qquad u^{(-i)}=-e_i,\qquad i=1,\ldots,n.
	\]
	Let $\widehat A\in\mathbb{R}^{(2n+1)\times q}$ be the scaled design matrix with rows
	$\widehat\phi(u^{(\cdot)})^\top$.
	
	Define the square submatrix $\widetilde A\in\mathbb{R}^{(2n+1)\times(2n+1)}$ obtained by selecting from $\widehat A$ only:
	\begin{itemize}[leftmargin=2em]
		\item the constant column,
		\item the $n$ linear columns $u_i$,
		\item the $n$ diagonal quadratic columns $\frac12 u_i^2$.
	\end{itemize}
	Then
	\begin{equation}\label{eq:app-AAT-psd}
		\widehat A\widehat A^\top
		=
		\widetilde A\widetilde A^\top + A_{\mathrm{extra}}A_{\mathrm{extra}}^\top
		\succeq
		\widetilde A\widetilde A^\top,
	\end{equation}
	hence
	\begin{equation}\label{eq:app-lmin-reduce}
		\lambda_{\min}(\widehat A\widehat A^\top)\ \ge\ \lambda_{\min}(\widetilde A\widetilde A^\top).
	\end{equation}
	
	Let $G:=\widetilde A\widetilde A^\top\in\mathbb{R}^{(2n+1)\times(2n+1)}$.
	We index rows by $\{0,+1,\ldots,+n,-1,\ldots,-n\}$.
	
	\subsection{Explicit Gram matrix entries}\label{app:fallback-entries}
	
	For the selected features, the row vectors are:
	\[
	r_0 = [1;\ 0;\ 0],\qquad
	r_{+i} = [1;\ e_i;\ \tfrac12 e_i],\qquad
	r_{-i} = [1;\ -e_i;\ \tfrac12 e_i].
	\]
	Thus $G_{ab}=r_a^\top r_b$. A direct computation gives:
	\begin{align}
		G_{00} &= 1, \label{eq:app-G00}\\
		G_{0,+i}=G_{0,-i} &= 1, \qquad \forall i, \label{eq:app-G0i}\\
		G_{+i,+i}=G_{-i,-i} &= 1+\|e_i\|^2 + \tfrac14\|e_i\|^2 = \tfrac94, \label{eq:app-Gdiag}\\
		G_{+i,-i} &= 1 + (e_i)^\top(-e_i) + \tfrac14\|e_i\|^2 = \tfrac14, \label{eq:app-Gpair}\\
		G_{+i,+j}=G_{+i,-j}=G_{-i,-j} &= 1,\qquad \forall i\ne j. \label{eq:app-Goff}
	\end{align}
	
	\subsection{Eigen-decomposition by invariant subspaces}\label{app:fallback-eigs}
	
	Define standard basis vectors $e_{0},e_{+i},e_{-i}\in\mathbb{R}^{2n+1}$.
	Introduce the following vectors for $i=1,\ldots,n$:
	\[
	d_i := e_{+i}-e_{-i},
	\qquad
	s_i := e_{+i}+e_{-i},
	\qquad
	S := \sum_{i=1}^n s_i.
	\]
	
	\begin{lemma}[Difference-subspace eigenvalue]\label{lem:app-eig2}
		For each $i=1,\ldots,n$, $G d_i = 2 d_i$. Hence $2$ is an eigenvalue with multiplicity at least $n$.
	\end{lemma}
	\begin{proof}
		Using \eqref{eq:app-G0i}--\eqref{eq:app-Goff},
		\[
		(G d_i)_0 = G_{0,+i}-G_{0,-i}=1-1=0.
		\]
		For any $j\ne i$,
		\[
		(G d_i)_{+j}=G_{+j,+i}-G_{+j,-i}=1-1=0,
		\qquad
		(G d_i)_{-j}=G_{-j,+i}-G_{-j,-i}=1-1=0.
		\]
		For the $+i$ and $-i$ components,
		\[
		(G d_i)_{+i}=G_{+i,+i}-G_{+i,-i}=\tfrac94-\tfrac14=2,
		\qquad
		(G d_i)_{-i}=G_{-i,+i}-G_{-i,-i}=\tfrac14-\tfrac94=-2.
		\]
		Thus $G d_i = 2(e_{+i}-e_{-i})=2 d_i$.
	\end{proof}
	
	\begin{lemma}[Sum-deviation subspace eigenvalue]\label{lem:app-eig-half}
		Let $w=\sum_{i=1}^n \alpha_i s_i$ satisfy $\sum_{i=1}^n \alpha_i=0$. Then $G w=\frac12 w$.
		Hence $\frac12$ is an eigenvalue with multiplicity at least $n-1$.
	\end{lemma}
	\begin{proof}
		First compute $G s_i$.
		From \eqref{eq:app-G0i}, \eqref{eq:app-Gdiag}, \eqref{eq:app-Gpair}, and \eqref{eq:app-Goff}:
		\[
		(G s_i)_0 = G_{0,+i}+G_{0,-i}=1+1=2,
		\]
		\[
		(G s_i)_{+i}=G_{+i,+i}+G_{+i,-i}=\tfrac94+\tfrac14=\tfrac52,
		\qquad
		(G s_i)_{-i}=G_{-i,+i}+G_{-i,-i}=\tfrac14+\tfrac94=\tfrac52,
		\]
		and for $j\ne i$,
		\[
		(G s_i)_{+j}=G_{+j,+i}+G_{+j,-i}=1+1=2,
		\qquad
		(G s_i)_{-j}=G_{-j,+i}+G_{-j,-i}=1+1=2.
		\]
		Therefore,
		\[
		G s_i
		=
		2 e_0 + \frac52 s_i + 2\sum_{j\ne i} s_j
		=
		2 e_0 + \frac52 s_i + 2(S-s_i)
		=
		2 e_0 + \frac12 s_i + 2S.
		\]
		Now take $w=\sum_i \alpha_i s_i$ with $\sum_i\alpha_i=0$. Then
		\begin{align*}
		G w
		&=
		\sum_i \alpha_i G s_i \\
		&=
		\sum_i \alpha_i\big(2e_0 + \tfrac12 s_i + 2S\big) \\
		&=
		2\Big(\sum_i \alpha_i\Big)e_0
		+ \tfrac12\sum_i \alpha_i s_i
		+ 2\Big(\sum_i \alpha_i\Big)S \\
		&=
		\tfrac12 w.
		\end{align*}
		The dimension of $\{(\alpha_i):\sum_i\alpha_i=0\}$ is $n-1$, giving multiplicity at least $n-1$.
	\end{proof}
	
	\begin{lemma}[Two-dimensional invariant block]\label{lem:app-eig2d}
		The subspace $\mathrm{span}\{e_0,S\}$ is $G$-invariant, and the restriction of $G$ to this subspace has matrix
		\[
		M=
		\begin{bmatrix}
			1 & 2n\\
			1 & 2n+\tfrac12
		\end{bmatrix}
		\quad\text{in the basis }\{e_0,S\}.
		\]
		Hence the remaining two eigenvalues are the roots of
		\begin{equation}\label{eq:app-poly}
			2\lambda^2 - (4n+3)\lambda + 1 = 0,
		\end{equation}
		namely
		\begin{equation}\label{eq:app-lpm}
			\lambda_{\pm} = \frac{(4n+3)\pm \sqrt{(4n+3)^2-8}}{4}.
		\end{equation}
	\end{lemma}
	\begin{proof}
		From \eqref{eq:app-G00}--\eqref{eq:app-G0i}, the first column gives
		\[
		G e_0 = e_0 + \sum_{i=1}^n (e_{+i}+e_{-i}) = e_0 + S.
		\]
		Also, summing the identity $G s_i = 2e_0 + \frac12 s_i + 2S$ over $i=1,\ldots,n$ yields
		\[
		G S = \sum_{i=1}^n G s_i = 2n\,e_0 + \frac12 S + 2n\,S = 2n\,e_0 + \left(2n+\frac12\right)S.
		\]
		Thus the restriction to $\mathrm{span}\{e_0,S\}$ has the stated matrix $M$.
		The characteristic polynomial is
		\[
		\det(M-\lambda I)
		=
		(1-\lambda)\left(2n+\tfrac12-\lambda\right) - 2n\cdot 1
		=
		\lambda^2 - \left(2n+\tfrac32\right)\lambda + \tfrac12,
		\]
		and multiplying by $2$ gives \eqref{eq:app-poly}. Solving yields \eqref{eq:app-lpm}.
	\end{proof}
	
	\begin{lemma}[Fallback Gram lower bound]\label{lem:app-lminG}
		Let $\lambda_-:=\lambda_{\min}$ of the two roots in \eqref{eq:app-lpm}. Then
		\begin{equation}\label{eq:app-lminG-bound}
			\lambda_{\min}(G) = \lambda_- \ge \frac{1}{4n+3}.
		\end{equation}
	\end{lemma}
	\begin{proof}
		By Lemmas~\ref{lem:app-eig2}, \ref{lem:app-eig-half}, and \ref{lem:app-eig2d}, the spectrum of $G$ is
		\[
		\{2\ \text{(mult. }n),\ \tfrac12\ \text{(mult. }n-1),\ \lambda_+,\ \lambda_-\}.
		\]
		For $n\ge 1$, one checks $\lambda_-<\tfrac12$, hence $\lambda_{\min}(G)=\lambda_-$.
		From \eqref{eq:app-lpm},
		\[
		\lambda_-=\frac{(4n+3)-\sqrt{(4n+3)^2-8}}{4}
		=
		\frac{2}{(4n+3)+\sqrt{(4n+3)^2-8}}.
		\]
		Since $\sqrt{(4n+3)^2-8}\le (4n+3)$, the denominator is at most $2(4n+3)$, hence
		\[
		\lambda_- \ge \frac{2}{2(4n+3)} = \frac{1}{4n+3}.
		\]
	\end{proof}
	
	\subsection{Fallback MAP-poisedness under precision bounds}\label{app:fallback-map}
	
	Suppose the precision satisfies $\widehat W\preceq w_{\max} I$, hence $\widehat W^{-1}\succeq \frac{1}{w_{\max}}I$.
	Then for hard-MAP,
	\[
	\widehat M^{\mathrm{hard}} = \widehat A\widehat W^{-1}\widehat A^\top
	\succeq
	\frac{1}{w_{\max}}\widehat A\widehat A^\top.
	\]
	Taking minimum eigenvalues and applying \eqref{eq:app-lmin-reduce} and Lemma~\ref{lem:app-lminG}:
	\begin{equation}\label{eq:app-mu0-final}
		\lambda_{\min}\!\left(\widehat M^{\mathrm{hard}}\right)
		\ge
		\frac{1}{w_{\max}}\lambda_{\min}(\widehat A\widehat A^\top)
		\ge
		\frac{1}{w_{\max}}\lambda_{\min}(G)
		\ge
		\frac{1}{w_{\max}}\cdot \frac{1}{4n+3}.
	\end{equation}
	This is exactly the fallback MAP-poisedness bound used in Lemma~\ref{lem:conv-fallback}.
	
	\section{Soft-MAP Completion with Exact Center Interpolation}\label{app:softmap}
	
	This appendix derives the closed form for the constrained ridge (soft-MAP) problem that enforces $m_k(0)=f(x_k)$ exactly,
	as stated in \eqref{eq:app-softmap-prob}.
	
	\subsection{Problem statement}\label{app:softmap-prob}
	
	Fix iteration $k$ and suppress indices. Let the center row be $a_0^\top:=\widehat\phi(0)^\top$.
	Let $\widehat A_{\mathrm{nc}}\in\mathbb{R}^{m\times q}$ and $b_{\mathrm{nc}}\in\mathbb{R}^m$ denote the non-center rows/values.
	Let $R_{\mathrm{nc}}\succ 0$ be the observation covariance on these $m$ points, and $\widehat W\succ 0$ be the precision matrix.
	Given prior mean $\widehat c^\pi$, the constrained soft-MAP problem is:
	\begin{equation}\label{eq:app-softmap-prob}
		\min_{\widehat c\in\mathbb{R}^q}\ 
		\frac12\|\widehat A_{\mathrm{nc}}\widehat c-b_{\mathrm{nc}}\|_{R_{\mathrm{nc}}^{-1}}^2
		+\frac12\|\widehat c-\widehat c^\pi\|_{\widehat W}^2
		\quad \text{s.t.}\quad
		a_0^\top \widehat c = f(x_k).
	\end{equation}
	
	\subsection{Closed form via KKT and Schur complement}\label{app:softmap-kkt}
	
	Define
	\begin{equation}\label{eq:app-soft-K}
		K := \widehat A_{\mathrm{nc}}^\top R_{\mathrm{nc}}^{-1}\widehat A_{\mathrm{nc}} + \widehat W
		\quad\in\mathbb{R}^{q\times q},
		\qquad
		r := \widehat A_{\mathrm{nc}}^\top R_{\mathrm{nc}}^{-1} b_{\mathrm{nc}} + \widehat W \widehat c^\pi.
	\end{equation}
	Then $K\succ 0$ because $\widehat W\succ 0$ and $\widehat A_{\mathrm{nc}}^\top R_{\mathrm{nc}}^{-1}\widehat A_{\mathrm{nc}}\succeq 0$.
	
	The Lagrangian with multiplier $\nu\in\mathbb{R}$ is
	\[
	\mathcal{L}(\widehat c,\nu)
	=
	\frac12(\widehat A_{\mathrm{nc}}\widehat c-b_{\mathrm{nc}})^\top R_{\mathrm{nc}}^{-1}(\widehat A_{\mathrm{nc}}\widehat c-b_{\mathrm{nc}})
	+\frac12(\widehat c-\widehat c^\pi)^\top \widehat W(\widehat c-\widehat c^\pi)
	+\nu(a_0^\top \widehat c - f(x_k)).
	\]
	KKT conditions are
	\begin{align}
		\nabla_{\widehat c}\mathcal{L}
		&=
		\widehat A_{\mathrm{nc}}^\top R_{\mathrm{nc}}^{-1}(\widehat A_{\mathrm{nc}}\widehat c-b_{\mathrm{nc}})
		+\widehat W(\widehat c-\widehat c^\pi)
		+\nu a_0
		= 0, \label{eq:app-soft-kkt1}\\
		a_0^\top \widehat c &= f(x_k). \label{eq:app-soft-kkt2}
	\end{align}
	Rearrange \eqref{eq:app-soft-kkt1} to obtain
	\begin{equation}\label{eq:app-soft-linear}
		K\widehat c + \nu a_0 = r,
		\qquad\text{so}\qquad
		\widehat c = K^{-1}(r-\nu a_0).
	\end{equation}
	Plugging into \eqref{eq:app-soft-kkt2} yields a scalar equation for $\nu$:
	\[
	a_0^\top K^{-1}(r-\nu a_0) = f(x_k)
	\quad\Longrightarrow\quad
	\nu = \frac{a_0^\top K^{-1} r - f(x_k)}{a_0^\top K^{-1} a_0}.
	\]
	Therefore, the closed form solution is
	\begin{equation}\label{eq:app-soft-closed}
		\widehat c
		=
		K^{-1}\left(r - a_0\cdot
		\frac{a_0^\top K^{-1} r - f(x_k)}{a_0^\top K^{-1} a_0}\right),
	\end{equation}
	which is unique since $K\succ 0$ implies $a_0^\top K^{-1}a_0>0$.
	
	\section{Sufficient Conditions for Prior Accuracy}\label{app:gp-to-A6}
	
This appendix provides a bridge from surrogate accuracy at
the center to the prior-accuracy condition in
Assumption~\ref{asmp:conv-prior}, where the error is measured in the $\widehat W_k$-norm.
Any probabilistic justification for a GP or another surrogate may then
be layered on top of this lemma without altering the
trust-region analysis developed in the main text.

\subsection{A prior-accuracy bridge}\label{app:gp-bridge}

Suppose the prior mean is chosen as in \eqref{eq:bup-prior-gp}, with the exact center value and derivative posterior moments
of a surrogate mean $\mu_k(\cdot)$ at $x_k$:
\begin{equation}\label{eq:app-prior-taylor}
	\widehat c_k^\pi
	:=
	\begin{bmatrix}
		f(x_k)\\
		\Delta_k \nabla \mu_k(x_k)\\
		\Delta_k^2\,\mathrm{vech}\!\big(\nabla^2 \mu_k(x_k)\big)
	\end{bmatrix}.
\end{equation}
Recall $\widehat c_k^\ell=[f(x_k);\ \Delta_k\nabla f(x_k);\ 0]$ and the
block structure $\widehat W_k=\mathrm{diag}(w_0,W_{g,k},W_{H,k})$ from \eqref{eq:bup-W-block}.
The precision matrix is constructed by clipping the diagonal precision entries to a fixed interval
(the explicit formula is given in~\eqref{eq:bup-W-gp} below).

\begin{lemma}[A sufficient condition for prior accuracy]\label{lem:app-A6-sufficient}
	Suppose there exist constants $a_1,a_2\ge 0$ such that for all $k$,
	\begin{equation}\label{eq:app-sur-err}
		\|\nabla \mu_k(x_k)-\nabla f(x_k)\|_2 \le a_1 \Delta_k,\qquad
		\|\nabla^2 \mu_k(x_k)\|_F \le a_2.
	\end{equation}
	Define
	$\bar w_g:=\sup_k \|W_{g,k}\|_\infty$ and
	$\bar w_H:=\sup_k \|W_{H,k}\|_\infty$
	(both bounded by $w_{\max}$ via \eqref{eq:bup-W-gp}).
	Then Assumption~\ref{asmp:conv-prior} holds with
	\begin{equation}\label{eq:app-kpi}
		\bar\kappa_\pi = \sqrt{\bar w_g\,a_1^2 + \bar w_H\,a_2^2}.
	\end{equation}
\end{lemma}

\begin{proof}
	From \eqref{eq:app-prior-taylor}, the constant block cancels:
	\[
	d_k:=\widehat c_k^\pi-\widehat c_k^\ell
	=
	\begin{bmatrix}
		0\\
		\Delta_k(\nabla \mu_k(x_k)-\nabla f(x_k))\\
		\Delta_k^2\,\mathrm{vech}(\nabla^2 \mu_k(x_k))
	\end{bmatrix}.
	\]
	Since $\widehat W_k$ is diagonal with $w_0$ on the constant entry,
	$W_{g,k}$ on the gradient block, and $W_{H,k}$ on the Hessian block:
	\begin{align*}
		\|d_k\|_{\widehat W_k}^2
		&=
		w_0\cdot 0^2
		+\sum_{i=1}^{n}[W_{g,k}]_{ii}\,\Delta_k^2|\nabla_i \mu_k-\nabla_i f|^2
		+ \sum_{j}[W_{H,k}]_{jj}\,\Delta_k^4|[\mathrm{vech}(\nabla^2\mu_k)]_j|^2\\
		&\le
		\bar w_g\,\Delta_k^2\|\nabla \mu_k-\nabla f\|_2^2
		+ \bar w_H\,\Delta_k^4\|\mathrm{vech}(\nabla^2 \mu_k)\|_2^2\\
		&\le
		\bar w_g\,a_1^2\,\Delta_k^4
		+ \bar w_H\,\Delta_k^4\|\nabla^2 \mu_k\|_F^2
		\qquad(\text{since }\|\mathrm{vech}(S)\|_2\le \|S\|_F)\\
		&\le
		(\bar w_g\,a_1^2 + \bar w_H\,a_2^2)\,\Delta_k^4.
	\end{align*}
	Taking square roots yields \eqref{eq:app-kpi}.
\end{proof}

\begin{remark}[Role of the precision metric]
	Because $\bar w_H$ can be much smaller than $w_{\max}$ when the GP
	Hessian posterior is uncertain (diagonal entries of $W_{H,k}$ near
	$w_{\min}$), the constant $\bar\kappa_\pi$ is generically
	tighter than $\sqrt{w_{\max}}\sqrt{a_1^2+a_2^2}$, which would
	result from a Euclidean-norm assumption under identical
	surrogate accuracy.
	Thus coefficient blocks assigned low precision may be less accurate without inflating the constant as much as they would under a Euclidean-norm requirement.
\end{remark}

Consequently, whenever the surrogate estimates satisfy \eqref{eq:app-sur-err},
either deterministically or with high probability,
Assumption~\ref{asmp:conv-prior} follows from Lemma~\ref{lem:app-A6-sufficient}.
The convergence argument in Section~\ref{sec:conv} then applies without
modification.

\section{Gaussian-Process Prior Construction}\label{app:gp-details}

This appendix provides the complete construction of the GP-based
prior mean~$\widehat c_k^\pi$ and precision~$\widehat W_k$
summarized in Section~\ref{subsec:bup-prior-mean}.

\subsection*{GP posterior at the trust-region center}
Let $\mu_k(\cdot)$ denote the GP posterior mean fitted to the dataset
\[
\mathcal{D}_k :=
\{(x,f(x)):\ x \text{ has been evaluated up to iteration } k\},
\]
subsampled to the $N_{\mathrm{pool}}$ points nearest to~$x_k$.
We compute
\[
\mu_{g,k}:=\nabla \mu_k(x_k)\in\R^n,\qquad
\mu_{H,k}:=\nabla^2 \mu_k(x_k)\in\mathbb{S}^n,
\]
together with the posterior covariance of the stacked derivative vector
\[
z_k := \begin{bmatrix}\nabla \tilde f(x_k)\\ \mathrm{vech}(\nabla^2 \tilde f(x_k))\end{bmatrix}\in\R^{n+q_H},
\qquad
\Sigma_{z,k} := \mathrm{Cov}(z_k\mid \mathcal{D}_k)\in\R^{(n+q_H)\times(n+q_H)},
\]
where $\tilde f$ denotes the latent GP function under the posterior.
For a sufficiently smooth kernel (e.g., Mat\'ern with $\nu\ge 3$ or
squared-exponential), all moments are available in closed form.
Hyperparameters are re-estimated every $K_{\mathrm{hp}}=\max(20,2n)$
iterations via marginal log-likelihood.

Normalization to the coefficient scale.
The scaling matrix
\[
D_k := \mathrm{diag}(\Delta_k I_n,\, \Delta_k^2 I_{q_H})
\]
maps derivative moments to the scaled coefficient basis.
The prior mean~\eqref{eq:bup-prior-gp} then reads
\[
\widehat c_k^\pi =
\bigl[f(x_k);\, D_k[\mu_{g,k};\mathrm{vech}(\mu_{H,k})]\bigr].
\]

Diagonal precision from posterior variance.
Define the scaled covariance
$\Sigma_{\widehat z,k} := D_k\,\Sigma_{z,k}\,D_k$
and let $v=[v_g;v_H]$ collect its diagonal entries:
\[
(v_g)_i = \Delta_k^2\,\mathrm{Var}[\partial_i \tilde f(x_k)\mid \mathcal{D}_k],
\qquad
(v_H)_j = \Delta_k^4\,\mathrm{Var}[(\mathrm{vech}(\nabla^2 \tilde f(x_k)))_j\mid \mathcal{D}_k].
\]
We set
\begin{equation}\label{eq:bup-W-gp}
\widehat W_k
:=
\mathrm{diag}\!\big(
w_0,\;
\clip_{[w_{\min},w_{\max}]}(v_g^{-1}),\;
\clip_{[w_{\min},w_{\max}]}(v_H^{-1})
\big),
\end{equation}
where $v_g^{-1},v_H^{-1}$ denote entry-wise reciprocals
(with $1/0:=+\infty$, clipped to $w_{\max}$),
and $w_0=w_{\max}$.
Higher posterior uncertainty yields smaller precision (weaker pull
toward the prior mean).
When cross-covariance structure matters, one may replace
\eqref{eq:bup-W-gp} with the full spectral clipping of
$\Sigma_{\widehat z,k}^{-1}$; eigenvalue clipping keeps all
eigenvalues in $[w_{\min},w_{\max}]$, so~\eqref{eq:bup-W-bounds}
holds and the theory applies unchanged.

Computational cost.
The dominant per-iteration costs are $\mathcal{O}(N_{\mathrm{pool}}^3)$ for
the GP Cholesky factorization and $\mathcal{O}(n^2 N_{\mathrm{pool}}^2)$ for
derivative posterior moments; the MAP system adds $\mathcal{O}(m^3)$
with $m=2n$.
For $n\le 20$ and $N_{\mathrm{pool}}\lesssim 200$, the per-iteration
overhead is usually secondary relative to an expensive function evaluation.
Beyond $n\approx 20$, the $\mathcal{O}(n^4)$ Hessian-block covariance
becomes the bottleneck; mitigations include restricting the precision
to the gradient block ($W_{H,k}=w_{\min}I$) or sparse GP
approximations.

\section{Proofs of Convergence Results}\label{app:proofs}

This appendix collects the complete proofs deferred from
Section~\ref{sec:conv}.

\subsection{Proof of Lemma~\ref{lem:conv-FL-BUP}}\label{app:proof-FL}

\begin{proof}
	Fix $k$ and write $\Delta:=\Delta_k$, $\widehat A:=\widehat A_k$,
	$\widehat W:=\widehat W_k$, and
	$\widehat M:=\widehat A\widehat W^{-1}\widehat A^\top$.
	Let $\widehat c^\pi:=\widehat c_k^\pi$, $\widehat c:=\widehat c_k$, and $\widehat c^\ell:=\widehat c_k^\ell$.
	
	Step 1: interpolation residual for the linear Taylor model.
	Let $b\in\R^{m+1}$ collect the sampled values, and define the residual vector
	\[
	r:=b-\widehat A\widehat c^\ell.
	\]
	Componentwise, $r_i=f(x_k+s^{(i)})-\ell_k(s^{(i)})$. By Lipschitz gradient (Assumption~\ref{asmp:conv-smooth})
	and $\|s^{(i)}\|_2\le\ctrim\Delta$,
	\[
	|r_i|\le \frac{L_g}{2}\|s^{(i)}\|_2^2 \le \frac{L_g}{2}\ctrim^2\Delta^2,
	\qquad
	\|r\|_2\le \sqrt{m+1}\,\frac{L_g}{2}\ctrim^2\Delta^2.
	\]
	
	Step 2: coefficient error bound.
	Using the closed form \eqref{eq:bup-hard-closed}, and adding and
	subtracting $\widehat A\widehat c^\ell$, gives
	\begin{equation}\label{eq:conv-e-decomp}
		e:=\widehat c-\widehat c^\ell
		=
		P(\widehat c^\pi-\widehat c^\ell)+\widehat W^{-1}\widehat A^\top \widehat M^{-1}r,
		\qquad
		P:=I-\widehat W^{-1}\widehat A^\top \widehat M^{-1}\widehat A.
	\end{equation}
	By Lemma~\ref{lem:app-projector} (Appendix~\ref{app:hardmap}),
	$P$ is the $\widehat W$-orthogonal projector onto
	$\mathrm{Null}(\widehat A)$, and Lemma~\ref{lem:app-nonexp} gives
	$\|Pv\|_{\widehat W}\le \|v\|_{\widehat W}$ for all $v$.
	Moreover, $w_{\min}I\preceq \widehat W$
	(Assumption~\ref{asmp:conv-Wbounds}) implies
	$\|Pv\|_2 \le w_{\min}^{-1/2}\|Pv\|_{\widehat W}$.
	Applying these bounds to
	$v=\widehat c^\pi-\widehat c^\ell$ and then using
	Assumption~\ref{asmp:conv-prior} yields
	\[
	\|P(\widehat c^\pi-\widehat c^\ell)\|_2
	\le \frac{1}{\sqrt{w_{\min}}}
	\|P(\widehat c^\pi-\widehat c^\ell)\|_{\widehat W}
	\le \frac{1}{\sqrt{w_{\min}}}
	\|\widehat c^\pi-\widehat c^\ell\|_{\widehat W}
	\le \frac{\bar\kappa_\pi}{\sqrt{w_{\min}}}\,\Delta^2.
	\]
	For the second term in \eqref{eq:conv-e-decomp}, use
	$\|\widehat W^{-1}\|_2\le 1/w_{\min}$, $\|\widehat M^{-1}\|_2\le 1/\mu_M$ (Assumption~\ref{asmp:conv-mappoised}), and
	$\|\widehat A\|_2\le \sqrt{m+1}\Bgeo$ from \eqref{eq:conv-Bphi-geo} to obtain
	\[
	\|\widehat W^{-1}\widehat A^\top \widehat M^{-1}r\|_2
	\le \frac{1}{w_{\min}}\cdot \|\widehat A\|_2\cdot \frac{1}{\mu_M}\cdot \|r\|_2
	\le \frac{\ctrim^2\,\Bgeo\, L_g(m+1)}{2w_{\min}\mu_M}\,\Delta^2.
	\]
	Combining the two bounds gives $\|e\|_2\le \kappa_e\Delta^2$ with $\kappa_e$ in \eqref{eq:conv-ke}.
	
	Step 3: Hessian bound.
	The Hessian block of $\widehat c^\ell$ is zero. Let $e^{(H)}$ be the Hessian block of $e$.
	Mapping back by \eqref{eq:bup-mapback} gives $\mathrm{vech}(H_k)=\Delta^{-2}e^{(H)}$.
	Using $\|H\|_2\le \|H\|_F\le \sqrt{2}\|\mathrm{vech}(H)\|_2$
	(Appendix~\ref{app:norm}) and $\|e\|_2\le\kappa_e\Delta^2$,
	\[
	\|H_k\|_2
	\le \sqrt{2}\,\Delta^{-2}\|e^{(H)}\|_2
	\le \sqrt{2}\,\Delta^{-2}\|e\|_2
	\le \sqrt{2}\kappa_e.
	\]
	
	Step 4: function and gradient errors.
	We evaluate errors for $\|s\|_2\le\Delta$, equivalently
	$\|u\|_2\le 1$ with $s=\Delta u$.
	The unit-ball feature bound $B_\phi$ applies here, whereas the
	geo-ball bound $\Bgeo$ was used earlier for rows of the design
	matrix.
	Let $s=\Delta u$ with $\|u\|_2\le 1$. Then
	\[
	|f(x_k+s)-m_k(s)|
	\le |f(x_k+s)-\ell_k(s)| + |\ell_k(s)-m_k(s)|.
	\]
	The first term is bounded by $(L_g/2)\Delta^2$ by Assumption~\ref{asmp:conv-smooth}. For the second term,
	$m_k(\Delta u)-\ell_k(\Delta u)=\widehat\phi(u)^\top e$, hence
	$|\ell_k(s)-m_k(s)|\le \|\widehat\phi(u)\|_2\|e\|_2 \le B_\phi \kappa_e\Delta^2$.
	This proves \eqref{eq:conv-FL1} with $\kappa_f$ in \eqref{eq:conv-kappas}.
	
	For the gradient bound, for $\|s\|_2\le \Delta$,
	\[
	\|\nabla f(x_k+s)-\nabla m_k(s)\|_2
	\le \|\nabla f(x_k+s)-\nabla f(x_k)\|_2 + \|\nabla f(x_k)-g_k\|_2 + \|H_k s\|_2.
	\]
	The first term is $\le L_g\Delta$. The second term equals
	$\|\nabla f(x_k)-g_k\|_2=\Delta^{-1}\|e^{(g)}\|_2\le \Delta^{-1}\|e\|_2\le \kappa_e\Delta$.
	The third term is $\le \|H_k\|_2\Delta \le (\sqrt{2}\kappa_e)\Delta$.
	This proves \eqref{eq:conv-FL2} with $\kappa_g$ in \eqref{eq:conv-kappas}.
\end{proof}

\subsection{Proofs of Lemmas~\ref{lem:conv-aredpred}--\ref{lem:conv-crit-exit}}\label{app:proof-accept}

\begin{proof}[Proof of Lemma~\ref{lem:conv-aredpred}]
Since $m_k(0)=f(x_k)$, we have $\ared_k-\pred_k = m_k(s_k)-f(x_k+s_k)$, and \eqref{eq:conv-aredpred} follows from
\eqref{eq:conv-FL1}.
\end{proof}

\begin{proof}[Proof of Lemma~\ref{lem:conv-succ-small}]
By Lemma~\ref{lem:conv-aredpred}, $\ared_k\ge \pred_k-\kappa_f\Delta_k^2$, hence
$\rho_k\ge 1-\kappa_f\Delta_k^2/\pred_k$. Condition \eqref{eq:conv-delta-thr} yields $\rho_k\ge \eta_1$.
\end{proof}

\begin{proof}[Proof of Lemma~\ref{lem:conv-succ-eps}]
From \eqref{eq:conv-FL2} at $s=0$, $\|g_k-\nabla f(x_k)\|_2\le \kappa_g\Delta_k\le \varepsilon/2$, hence $\|g_k\|_2\ge \varepsilon/2$.
By Assumption~\ref{asmp:conv-cauchy} and $\Delta_k\le \varepsilon/(2H_{\max})$,
\[
\pred_k
\ge \frac12\|g_k\|_2\min\!\left\{\Delta_k,\frac{\|g_k\|_2}{\|H_k\|_2}\right\}
\ge \frac12\cdot\frac{\varepsilon}{2}\cdot \Delta_k
= \frac{\varepsilon}{4}\Delta_k.
\]
Lemma~\ref{lem:conv-aredpred} gives $\rho_k\ge 1-\kappa_f\Delta_k^2/\pred_k \ge 1-(4\kappa_f\Delta_k/\varepsilon)$.
If $\Delta_k\le \frac{1-\eta_1}{4\kappa_f}\varepsilon$, then $\rho_k\ge \eta_1$.
\end{proof}

\begin{proof}[Proof of Lemma~\ref{lem:conv-crit-exit}]
From \eqref{eq:conv-FL2} at $s=0$, $\|g_k-\nabla f(x_k)\|_2\le \kappa_g\Delta_k$.
Since $\Delta_k\le \varepsilon/(4\kappa_g)$, we have $\kappa_g\Delta_k\le \varepsilon/4$, and therefore
$\|g_k\|_2\ge \|\nabla f(x_k)\|_2-\kappa_g\Delta_k\ge \varepsilon-\varepsilon/4=3\varepsilon/4$.
Since $\Delta_k\le \varepsilon/(2\kappa_\Delta)$, we obtain $\kappa_\Delta\Delta_k\le \varepsilon/2 < 3\varepsilon/4 \le \|g_k\|_2$.
\end{proof}

\subsection{Proof of Theorem~\ref{thm:conv-global}}\label{app:proof-global}

\begin{proof}
	For contradiction, suppose that there exist $\varepsilon\in(0,1]$ and $k_0$ such that $\|\nabla f(x_k)\|_2\ge \varepsilon$ for all $k\ge k_0$.
	Define the combined threshold
	\[
	\bar\Delta_\varepsilon
	:= \min\!\big\{\Delta_{\mathrm{succ}}(\varepsilon),\,\Delta_{\mathrm{crit}}(\varepsilon)\big\}
	= \min\!\left\{\frac{\varepsilon}{4\kappa_g},\,\frac{\varepsilon}{2H_{\max}},\,\frac{1-\eta_1}{4\kappa_f}\varepsilon,\,
	  \frac{\varepsilon}{2\kappa_\Delta},\,1\right\},
	\]
	where $\Delta_{\mathrm{succ}}(\varepsilon)$ is from Lemma~\ref{lem:conv-succ-eps} and $\Delta_{\mathrm{crit}}(\varepsilon)$ from
	Lemma~\ref{lem:conv-crit-exit}.
	
	Define $\underline\Delta_\varepsilon:=\min\{\Delta_{k_0},\gamma_{\mathrm{dec}}\bar\Delta_\varepsilon\}$. We claim $\Delta_k\ge \underline\Delta_\varepsilon$ for all $k\ge k_0$.
	To see this, suppose $\Delta_k\le \bar\Delta_\varepsilon$ at some $k\ge k_0$. Then:
	\begin{enumerate}[leftmargin=2em,topsep=2pt,itemsep=1pt]
	\item By Lemma~\ref{lem:conv-crit-exit} (since $\|\nabla f(x_k)\|_2\ge\varepsilon$ and $\Delta_k\le\Delta_{\mathrm{crit}}(\varepsilon)$),
	the criticality loop exits immediately without further reducing $\Delta_k$.
	\item By Lemma~\ref{lem:conv-succ-eps} (since $\Delta_k\le\Delta_{\mathrm{succ}}(\varepsilon)$), the trust-region step is successful,
	so $\Delta_{k+1}\ge\Delta_k$.
	\end{enumerate}
	Conversely, if $\Delta_k>\bar\Delta_\varepsilon$ and the iteration is unsuccessful, then
	$\Delta_{k+1}=\gamma_{\mathrm{dec}}\Delta_k>\gamma_{\mathrm{dec}}\bar\Delta_\varepsilon\ge \underline\Delta_\varepsilon$.
	In all cases $\Delta_{k+1}\ge\underline\Delta_\varepsilon$; induction on $k$ proves the stated lower bound.
	
	By the criticality safeguard (Assumption~\ref{asmp:conv-crit}), every iteration that proceeds to compute a step satisfies
	$\|g_k\|_2>\kappa_\Delta\Delta_k$. Together with Assumption~\ref{asmp:conv-cauchy} and $\|H_k\|_2\le H_{\max}$,
	\[
	\pred_k
	\ge \frac12\|g_k\|_2 \min\!\left\{\Delta_k,\frac{\|g_k\|_2}{\|H_k\|_2}\right\}
	\ge \frac12\kappa_\Delta\Delta_k \min\!\left\{\Delta_k,\frac{\kappa_\Delta\Delta_k}{H_{\max}}\right\}
	= c_{\mathrm{pred}}\Delta_k^2,
	\]
	where $c_{\mathrm{pred}}:=\frac12\kappa_\Delta\min\{1,\kappa_\Delta/H_{\max}\}>0$.
	On successful iterations, $\ared_k\ge \eta_1\pred_k$, hence
	$f(x_{k+1})\le f(x_k)-\eta_1 c_{\mathrm{pred}}\Delta_k^2\le f(x_k)-\eta_1 c_{\mathrm{pred}}\underline\Delta_\varepsilon^{\,2}$.
	
	Finally, there must be infinitely many successful iterations. Otherwise, after the last successful iteration, the iterates
	would remain at a fixed point while $\Delta_k$ decreases geometrically toward $0$. Once
	$\Delta_k\le \bar\Delta_\varepsilon$, Lemma~\ref{lem:conv-crit-exit} prevents further criticality shrinkage
	and Lemma~\ref{lem:conv-succ-eps} forces a successful step,
	contradicting the assumption. Therefore, $f$ decreases by a fixed positive amount
	infinitely often, contradicting the lower bound $f\ge f_{\inf}$.
\end{proof}

\subsection{Proof of Theorem~\ref{thm:conv-complexity}}\label{app:proof-complexity}

\begin{proof}
	Fix $\varepsilon\in(0,1]$.  By Theorem~\ref{thm:conv-global}, such indices exist for every $\varepsilon\in(0,1]$; let $K$ be the first index such that $\|\nabla f(x_K)\|_2\le \varepsilon$.
	For all $k<K$, we have $\|\nabla f(x_k)\|_2>\varepsilon$.
	
As in the proof of Theorem~\ref{thm:conv-global}, define
\[
\bar\Delta_\varepsilon:=\min\!\big\{\Delta_{\mathrm{succ}}(\varepsilon),\,\Delta_{\mathrm{crit}}(\varepsilon)\big\},
\qquad
\underline\Delta_\varepsilon:=\min\!\big\{\Delta_0,\,\gamma_{\mathrm{dec}}\bar\Delta_\varepsilon\big\}.
\]
By the same argument (using Lemmas~\ref{lem:conv-succ-eps} and \ref{lem:conv-crit-exit}),
$\Delta_k\ge \underline\Delta_\varepsilon$ for all $k<K$.
	
	Let $\mathcal{S}$ denote the set of successful indices in $\{0,1,\ldots,K-1\}$ and set $N_s:=|\mathcal{S}|$.
	By Lemma~\ref{lem:conv-FL-BUP}, $\|H_k\|_2\le H_{\max}$ uniformly. By the criticality safeguard,
	$\|g_k\|_2>\kappa_\Delta\Delta_k$ on iterations that compute a step. Therefore, for $k\in\mathcal{S}$,
	\[
	f(x_k)-f(x_{k+1})=\ared_k \ge \eta_1\pred_k \ge \eta_1 c_{\mathrm{pred}}\Delta_k^2
	\ge \eta_1 c_{\mathrm{pred}}\,\underline\Delta_\varepsilon^{\,2},
	\]
	where $c_{\mathrm{pred}}:=\frac12\kappa_\Delta\min\{1,\kappa_\Delta/H_{\max}\}>0$.
	Summing over $k\in\mathcal{S}$ and using $f\ge f_{\inf}$ yields
	\[
	f(x_0)-f_{\inf} \ge N_s\,\eta_1 c_{\mathrm{pred}}\,\underline\Delta_\varepsilon^{\,2}.
	\]
	
	Since $\bar\Delta_\varepsilon\ge c_\Delta\,\varepsilon$ with
	\[
	c_\Delta:=\min\left\{\frac{1}{4\kappa_g},\ \frac{1}{2H_{\max}},\ \frac{1-\eta_1}{4\kappa_f},\ \frac{1}{2\kappa_\Delta},\ 1\right\},
	\]
	we have
	\[
	\underline\Delta_\varepsilon
	=\min\{\Delta_0,\gamma_{\mathrm{dec}}\bar\Delta_\varepsilon\}
	\ \ge\ \min\{\Delta_0,\gamma_{\mathrm{dec}}c_\Delta\,\varepsilon\}.
	\]
	For $\varepsilon\in(0,1]$, this implies the uniform bound
	\[
	\frac{1}{\underline\Delta_\varepsilon^{\,2}}
	\ \le\
	\max\left\{\frac{1}{\Delta_0^2},\ \frac{1}{\gamma_{\mathrm{dec}}^2c_\Delta^2\,\varepsilon^2}\right\}
	\ \le\
	\left(\frac{1}{\Delta_0^2}+\frac{1}{\gamma_{\mathrm{dec}}^2c_\Delta^2}\right)\varepsilon^{-2}.
	\]
	Combining with the decrease bound yields $N_s\le C_s\,\varepsilon^{-2}$ for a constant $C_s$ independent of $\varepsilon$.
	
	Criticality-loop overhead.
By Lemma~\ref{lem:conv-repair-finite}, $N_k^{\mathrm{rep,base}}\le T_{\mathrm{try}}+2n$ for every iteration.
It remains to bound $\sum_k N_k^{\mathrm{rep,crit}}$.
	Let $N_{\mathrm{crit}}:=\sum_{k=0}^{K-1}L_k$ denote the total number of extra criticality shrinks across all
	$K$ main iterations (where $L_k\ge 0$ is the number of extra shrinks at iteration~$k$).
	Each extra shrink reduces $\log\Delta_k$ by $\log(1/\gamma_{\mathrm{dec}})$, while only very-successful
	iterations can increase $\log\Delta_k$ (by at most $\log\gamma_{\mathrm{inc}}$).
	Since $\Delta_k\in[\underline\Delta_\varepsilon,\Delta_{\max}]$ for all $k<K$, a log-radius potential argument gives
	\[
	N_{\mathrm{crit}}\,\log(1/\gamma_{\mathrm{dec}})
	\;\le\;
	\log(\Delta_{\max}/\underline\Delta_\varepsilon) + N_s\log\gamma_{\mathrm{inc}},
	\]
	so $N_{\mathrm{crit}}\le C_{\mathrm{crit}}\,\varepsilon^{-2}$ for a constant $C_{\mathrm{crit}}$ depending
	only on the algorithm parameters
	(using $\log(1/\varepsilon)\le \varepsilon^{-2}$ for $\varepsilon\in(0,1]$
	and $N_s\le C_s\varepsilon^{-2}$).
By Lemma~\ref{lem:conv-repair-finite}, each extra shrink costs at most $T_{\mathrm{try}}+2n$ evaluations, so
\[
\sum_{k=0}^{K-1} N_k^{\mathrm{rep,crit}}
\;\le\; N_{\mathrm{crit}}(T_{\mathrm{try}}+2n)
\;\le\; C_{\mathrm{crit}}(T_{\mathrm{try}}+2n)\,\varepsilon^{-2}.
\]

	Now let $N_u$ be the number of unsuccessful iterations in $\{0,\ldots,K-1\}$ and let
	$N_{++}:=\#\{k<K:\rho_k\ge\eta_2\}$ count the very-successful iterations (which may increase $\Delta_k$); since every very-successful iteration is successful, $N_{++}\le N_s$.
	Define the log-radius potential $\psi_k:=\log\Delta_k$. At each iteration:
	unsuccessful $\Rightarrow$ $\psi_{k+1}\le \psi_k-\log(1/\gamma_{\mathrm{dec}})$;
	very-successful $\Rightarrow$ $\psi_{k+1}\le \psi_k+\log\gamma_{\mathrm{inc}}$;
	otherwise $\psi_{k+1}\le\psi_k$.
Since $\Delta_k\in[\underline\Delta_\varepsilon,\Delta_{\max}]$ for all
$k<K$, the total decrease in $\psi_k$ is bounded by the available range
plus the total increase, giving
\[
N_u\,\log(1/\gamma_{\mathrm{dec}})
\;\le\;
\log(\Delta_{\max}/\underline\Delta_\varepsilon) + N_{++}\log\gamma_{\mathrm{inc}}
\;\le\;
\log(\Delta_{\max}/\underline\Delta_\varepsilon) + N_s\log\gamma_{\mathrm{inc}}.
\]
Since
\[
\underline\Delta_\varepsilon
\ge
\min\{\Delta_0,\gamma_{\mathrm{dec}}c_\Delta\varepsilon\},
\]
we obtain
\begin{align*}
\log(\Delta_{\max}/\underline\Delta_\varepsilon)
&\le \log(\Delta_{\max}/\Delta_0)
+ \log(1/(\gamma_{\mathrm{dec}}c_\Delta))
+ \log(1/\varepsilon).
\end{align*}
Using $\log(1/\varepsilon)\le\varepsilon^{-2}$ for $\varepsilon\in(0,1]$
and $N_s\le C_s\varepsilon^{-2}$, we obtain
$N_u\le C_u\varepsilon^{-2}$ for a constant $C_u$ depending only on the
algorithm parameters.
	
	Hence the total number of main iterations satisfies $K=N_s+N_u \le C_K\,\varepsilon^{-2}$ for a constant $C_K$.
	Combining the base and criticality components:
	\[
	\sum_{k=0}^{K-1}\bigl(N_k^{\mathrm{trial}}+N_k^{\mathrm{rep,base}}+N_k^{\mathrm{rep,crit}}\bigr)
	\;\le\;
K(1+T_{\mathrm{try}}+2n) + N_{\mathrm{crit}}(T_{\mathrm{try}}+2n)
\;\le\;
C_{\mathrm{eval}}\,\varepsilon^{-2},
\]
where $C_{\mathrm{eval}}:=(1+T_{\mathrm{try}}+2n)\,C_K + (T_{\mathrm{try}}+2n)\,C_{\mathrm{crit}}$.
The fixed initial-design cost is absorbed into $C_{\mathrm{eval}}$, since $\varepsilon\in(0,1]$.
\end{proof}

\section{Geometry Repair Details}\label{app:repair-details}

This appendix provides the implementation details of the geometry
repair mechanism summarized in Section~\ref{subsec:repair-acq}.

\subsection{Incremental repair procedure}\label{app:repair-incr}

Candidate pool alternatives.
In higher dimensions the probability that a uniformly sampled
point satisfies the replace-one test \eqref{eq:repair-feasible}
decreases, so the incremental phase may defer to the fallback
more frequently.
Alternative candidate strategies---such as sampling along
directions of maximal predictive variance, using orthogonal
designs, or perturbing the coordinate frame of
$Y_k^{\mathrm{fb}}$---can increase the hit rate without
affecting the theoretical guarantees.

Repair using previously evaluated points.
Before spending evaluation budget on new candidate points,
we scan the evaluation database $\mathcal{D}_k$ for previously
evaluated points that satisfy the feasibility condition
\eqref{eq:repair-feasible}.
Let
$C_k^{\mathrm{hist}} := \{y \in \mathcal{D}_k : \|y - x_k\|_2 \le
\ctrim\Delta_k\} \setminus Y_k$
be the set of nearby previously evaluated points outside the
interpolation set.
We check \eqref{eq:repair-feasible} for each $y \in C_k^{\mathrm{hist}}$,
ranked by GP posterior variance
$\sigma_k^2(y) := \mathrm{Var}[f(y) \mid \mathcal{D}_k]$
in decreasing order (so that the most informative points are tried first).
Each feasible previously evaluated point is swapped in at zero evaluation cost
(its function value is already in $\mathcal{D}_k$).
If MAP-poisedness is restored during this scan, no further evaluations
are needed.
This phase preserves the worst-case bound in
Lemma~\ref{lem:conv-repair-finite} (which counts only new evaluations)
and reduces the expected repair overhead, particularly in later
iterations when $|\mathcal{D}_k|$ is large.

\begin{remark}[A computational shortcut]\label{rem:repair-leverage}
	Evaluating \eqref{eq:repair-jstar} exactly may be expensive when $m$ is large. A common shortcut is to drop a point based
	on leverage scores with respect to
	\[
	B_k(Y):=\widehat A(Y)\widehat W_k^{-1/2},
	\qquad
	\widehat M_k(Y)=B_k(Y)B_k(Y)^\top.
	\]
	For full row rank, the leverage score of the $i$-th row $b_i^\top$ of $B_k(Y)$ is
	\begin{equation}\label{eq:repair-leverage}
		\ell_i(Y):=b_i^\top\big(\widehat M_k(Y)\big)^{-1}b_i.
	\end{equation}
	Dropping the point with the largest leverage score often improves conditioning in practice. Our guarantees,
	rely on the feasibility condition \eqref{eq:repair-feasible} and the fallback reset below.
\end{remark}

\begin{remark}[Computational cost of the replace-one test]\label{rem:repair-cost}
	A na\"ive implementation of \eqref{eq:repair-jstar} requires
	$m \cdot N_{\mathrm{cand}}$ eigenvalue computations of
	$(m{+}1)\times(m{+}1)$ matrices, each costing $\mathcal{O}(m^3)$,
	for a total of $\mathcal{O}(m^4 N_{\mathrm{cand}})$ per repair attempt.
	In practice, once $\widehat M_k(Y_k)$ and its Cholesky factor are
	available, each replace-one update can be evaluated via a rank-one
	Cholesky update/downdate at cost $\mathcal{O}(m^2)$, reducing the
	total to $\mathcal{O}(m^2 \cdot m \cdot N_{\mathrm{cand}})
	= \mathcal{O}(m^3 N_{\mathrm{cand}})$.
	For $m = 2n \le 40$ and $N_{\mathrm{cand}} = 30$, this cost is
	typically secondary compared to a single expensive function evaluation.
\end{remark}

\subsection{Fallback reconstruction}\label{app:repair-fallback-details}

Attempt-counting convention.
If $C_k^{\mathrm{geo}}=\emptyset$ for the sampled candidate
pool at a given attempt, the attempt is excluded from the count
toward $T_{\mathrm{try}}$ (since no function evaluation occurs);
a fresh candidate pool is generated and the test is repeated.
If $C_k^{\mathrm{geo}}$ remains empty after a single re-sample
(zero evaluations consumed), the algorithm proceeds directly
to the fallback.
This convention ensures that exactly one function evaluation
corresponds to each counted attempt, preserving the per-iteration
evaluation bound in \eqref{eq:repair-eval-bound-reset} and
Lemma~\ref{lem:conv-repair-finite}.

Reconstruction from previously evaluated points scoring.
Before resorting to the coordinate-direction fallback,
we attempt to assemble a MAP-poised set purely from $\mathcal{D}_k$.
Let $\mathcal{H}_k := \{y \in \mathcal{D}_k : \|y - x_k\|_2 \le \ctrim\Delta_k\}$
be the set of nearby previously evaluated points.
If $|\mathcal{H}_k| \ge m + 1$, we greedily select $m$ non-center points
scored by a surrogate-variance criterion,
\begin{equation}\label{eq:repair-hfb-score}
\mathrm{score}(y) := \sigma_k^2(y) \,\exp\!\Bigl(-\tfrac{\|y - x_k\|_2^2}{2\Delta_k^2}\Bigr),
\end{equation}
where $\sigma_k^2(y) = \mathrm{Var}[f(y) \mid \mathcal{D}_k]$
is the GP posterior variance.
If the resulting set satisfies $\lambda_{\min}(\widehat M_k) \ge \mu_M$,
the fallback is complete with zero new evaluations.
Otherwise, the algorithm proceeds to the coordinate-direction reset.
This surrogate-guided search over previously evaluated points reuses the surrogate information
when such information is available; the evaluation database is reused
throughout. The worst-case guarantee is unaffected because the coordinate
fallback is always available, while the expected repair cost can be lower
in later iterations when $|\mathcal{D}_k|$ is large.

\clearpage
\renewcommand{\thesection}{S\arabic{section}}
\renewcommand{\theHsection}{supp.\arabic{section}}
\setcounter{section}{0}
\setcounter{table}{0}
\renewcommand{\thetable}{S\arabic{table}}
\renewcommand{\theHtable}{supp.table.\arabic{table}}
\setcounter{figure}{0}
\renewcommand{\thefigure}{S\arabic{figure}}
\renewcommand{\theHfigure}{supp.figure.\arabic{figure}}

\section*{Supplementary Material}
\addcontentsline{toc}{section}{Supplementary Material}
This supplementary block records the benchmark definitions and the full
solver parameter table used in the implementation.

\section{Benchmark Problem List}\label{app:supp-problems}

Table~\ref{tab:supp-problems} lists all 17 benchmark problems used in the
numerical experiments, tested at dimensions $n\in\{5,10,20,30,50\}$
($17\times 5 = 85$ problem--dimension pairs).
Sources include the Mor{\'e}--Garbow--Hillstrom collection~\cite{more2009benchmarking},
the CUTEst test environment~\cite{gould2015cutest}, and custom engineering-inspired problems.
For problems marked ``$f^\star=$ varies'', the optimum is
dimension-dependent and is determined analytically by the benchmark
definition used in our experiments.

\small
\begin{longtable}{llcrl}
\caption{Benchmark problems (17 functions).}\label{tab:supp-problems}\\
\toprule
\textbf{Problem} & \textbf{Source} & \textbf{Dim} & \textbf{$f^\star$} & \textbf{Type} \\
\midrule
\endfirsthead
\toprule
\textbf{Problem} & \textbf{Source} & \textbf{Dim} & \textbf{$f^\star$} & \textbf{Type} \\
\midrule
\endhead
\midrule
\endfoot
\bottomrule
\endlastfoot
Rosenbrock           & Classical & 5--50 & 0 & Valley \\
DixonPrice           & Classical & 5--50 & 0 & Valley \\
Trid                 & MGH       & 5--50 & varies & Quadratic \\
Edensch              & CUTEst    & 5--50 & varies & Mixed \\
Cube                 & CUTEst    & 5--50 & 0 & Cubic \\
Genrose              & CUTEst    & 5--50 & 0 & Valley \\
ScaledRosen          & CUTEst    & 5--50 & 0 & Valley \\
Engval1              & CUTEst    & 5--50 & varies & Mixed \\
Fletchcr             & CUTEst    & 5--50 & 0 & Cyclic \\
Nondquar             & Toint     & 5--50 & 0 & Quartic \\
Quartc               & CUTEst    & 5--50 & 0 & Quartic \\
Himmelbh             & CUTEst    & 5--50 & 0 & Himmelblau \\
Bdqrtic              & CUTEst    & 5--50 & 0 & Quartic \\
Cragglvy             & CUTEst    & 5--50 & 0 & Mixed \\
ReactorCascade       & Custom    & 5--50 & varies & Engineering \\
PipelineDesign       & Custom    & 5--50 & varies & Engineering \\
SeqProcess           & Custom    & 5--50 & varies & Engineering \\
\end{longtable}
\normalsize

\section{Solver Parameters}\label{app:supp-params}

Shared trust-region parameters are listed in Table~\ref{tab:exp-params}
(Section~\ref{subsec:exp-setup}).
Table~\ref{tab:supp-params-bup} below provides the full list of
BUP-NEWUOA-specific parameters, complementing the summary in
Table~\ref{tab:exp-params}.

\clearpage
\begin{table}[p]
\centering
\caption{BUP-NEWUOA-specific parameters.
Parameters not listed use the NEWUOA defaults.}\label{tab:supp-params-bup}
\scriptsize
\setlength{\tabcolsep}{2pt}
\begin{tabularx}{\textwidth}{@{}l l X@{}}
\toprule
\textbf{Parameter} & \textbf{Default} & \textbf{Description} \\
\midrule
\multicolumn{3}{l}{\textbf{Accepted-Model Prior / Local Statistics}} \\
\texttt{prior\_mode}         & merged & Accepted-model prior with local WLS scale calibration \\
\texttt{tau}                 & 0.3   & Prior strength coefficient \\
\texttt{hess\_prior\_scale}  & 0.3   & Relative scale for Hessian-block precision \\
\texttt{hess\_decay} ($\lambda_{\mathrm{decay}}$) & 1.5 & Structured Hessian decay rate \\
\texttt{wls\_align}          & B     & Accepted-model center-shift option \\
\texttt{w\_min}              & 0.1   & Minimum precision after clipping \\
\texttt{w\_max}              & 100.0 & Maximum precision after clipping \\
\midrule
\multicolumn{3}{l}{\textbf{Data Pool / Optional GP Prior}} \\
\texttt{pool\_radius\_ratio} & 5.0   & Local data-pool radius multiplier \\
\texttt{max\_pool\_gp}       & $\min(100, 10(n{+}1))$ & Maximum local pool size for GP/WLS fitting \\
\texttt{gp\_length\_scale}   & 2.0   & Length scale for optional GP prior \\
\texttt{gp\_refit\_interval} & 3     & Refit interval for optional GP prior \\
\texttt{map\_threshold}      & 3     & Iteration threshold for MAP trigger \\
\midrule
\multicolumn{3}{l}{\textbf{MAP Gating}} \\
\texttt{gate\_cos\_sim}      & 0.3   & Minimum cosine similarity for accepting MAP model \\
\texttt{gate\_norm\_ratio}   & 10.0  & Maximum gradient-norm ratio for accepting MAP model \\
\texttt{gate\_event\_norm}   & 5.0   & Maximum norm ratio for event-triggered MAP model \\
\texttt{gate\_eig\_ratio}    & 10.0  & Maximum eigenvalue ratio for event-triggered MAP model \\
\midrule
\multicolumn{3}{l}{\textbf{Trust Region and Restart}} \\
\texttt{restart}             & True  & Enable implementation restart heuristic \\
\texttt{max\_restarts}       & 2     & Maximum number of implementation restarts \\
\texttt{rho\_restart\_factor}& 0.1   & Initial-radius factor after restart \\
\texttt{rho\_tau\_decay}     & True  & Radius-based prior-strength decay switch \\
\texttt{rho\_tau\_schedule}  & log   & Prior-strength decay schedule \\
\midrule
\multicolumn{3}{l}{\textbf{Ablation Switches (All True by Default)}} \\
\texttt{use\_map\_model}     & True  & Enable MAP model construction \\
\texttt{use\_gp\_geometry}   & True  & Enable distance-enhanced geometry step \\
\texttt{use\_data\_pool}     & True  & Enable evaluated data pool \\
\texttt{use\_periodic\_corr} & False & Enable periodic MAP correction \\
\midrule
\multicolumn{3}{l}{\textbf{Geometry Repair}} \\
$\mu_M$                      & $0.1 \cdot \mu_0$ & MAP-poisedness threshold \\
\texttt{T\_try}              & 3     & Repair attempts per geometry round \\
\texttt{N\_cand}             & 30    & Candidate points per repair round \\
\texttt{shrink\_safeguard}   & 0.5   & Minimum ratio for adaptive shrinkage safeguard \\
\bottomrule
\end{tabularx}
\end{table}
\clearpage

\end{document}